\documentclass[
	final,
	paper=a4,
	pagesize=auto,
	fontsize=10pt,
	version=last,
]{scrartcl}

\usepackage[utf8]{inputenc}
\usepackage[
	english,
]{babel}

\KOMAoptions{%
   DIV=10,
}%
\KOMAoptions{%
   numbers=noenddot
}%
\KOMAoptions{
   twoside=semi,  
   twocolumn=false, 
   headinclude=false,%
   footinclude=false,%
   mpinclude=false,%
   headlines=2.1,%
   headsepline=true,%
   footsepline=false,%
   cleardoublepage=empty 
}%
\KOMAoptions{%
   parskip=relative, 
}%
\KOMAoptions{%
}%
\KOMAoptions{%
}%
\KOMAoptions{%
   bibliography=totoc 
}%
\KOMAoptions{%
}%
\KOMAoptions{%
}%
\KOMAoptions{%
}%

\makeatletter
\newcommand{\LoadPackagesNow}{}
\newcommand{\LoadPackageLater}[2][]{%
   \g@addto@macro{\LoadPackagesNow}{%
      \usepackage[#1]{#2}%
   }%
}
\makeatother

\usepackage{xspace}
\usepackage{xargs}
\usepackage{ifthen}
\usepackage{etoolbox}

\usepackage{calc}
\usepackage[
	margin=3cm,
	]{geometry}
\usepackage{framed}
\usepackage{mdframed}
\usepackage{pbox}
\usepackage{enumitem}
\usepackage{array}
\usepackage{tabularx}
\newcolumntype{Z}{>{\raggedright\let\newline\\\arraybackslash\hspace{0pt}}p}

\usepackage{multirow}
\usepackage{longtable}
\usepackage{hhline}
\usepackage{afterpage}
\usepackage{pdflscape}
\usepackage{rotating}
\usepackage{booktabs}
\usepackage{setspace}
\usepackage[
   bottom,      
   stable,      
   ragged,      
   multiple,    
   flushmargin,
   hang,
]{footmisc}

\usepackage{mathpazo}

\usepackage[%
   headsepline,
   automark,
   komastyle,
]{scrpage2}

\clearscrheadings
\clearscrplain
\lehead{\pagemark}
\rohead{\pagemark}
\rehead{\headmark}
\lohead{\shortauthor}
\automark[section]{section} 

\usepackage{relsize}
\usepackage[normalem]{ulem}      
\usepackage{soul}		           
\usepackage{url}
\usepackage{lipsum}

\usepackage[table]{xcolor}
\usepackage{tikz}
\usetikzlibrary{arrows,shapes,calc,decorations.pathreplacing,intersections,quotes,angles,positioning,fit,petri,through}

\usepackage[%
]{graphicx}

\usepackage{epstopdf}
\usepackage{wrapfig}
\usepackage{caption}
\usepackage{subcaption}
\usepackage[
   tbtags,    
   sumlimits,  
   nointlimits, 
   namelimits, 
   reqno,     
]{amsmath} %

\usepackage{amsfonts}
\usepackage{mathrsfs} 
\usepackage{dsfont}
\usepackage{amssymb}
\usepackage{units}
\LoadPackageLater{amsthm}
\LoadPackageLater{thmtools}
\usepackage[fixamsmath,disallowspaces]{mathtools}
\mathtoolsset{showonlyrefs}
\mathtoolsset{centercolon=true}
\usepackage{bm} 
\makeatletter
\g@addto@macro\bfseries{\boldmath}
\makeatother
\allowdisplaybreaks[4]
\numberwithin{equation}{section}

\usepackage[toc]{appendix}

\usepackage{csquotes}
\usepackage[
	style=alphabetic,
	giveninits=true,
	maxbibnames=99,
	backend=biber,
	maxalphanames=5,
]{biblatex}

\renewbibmacro{in:}{}
\DeclareFieldFormat{pages}{#1}

\usepackage[textsize=tiny,english,colorinlistoftodos,
	disable,
	]{todonotes}

\definecolor{pdfurlcolor}{rgb}{0,0,0.6}
\definecolor{pdffilecolor}{rgb}{0.7,0,0}
\definecolor{pdflinkcolor}{rgb}{0,0,0.6}
\definecolor{pdfcitecolor}{rgb}{0,0,0.6}
\usepackage[
  colorlinks=true,         
  urlcolor=pdfurlcolor,    
  filecolor=pdffilecolor,  
  linkcolor=pdflinkcolor,  
  citecolor=pdfcitecolor,  %
  raiselinks=true,			 
  breaklinks,              
  verbose,
  hyperindex=true,         
  linktocpage=true,        
  hyperfootnotes=false,     
  bookmarks=true,          
  bookmarksopenlevel=1,    
  bookmarksopen=true,      
  bookmarksnumbered=true,  
  bookmarkstype=toc,       
  plainpages=false,        
  pageanchor=true,         
  pdfdisplaydoctitle=true, 
  pdfstartview=FitH,       
  pdfpagemode=UseOutlines, 
  pdfpagelabels=true,           
  pdfpagelayout=OneColumn, 
]{hyperref}

\LoadPackagesNow

\newcommand{\ifargdef}[3][{}]{\ifthenelse{\equal{#2}{}}{#1}{#3}}

\newlength{\hangind}
\newcommand{\myhangindent}[1]{\settowidth{\hangind}{\widthof{#1}}\hangindent=\the\hangind}

\setkomafont{section}{\Large\rmfamily\bfseries}
\setkomafont{subsection}{\large\rmfamily\bfseries}
\setkomafont{subsubsection}{\rmfamily\bfseries}
\setkomafont{paragraph}{\rmfamily\bfseries}
\setkomafont{pageheadfoot}{\smaller\scshape\rmfamily}

\hypersetup{
	pdfauthor={Martin Genzel, Alexander Stollenwerk},
	pdftitle={Robust \texorpdfstring{$1$}{1}-Bit Compressed Sensing via Hinge Loss Minimization},
	pdfsubject={Article},
	pdfcreator={PDF-LaTeX},
}

\newenvironment{highlight}{\begin{quote}\itshape}{\end{quote}}

\newenvironment{properties}[2][2em]
{\begin{enumerate}[label={\textsc{(#2\arabic*)}},leftmargin=#1]}
{\end{enumerate}}

\newenvironment{listing}
{\begin{itemize}[itemindent=0em,leftmargin=1.2em]}
{\end{itemize}}

\newenvironment{rmklist}
{\begin{enumerate}[label={(\arabic*)},itemindent=2em,leftmargin=0em]}
{\end{enumerate}}

\newenvironment{thmproperties}
{\begin{enumerate}[label={(\roman*)}]}
{\end{enumerate}}

\renewcommand{\qedsymbol}{$_\blacksquare$}

\providecommand{\qedhere}{\hfill\qedsymbol}

\newtheoremstyle{claim}
	{\topsep}{\topsep}%
	{\itshape}
	{}
	{}
	{}
	{.5em}
	{{\bfseries\boldmath\thmname{#1} \thmnumber{#2}} \thmnote{(#3)}}

\newtheoremstyle{definition}
	{\topsep}{\topsep}%
	{}
	{}
	{}
	{}
	{.5em}
	{\textbf{\thmname{#1} \thmnumber{#2}} \thmnote{(#3)}}
	
\newtheoremstyle{algorithm}
	{\topsep}{\topsep}%
	{}
	{}
	{\bfseries\boldmath}
	{}
	{.5em}
	{\thmname{#1} \thmnumber{#2} \thmnote{(#3)}}

\declaretheorem[style=claim,numberwithin=section]{theorem}
\declaretheorem[style=claim,sibling=theorem]{lemma}
\declaretheorem[style=claim,sibling=theorem]{corollary}
\declaretheorem[style=claim,sibling=theorem]{proposition}
\declaretheorem[style=claim,sibling=theorem]{fact}
\declaretheorem[style=definition,sibling=theorem]{definition}

\declaretheorem[style=definition,sibling=theorem]{assumption}

\declaretheorem[style=definition,sibling=theorem,qed=$\Diamond$]{remark}
\declaretheorem[style=definition,sibling=theorem,qed=$\Diamond$]{example}

\newcommand{\opleft}[1]{\mathopen{}\left#1}
\newcommand{\opright}[1]{\right#1\mathclose{}}
\newcommandx{\braces}[4]{%
\ifstrequal{#3}{normal}{#1#4#2}{%
\ifstrequal{#3}{auto}{\left#1#4\right#2}{%
\ifstrequal{#3}{opauto}{\opleft#1#4\opright#2}{%
#3#1#4#3#2}}}%
}
\newcommandx{\opannot}[3][3=\downarrow]{\stackrel{\mathclap{\substack{#1 \\ #3 \vspace{2pt}}}}{#2}}
\newcommandx{\lineannot}[3][3=\rightarrow]{\mathllap{\boxed{\text{\textsmaller{#1}}} #3} #2}
\newcommandx{\multilineannot}[4][4=\rightarrow]{\mathllap{\boxed{\parbox{#1}{\RaggedRight\textsmaller{#2}}} #4} #3}
 
\newcommand{\N}{\mathbb{N}} 
\newcommand{\R}{\mathbb{R}} 
\newcommand{\eps}{\varepsilon} 

\renewcommand{\implies}{\Rightarrow} 
\newcommand{\suchthat}[1][normal]{\ifstrequal{#1}{normal}{\mid}{#1|}} 

\newcommand{\setcompl}[1]{#1^c} 
\newcommand{\cardinality}[1]{\abs{#1}} 
\newcommand{\union}{\cup} 
\newcommand{\intersec}{\cap} 
\newcommand{\boundary}[1]{\partial#1} 

\newcommand{\diam}{\operatorname{diam}}

\newcommandx{\intvcl}[3][1=normal]{\braces{[}{]}{#1}{#2, #3}} 
\newcommandx{\intvop}[3][1=normal]{\braces{(}{)}{#1}{#2, #3}} 
\newcommandx{\intvclop}[3][1=normal]{\braces{[}{)}{#1}{#2, #3}} 
\newcommandx{\intvopcl}[3][1=normal]{\braces{(}{]}{#1}{#2, #3}} 
\DeclareMathOperator*{\argmin}{argmin} 
\DeclareMathOperator{\sign}{sign}
\newcommandx{\abs}[2][1=normal]{\braces{\lvert}{\rvert}{#1}{#2}} 
\newcommandx{\ceil}[2][1=normal]{\braces{\lceil}{\rceil}{#1}{#2}} 
\newcommandx{\floor}[2][1=normal]{\braces{\lfloor}{\rfloor}{#1}{#2}} 
\newcommandx{\round}[2][1=normal]{\braces{[}{]}{#1}{#2}} 
\newcommandx{\der}[1]{D^{#1}} 
\newcommandx{\gradient}{\nabla} 
\newcommandx{\partder}[4][1={},4={}]{\frac{\partial^{#4} #2}{\partial #3^{#4}}\ifargdef{#1}{\Big|_{#1}}} 
\newcommandx{\integ}[4][1={},2={}]{\int_{#1}^{#2} #3 \, #4} 
\newcommandx{\asympffaster}[2][1=normal]{o\braces{(}{)}{#1}{#2}} 
\newcommandx{\asympfaster}[2][1=normal]{O\braces{(}{)}{#1}{#2}} 
\newcommandx{\asympeq}[2][1=normal]{\Theta\braces{(}{)}{#1}{#2}} 
\newcommandx{\asympsslower}[2][1=normal]{\omega\braces{(}{)}{#1}{#2}} 
\newcommandx{\asympslower}[2][1=normal]{\Omega\braces{(}{)}{#1}{#2}} 

\newcommandx{\norm}[2][1=normal]{\braces{\|}{\|}{#1}{#2}} 
\renewcommandx{\sp}[3][1=normal]{\braces{\langle}{\rangle}{#1}{#2, #3}} 
\newcommandx{\End}[2][2={}]{\mathcal{L}\opleft( #1 \ifargdef{#2}{, #2} \opright)} 
\newcommand{\orthcompl}[1]{{#1}^\perp} 
\DeclareMathOperator{\spann}{\operatorname{span}} 
\renewcommand{\vec}[1]{\boldsymbol{#1}} 

\newcommandx{\measure}[2][1=normal]{\operatorname{vol}\braces{(}{)}{#1}{#2}} 
\newcommand{\indset}[1]{\chi_{#1}} 
\DeclareMathOperator{\supp}{supp} 
\newcommandx{\Leb}[3][1={},3=normal]{L^{#2}\ifargdef{#1}{\braces{(}{)}{#3}{#1}}{}} 
\newcommandx{\Lebnorm}[4][1=normal,3={2},4={}]{\norm[#1]{#2}_{#3}} 
\renewcommandx{\l}[3][1={},3=normal]{\ell^{#2}\ifargdef{#1}{\braces{(}{)}{#3}{#1}}} 
\newcommandx{\lnorm}[4][1=normal,3={2},4={}]{\norm[#1]{#2}_{#3}} 
\newcommandx{\Smooth}[4][1={},3={},4=normal]{C_{#3}^{#2}\ifargdef{#1}{\braces{(}{)}{#4}{#1}}} 
\newcommandx{\Schwartz}[2][1={},2=normal]{\mathscr{S}\ifargdef{#1}{\braces{(}{)}{#2}{#1}}} 
\newcommandx{\Schwartzpoly}[2][1=normal]{\braces{\langle}{\rangle}{#1}{\abs[#1]{#2}} } 
\newcommandx{\Tempdistr}[2][1={},2=normal]{\mathscr{S}'\ifargdef{#1}{\braces{(}{)}{#2}{#1}}} 
\newcommandx{\distrinp}[3][1=normal]{\braces{\langle}{\rangle}{#1}{#2, #3}} 
\newcommandx{\ft}[3][1=default,2=auto]{
\ifstrequal{#1}{default}{\widehat{#3}}{
\ifstrequal{#1}{long}{{\braces{(}{)}{#2}{#3}}^{\wedge}}{}}} 
\newcommandx{\ift}[3][1=default,2=auto]{
\ifstrequal{#1}{default}{\check{#3}}{
\ifstrequal{#1}{long}{{\braces{(}{)}{#2}{#3}}^{\vee}}{}}} 

\newcommand{\erf}{\operatorname{erf}} 

\renewcommand{\vec}[1]{\bm{#1}}

\renewcommand{\a}{\vec{a}}

\newcommand{\x}{\vec{x}}
\newcommand{\grtr}{\vec{x}_0}

\newcommand{\grtrmu}{\scalfac\vec{x}_0}
\newcommand{\solu}{\vec{\hat{x}}}
\newcommand{\sset}{K}
\newcommand{\ssetmu}{\scalfac K}
\newcommand{\hypospace}{\mathcal{H}}
\newcommand{\h}{\vec{h}}

\newcommand{\proj}{\vec{P}}

\newcommand{\loss}{\mathcal{L}}
\newcommand{\losssq}{\mathcal{L}^{\text{sq}}}
\newcommand{\losshng}{\mathcal{L}^{\text{hng}}}
\newcommand{\lossexp}{\mathcal{R}}
\newcommand{\excessloss}{\mathcal{E}}

\newcommand{\lossemp}[1][{}]{\bar{\mathcal{R}}_{#1}}

\newcommand{\multiplterm}[2]{\mathcal{M}(#1,#2)}
\newcommand{\multipl}{\mathcal{M}}
\newcommand{\quadrterm}[2]{\mathcal{Q}(#1,#2)}
\newcommand{\quadr}{\mathcal{Q}}

\newcommand{\fobs}{f}

\newcommand{\scalfac}{\mu}

\newcommand{\vnull}{\vec{0}}
\newcommand{\I}[1]{\vec{I}_{#1}}

\newcommandx{\prob}[2][1={},2=normal]{\mathbb{P}\ifargdef{#1}{\braces{[}{]}{#2}{#1}}}

\newcommandx{\mean}[2][1={},2=normal]{\mathbb{E}\ifargdef{#1}{\braces{[}{]}{#2}{#1}}}
\newcommandx{\var}[2][1={},2=normal]{\mathbb{V}\ifargdef{#1}{\braces{[}{]}{#2}{#1}}}
\newcommand{\normsubg}[1]{\norm{#1}_{\psi_2}}
\newcommand{\indprob}[1]{\mathds{1}_{#1}} 
\newcommand{\cov}{\operatorname{Cov}}
\newcommand{\distributed}{\sim}

\newcommand{\Normdistr}[2]{\mathcal{N}(#1, #2)}
\newcommand{\gaussian}{\vec{g}}
\newcommand{\gaussianuniv}{g}

\newcommand{\empproc}{\mathsf{F}}

\newcommandx{\pospart}[2][1=auto]{\braces{[}{]_+}{#1}{#2}}

\newcommandx{\ball}[2][1={},2={}]{B_{#1}^{#2}}
\DeclareMathOperator{\convhull}{conv}
\renewcommand{\S}{\mathbb{S}}

\newcommand{\meanwidth}[2][{}]{w_{#1}(#2)}
\newcommand{\effdim}[2][{}]{d_{#1}(#2)}
\newcommand{\cone}[2]{\mathcal{C}(#1,#2)}
\newcommand{\conic}{0}

\newcommand{\pcloud}{\mathcal{C}}
\newcommand{\cyl}{\operatorname{Cyl}}
\newcommand{\rad}{\operatorname{rad}}

\graphicspath{{images/}}

\begin{document}


\pagestyle{scrheadings}
\renewcommand*{\thefootnote}{\fnsymbol{footnote}}

\begin{center}
	\huge\bfseries 
	Robust $1$-Bit Compressed Sensing \\ via Hinge Loss Minimization
\end{center}

\vspace{1\baselineskip}
\begin{addmargin}[2em]{2em}
\begin{center}
	\noindent{\normalsize\bfseries{Martin Genzel\footnote[1]{Technische Universit\"at Berlin, Department of Mathematics, 10623 Berlin, Germany} \qquad Alexander Stollenwerk\footnote[2]{RWTH Aachen, Department of Mathematics, 52062 Aachen, Germany\label{ftn:rwth}}}}
	
%
\end{center}

\vspace{1.5\baselineskip}
{\smaller
\noindent\textbf{Abstract.} 
This work theoretically studies the problem of estimating a structured high-dimensional signal $\grtr \in \R^n$ from noisy $1$-bit Gaussian measurements.
Our recovery approach is based on a simple convex program which uses the hinge loss function as data fidelity term. While such a risk minimization strategy is very natural to learn binary output models, such as in classification, its capacity to estimate a specific signal vector is largely unexplored. A major difficulty is that the hinge loss is just piecewise linear, so that its ``curvature energy'' is concentrated in a single point. This is substantially different from other popular loss functions considered in signal estimation, e.g., the square or logistic loss, which are at least locally strongly convex. It is therefore somewhat unexpected that we can still prove very similar types of recovery guarantees for the hinge loss estimator, even in the presence of strong noise. More specifically, our non-asymptotic error bounds show that stable and robust reconstruction of $\grtr$ can be achieved with the optimal oversampling rate $\asympfaster{m^{-1/2}}$ in terms of the number of measurements $m$. Moreover, we permit a wide class of structural assumptions on the ground truth signal, in the sense that $\grtr$ can belong to an arbitrary bounded convex set $\sset \subset \R^n$. The proofs of our main results rely on some recent advances in statistical learning theory due to Mendelson. In particular, we invoke an adapted version of Mendelson's small ball method that allows us to establish a quadratic lower bound on the error of the first order Taylor approximation of the empirical hinge loss function.

\vspace{.5\baselineskip}
\noindent\textbf{Key words.}
$1$-bit compressed sensing, structured empirical risk minimization, hinge loss, Gaussian width, Mendelson's small ball method

}
\end{addmargin}
\newcommand{\shortauthor}{Genzel and Stollenwerk: Robust $1$-Bit Compressed Sensing via Hinge Loss Minimization
}

\renewcommand*{\thefootnote}{\arabic{footnote}}
\setcounter{footnote}{0}

\thispagestyle{plain}


\section{Introduction}
\label{sec:intro}

This paper considers the problem of estimating an unknown \emph{signal vector} $\grtr \in \R^n$ from $1$-bit observations of the form
\begin{equation}\label{eq:onebitmodel}
	y_i = f_i(\sp{\a_i}{\grtr}) \in \{-1, +1\}, \quad i=1,\dots,m,
\end{equation}
where $\a_1, \dots, \a_m \in \R^n$ is a collection of known \emph{measurement vectors} and $f_i \colon \R \to \{-1,+1\}$, $i = 1, \dots, m$, are binary-valued \emph{output functions}. The number of samples $m$ is typically much smaller than the ambient dimension $n$, so that the equation system of \eqref{eq:onebitmodel} is highly underdetermined.
Such types of recovery tasks have recently caught increasing attention in various research areas, most importantly in the field of \emph{compressed sensing}, which has become a state-of-the-art approach in signal processing during the last decade. It builds upon a novel paradigm according to which many reconstruction problems can be efficiently solved by linear or convex programming, if the low-dimensional structure of the ground truth signal is explicitly taken into account; see \cite{foucart2013cs} for a comprehensive introduction.

In contrast to traditional compressed sensing, the setup of \eqref{eq:onebitmodel} does also involve a \emph{non-linear} component.
More specifically, each output function $f_i$ plays the role of a \emph{quantizer} that distorts the linear observation rule $\a_i \mapsto \sp{\a_i}{\grtr}$.  
Such a quantization step is of particular interest to real-world sensing schemes in which only a finite number of bits can be (digitally) processed during transmission. In fact, the model of \eqref{eq:onebitmodel} assumes the most extreme case of quantization where only $1$-bit information is available per measurement --- that is why we speak of \emph{$1$-bit compressed sensing} in this context \cite{boufounos2008onebit,jacques2013onebit,plan2013onebit}.
Let us emphasize that the quantizers $f_i$ can be completely deterministic, e.g., $f_i = \sign$,\footnote{Hereafter, we make the convention that $\sign(v) = +1$ for $v \geq 0$ and $\sign(v) = -1$ for $v < 0$.} but they could be also contaminated by noise in the form of random bit flips. In the latter case, the output functions $f_i$ are assumed to be independent copies of a ``prototypical'' random function $f \colon \R \to \{-1,+1\}$ that describes the underlying noise model; see Assumption~\ref{model:measurements} for more details.

A large class of signal estimation methods can be formulated as an optimization problem of the form
\begin{equation}\label{eq:generalestimator}\tag{$P_{\loss,\sset}$}
	\min_{\x \in \R^n} \tfrac{1}{m} \sum_{i = 1}^m \loss(\sp{\a_i}{\x},y_i) \quad \text{subject to \quad $\x \in \sset$,}
\end{equation}
where $\loss\colon \R \times \R \to \R$ is a convex \emph{loss function} and $\sset \subset \R^n$ defines a \emph{convex} constraint set, usually called the \emph{signal set}.
The purpose of $\loss$ is to assess how well the candidate model $\a_i \mapsto \sp{\a_i}{\x}$ matches with the true outputs $y_i$.
With regard to our initial recovery task, we may therefore hope that a minimizer $\solu \in \R^n$ of \eqref{eq:generalestimator} provides an accurate approximation of the signal vector $\grtr$. 
Furthermore, the signal set $\sset$ encodes certain structural hypotheses about $\grtr$. 
The prototypical example considered in compressed sensing is \emph{sparsity}, meaning that only very few entries of $\grtr$ are non-zero (cf. Subsection~\ref{subsec:intro:notation}\ref{subsec:intro:notation:lp}).
In this situation, it is very common to use an $\l{1}$-constraint, i.e., $\sset = R \ball[1][n]$ for some $R > 0$, which serves as a \emph{convex relaxation} of a sparse signal model \cite{Candes2005,candes:stablesignalrecovery,candes2006recovery,Donoho2006a}.

In this work, we will focus on a special instance of \eqref{eq:generalestimator} that is based on the so-called \emph{hinge loss} given by $\losshng(v) \coloneqq \pospart{1 - v} \coloneqq \max\{0, 1 - v\}$ for $v \in \R$.
Using $\loss(v,v') \coloneqq \losshng(v \cdot v')$ as loss function, the program of \eqref{eq:generalestimator} now reads as follows:
\begin{equation}\label{eq:hingeestimator}\tag{$P_{\losshng,\sset}$}
	\min_{\x \in \R^n} \tfrac{1}{m} \sum_{i = 1}^m \max\{0, 1 - y_i \sp{\a_i}{\x}\} \quad \text{subject to \quad $\x \in \sset$.}
\end{equation}
This estimator is specifically tailored to deal with binary observations:
Intuitively, by minimizing the objective functional of \eqref{eq:hingeestimator}, one tries to select $\x \in \sset$ in such a way that the sign of $\sp{\a_i}{\x}$ equals $y_i \in \{-1,+1\}$ for most of the samples $i \in \{1, \dots, m\}$.
With other words, a solution $\solu$ to \eqref{eq:hingeestimator} yields a predictor $\a_i \mapsto \sign(\sp{\a_i}{\solu})$ of the true outputs $y_i$.
While this simple heuristic explains the success of hinge loss minimization in many classification tasks, it also indicates a certain capability to tackle the $1$-bit compressed sensing problem stated by \eqref{eq:onebitmodel}.
However, let us point out that the latter challenge does not just ask for finding \emph{any} good predictor but actually aims at retrieving the ground truth signal $\grtr$. Compared to reliable prediction, successful signal estimation usually relies on relatively strong model assumptions, and in fact, the performance of \eqref{eq:hingeestimator} is only poorly understood on this matter.
The key concern of this paper is therefore to establish theoretical recovery guarantees for \eqref{eq:hingeestimator} under the hypothesis of \eqref{eq:onebitmodel} with \emph{Gaussian} measurement vectors. In particular, we intend to address the following issues:
\begin{highlight}
	How many sample pairs $\{(\a_i, y_i)\}_{i = 1, \dots, m} \subset \R^n \times \{-1,+1\}$ are required to accurately estimate the ground truth signal $\grtr$ via hinge loss minimization \eqref{eq:hingeestimator}?
	How is this related to the structural assumptions on $\grtr$ and the choice of the signal set $\sset$?
\end{highlight}

\subsection{Main Contributions and Overview}
\label{subsec:intro:overview}

In Section~\ref{sec:results}, we will make the above model setting more precise and provide several definitions of complexity measures for signal sets, in particular the concepts of (local) Gaussian width and effective dimension.
Our first main result (Theorem~\ref{thm:euclideanball}) is then presented in Subsection~\ref{subsec:results:unitball}, considering signal sets in the Euclidean unit ball, i.e., $\sset \subset \ball[2][n]$.
In contrast, our second main results in Subsection~\ref{subsec:results:generalset} (Theorem~\ref{thm:results:generalset:global} and Theorem~\ref{thm:results:generalset:local}) drop this condition and allow for arbitrary bounded convex signal sets, but they are merely restricted to the perfect $1$-bit case with $f_i = \sign$ for all $i = 1, \dots, m$.
To give a first glance at these theoretical findings, let us state the following informal recovery guarantee.
\begin{theorem}[Informal]
	Let the measurement rule \eqref{eq:onebitmodel} hold true with i.i.d.\ standard Gaussian measurement vectors $\a_1, \dots, \a_m \distributed \Normdistr{\vnull}{\I{n}}$. If $\lnorm{\grtr} = 1$ and $\grtr \in \sset$, then with high probability, any minimizer $\solu$ of \eqref{eq:hingeestimator} satisfies
	\begin{equation}\label{eq:intro:errbound} 
		\lnorm[auto]{\grtr - \frac{\solu}{\lnorm{\solu}}} \leq C(f_i) \cdot \Big(\frac{\delta(\sset,\grtr)}{m}\Big)^{1/2},
	\end{equation}
	where $C(f_i) > 0$ only depends on the quantizers $f_1, \dots, f_m$, and $\delta(\sset,\grtr)$ is a measure of complexity for $\sset$ (with respect to $\grtr$) that may change in model situations of Subsection~\ref{subsec:results:unitball} and Subsection~\ref{subsec:results:generalset}.
\end{theorem}
The non-asymptotic error bound of \eqref{eq:intro:errbound} shows that $1$-bit compressed sensing via hinge loss minimization is feasible for a large class of measurement schemes and structural hypotheses.
In particular, $C(f_i)$ can be regarded as a model-dependent parameter that is well-behaved under very mild correlation conditions on the (noisy) quantizers.
Our guarantees significantly improve a recent result from Kolleck and Vyb\'{\i}ral \cite{kolleck2015l1svm}, whose analysis of the hinge loss estimator is just limited to $\l{1}$-constraints and a quite restrictive noise pattern. Apart from that, they do only achieve an oversampling rate of $\asympfaster{m^{-1/4}}$, which is clearly worse than the optimal rate of $\asympfaster{m^{-1/2}}$ promoted by \eqref{eq:intro:errbound}.

All proofs are postponed to Section~\ref{sec:proofs}. Our key arguments are based on tools from \emph{statistical learning theory}, more specifically, on some recent uniform lower and upper bounds for empirical stochastic processes established by Mendelson \cite{mendelson2014learning,mendelson2016upper}.
For this reason, we will also frequently use the common terminology of statistical learning throughout this paper --- for example, the estimators of \eqref{eq:generalestimator} and \eqref{eq:hingeestimator} are typically referred to as \emph{empirical risk minimizers}. 
While the proof strategy of Section~\ref{sec:proofs} loosely follows the learning framework of \cite{mendelson2014learninggeneral}, we wish to emphasize that our findings are by far not ``off the shelf.''
Indeed, to the best of our belief, the hinge loss function $\losshng$ does not meet the prerequisites of known signal estimation results. This in turn requires an extension of many arguments, especially in the case of general signal sets where the standard notion of local Gaussian width needs to be carefully adapted.
A major technical difficulty is that the second derivative of the hinge loss $\losshng$ does only exist in a distributional sense (due to its piecewise linearity), so that it is not even locally strongly convex.
Consequently, the proofs of this work are of independent interest because they may provide a template to establish important properties of loss functions, such as \emph{restricted strong convexity} \cite{negahban2009unified,negahban2012unified,genzel2016estimation}, under fairly mild regularity assumptions.

In order to illustrate the general recovery framework of Section~\ref{sec:results}, we present some specific applications of our main results in Section~\ref{sec:appl}.
On the one hand, this concerns the choice of the signal set $\sset$ in \eqref{eq:hingeestimator} --- and following the tenor of compressed sensing, we will particularly investigate the prototypical example of sparsity as a prior on $\grtr$.
On the other hand, two standard noise models are studied in Subsection~\ref{subsec:appl:examples}, demonstrating that signal recovery is still possible in the situation of very noisy $1$-bit measurements. 
Furthermore, we conduct three numerical experiments in Subsection~\ref{subsec:appl:numerics} to provide empirical evidence of our theoretical findings.

Section~\ref{sec:literature} is then devoted to a more extensive discussion of related literature.
Besides a comparison of our approach to previous works in ($1$-bit) compressed sensing, we will also outline similarities and differences to statistical learning theory.
But let us already emphasize at this point that the problem setup of this paper rather fits into the context of signal processing than into machine learning (see also Subsection~\ref{subsec:results:generalset:geometry} for a connection to \emph{support vector machines}).
Finally, some concluding remarks as well as several open issues can be found in Section~\ref{sec:conclusion}.

\subsection{Notation}
\label{subsec:intro:notation}

Let us fix some notations and conventions that will be frequently used in the remainder of this
paper:
\begin{rmklist}
\item
	For an integer $k\in \N$, we set $[k] \coloneqq \{1, \dots, k\}$.
	Vectors and matrices are denoted by lower- and uppercase boldface letters, respectively. Unless stated otherwise, their entries are indicated by subscript indices and lowercase letters, e.g., $\x = (x_1, \dots, x_n) \in \R^n$ for a vector and $\vec{A} = [a_{j,j'}] \in \R^{n\times n'}$ for a matrix.
\item\label{subsec:intro:notation:lp}
	The \emph{support} of $\x \in \R^n$ is the set of its non-zero components, $\supp(\x) \coloneqq \{ j \in [n] \suchthat x_j \neq 0\}$, and we set $\lnorm{\x}[0] \coloneqq \cardinality{\supp(\x)}$. In particular, $\x$ is called \emph{$s$-sparse} if $\lnorm{\x}[0] \leq s$.
	For $p\geq 1$, the \emph{$\l{p}$-norm} of $\x$ is given by 
	\begin{equation}
		\lnorm{\x}[p] \coloneqq\begin{cases} 
		(\sum_{j=1}^n \abs{x_j}^p)^{1/p}, &p < \infty,\\
		\max_{j\in [n]}\abs{x_j}, & p = \infty.
		\end{cases}  
	\end{equation}
	The associated \emph{unit ball} is denoted by $\ball[p][n]\coloneqq\{\x \in \R^n \suchthat \lnorm{\x}[p] \leq 1\}$ and the \emph{Euclidean unit sphere} is $\S^{n-1} \coloneqq \{\x\in \R^n \suchthat \lnorm{\x}=1\}$.
\item
	The positive part of a number $t\in \R$ is given by $\pospart{t} \coloneqq \max\{0, t\}$. For a subset $A \subset \R^n$, we denote the associated \emph{step function} (or \emph{characteristic function}) by
	\begin{equation}
		\indset{A}(\x)\coloneqq\begin{cases}
			1, &\x\in A,\\
			0, &\text{otherwise},
	\end{cases}\quad \x\in \R^n.
	\end{equation}
\item 
	Let $\sset, \sset' \subset \R^n$ and $\grtr \in \R^n$.
	We denote the \emph{linear hull} of $\sset$ by $\spann\sset$ and its \emph{convex hull} by $\convhull\sset$. 
	The \emph{Minkowski difference} between $\sset$ and $\sset'$ is defined as $\sset-\sset' \coloneqq \{ \x - \x' \suchthat \x\in \sset, \ \x' \in \sset'\}$ and we simply write $\sset - \grtr$ instead of $\sset - \{\grtr\}$. The \emph{descent cone} of $\sset$ at $\grtr$ is given by
	\begin{equation}
		\cone{\sset}{\grtr} \coloneqq \{\tau(\x-\grtr) \suchthat \x \in \sset,\ \tau \geq 0\}.
	\end{equation}
	Furthermore, $\rad(K) \coloneqq \sup_{\x \in \sset} \lnorm{\x}$ is the \emph{radius} of $\sset$ (around $\vnull$).
	
	If $E \subset \R^n$ is a linear subspace, the associated \emph{orthogonal projection} onto $E$ is denoted by $\proj_E \in \R^{n \times n}$.
	Then, we have $\proj_{\orthcompl{E}} = \I{n} - \proj_E$, where $\orthcompl{E} \subset \R^n$ is the orthogonal complement of $E$ and $\I{n} \in \R^{n \times n}$ is the identity matrix. Moreover, if $E = \spann\{\x\}$, we use the short notations $\proj_{\x} \coloneqq \proj_{E}$ and $\proj_{\orthcompl{\x}} \coloneqq \proj_{\orthcompl{E}}$.
\item 	
	For the expected value of a random variable $X$, we write $\mean[X]$. 
	The probability of an event $A$ is denoted by $\prob[A]$ and the corresponding indicator function is $\indprob{A}$. 
	We write $\gaussian\distributed \Normdistr{\vnull}{\I{n}}$ for an $n$-dimensional standard \emph{Gaussian random vector}, and similarly, $\gaussianuniv \distributed \Normdistr{0}{\nu^2}$ is a mean-zero Gaussian variable with variance $\nu^2$.
	We call a random variable $X$ \emph{sub-Gaussian} if its \emph{sub-Gaussian norm} $\normsubg{X} \coloneqq \inf\{t>0 \suchthat \mean[e^{X^2/t^2}]\leq 2\}$ is finite, see also \cite[Def.~5.22]{vershynin2012random}.
\item
	The letter $C$ is always reserved for a (generic) constant, whose value could change from time to time.
	We refer to $C > 0$ as a \emph{numerical constant} if its value does not depend on any other involved parameter.
	If an inequality holds true up to a numerical constant $C$, we sometimes simply write $A \lesssim B$ instead of $A \leq C \cdot B$, and if $C_1\cdot A\leq B \leq C_2 \cdot A$ for numerical constants $C_1,C_2 > 0$, we use the abbreviation $A\asymp B$.
\end{rmklist}

\section{Main Results}
\label{sec:results}

This part presents our main theoretical findings on signal estimation via hinge loss minimization. In Subsection~\ref{subsec:results:modelsetup}, we start with a formal definition of the $1$-bit measurement model that was stated in \eqref{eq:onebitmodel}. Moreover, we introduce the notions of (local) Gaussian width and effective dimension, which will serve as measures of complexity for the signal set $\sset$ in \eqref{eq:hingeestimator}. Subsection~\ref{subsec:results:unitball} then deals with recovery in the Euclidean unit ball, i.e., $\sset \subset \ball[2][n]$, whereas our second main result in Subsection~\ref{subsec:results:generalset} considers general signal sets with perfect $1$-bit observations.
Note that the exposition below frequently uses the terminology of statistical learning theory, on which the proofs in Section~\ref{sec:proofs} are based. For a very brief overview of this field, see also Subsection~\ref{subsec:literature:statlearn}.

\subsection{Model Setup and Gaussian Width}
\label{subsec:results:modelsetup}
Let us first give a precise definition of the observation model that was informally introduced in \eqref{eq:onebitmodel}:
\begin{assumption}[Measurement Model]\label{model:measurements}
Let $f:\R \to \{-1,+1\}$ be a (random) quantization function and let $\a \distributed \Normdistr{\vnull}{\I{n}}$ be a standard Gaussian random vector which is independent of $f$.\footnote{Writing $f \colon \R \to \{-1,+1\}$ is of course a slight abuse of notation if $f$ is a random function. More formally, $f$ would be defined as follows in this case: Let $\tilde{f}\colon \R \times \R \to \{-1,+1\}$ and let $\tau$ be a random variable that is independent of $\a$. Then, we set $f(t) \coloneqq f_\tau(t) \coloneqq \tilde{f} (t, \tau)$ for $t \in \R$. If $f_i$ is an independent copy of $f$, we simply mean that $f_i(t) \coloneqq \tilde{f} (t, \tau_i)$ with $\tau_i$ being an independent copy of $\tau$.}
We consider a \emph{noisy $1$-bit Gaussian measurement model} of the form
\begin{equation}\label{eq:results:measurements:prototype}
	y \coloneqq f(\sp{\a}{\grtr}) \in \{-1, +1\}
\end{equation}
where $\grtr\in \R^n$ is the (unknown) \emph{ground truth signal}.
Each of the $m$ samples $\{(\a_i, y_i)\}_{i \in [m]} \subset \R^n \times \{-1,+1\}$ is then drawn as an independent copy from the random pair $(\a, y)$.
Consequently, the binary observations are given by
\begin{equation}\label{eq:measurements}
 y_i = f_i(\sp{\a_i}{\grtr}), \quad i = 1, \dots, m,
\end{equation}
where $f_i$ is an independent copy of $f$.
\end{assumption} 

As already mentioned in the introduction, a prototypical example of a $1$-bit quantizer is the sign-function, that is, $f=\sign$.
We refer to this (noiseless) observation scheme as the \emph{perfect $1$-bit model}. 
Note that all information on the magnitude of $\grtr$ is lost in this case, implying that it is impossible to determine $\lnorm{\grtr}$ from the measurements.
For that reason, we will always assume that the signal vector $\grtr$ is normalized, so that our actual goal is to recover its direction $\grtr / \lnorm{\grtr} \in \S^{n-1}$.
It is also worth emphasizing that Assumption~\ref{model:measurements} imposes almost no restrictions on the quantization function~$f$. This particularly allows us to study different sources of noise, such as \emph{random bit flips} or \emph{additive Gaussian noise}. See Subsection~\ref{subsec:appl:examples} and Subsection~\ref{subsec:appl:numerics} (Experiment~\hyperref[para:appl:numerics:exp2]{2}) for more details on these types of distortions.

In many scenarios of interest, there is some prior knowledge about the unknown signal $\grtr$ available.
The hinge loss estimator \eqref{eq:hingeestimator} encodes such additional structural assumptions by means of a convex constraint set $\sset \subset \R^n$. 
Hence, we supplement Assumption~\ref{model:measurements} with the following signal model:
\begin{assumption}[Signal Model]\label{model:signal}
We assume that $\lnorm{\grtr}=1$ and $\grtr\in \sset$ for a certain subset $\sset\subset \R^n$, which is called the \emph{signal set}. Furthermore, we require that $\sset$ is convex, bounded, and $\vnull\in \sset$.
\end{assumption}
Perhaps the most prominent (low-dimensional) signal structure is \emph{sparsity}, for instance, if we have $\lnorm{\grtr}[0] \leq s$ for some $s \ll n$. The set of all
$s$-sparse vectors (in the unit ball) is however not convex, so that Assumption~\ref{model:signal} is not fulfilled. But the Cauchy-Schwarz inequality implies that both $\sset=\sqrt{s}\ball[1][n]$ and $\sset=\sqrt{s}\ball[1][n]\intersec \ball[2][n]$ may serve as admissible convex relaxations, which meet the conditions of Assumption~\ref{model:signal}.
For more examples of signal sets, see Subsection~\ref{subsec:appl:signalsets}. 

A key issue in signal estimation concerns the number of measurements $m$ that a certain recovery procedure requires to accurately approximate the target vector $\grtr$.
With regard to empirical risk minimization in \eqref{eq:generalestimator}, we may hope that, by restricting the search space to an appropriate signal set $\sset\subset \R^n$, reconstruction even succeeds in high-dimensional situations where $m \ll n$. 
This desirable behavior obviously relies on a good choice of $\sset$ that is supposed to capture low-dimensional ``features'' of $\grtr$.
In this context, the so-called \emph{Gaussian width} has turned out to be a very useful complexity parameter. Indeed, many recent results --- including ours below --- show that the square of the Gaussian width often determines the (minimal) number of samples to ensure recovery via convex optimization \cite{chandrasekaran2012geometry,amelunxen2014edge}.

In the following, we briefly introduce all notions of complexity that are required to formulate our main results. 
Let us emphasize that these parameters originate from the field of geometric functional analysis (e.g., see \cite{gordon1985gaussian,gordon1988escape,giannopoulos2004asymptotic}), but they also appear in equivalent forms as \emph{Talagrand’s $\gamma_2$-functional} in stochastic processes (cf. \cite{talagrand2014chaining}) or as \emph{Gaussian complexity} in learning theory (cf. \cite{bartlett2003complexity}). For a more extensive discussion on their role in (signal) estimation problems and learning theory, the reader is referred to \cite{mendelson2007subgaussian,rudelson2008sparse,chandrasekaran2012geometry,plan2013robust,amelunxen2014edge,tropp2014recovery,plan2014highdim,RomanHDP}. 
Apart from that, some basic properties and specific examples are presented in Subsection~\ref{subsec:appl:signalsets}.

\begin{definition}[Gaussian Width]\label{def:gaussianwidth}
Let $\gaussian \distributed \Normdistr{\vnull}{\I{n}}$ be a standard Gaussian random vector. 
The \emph{Gaussian width} of a bounded set $\sset\subset \R^n$ is defined as 
\begin{equation}
\meanwidth{\sset}\coloneqq\mean[\sup_{\x \in \sset}\sp{\gaussian}{\x}].
\end{equation}
\end{definition}
See Figure~\ref{fig:results:modelsetup:meanwidth:global} for an illustration of the Gaussian width and its geometric meaning.
Since our task is to estimate a fixed vector $\grtr$, it is natural to measure the complexity of the signal set $\sset$ in a small neighborhood of $\grtr$, rather than computing the ``global'' width $\meanwidth{\sset}$.
To this end, let us consider the following localized version: For $t>0$, the \emph{local Gaussian width} of a subset $\sset\subset \R^n$ at scale $t>0$ is given by
\begin{equation}
	\meanwidth[t]{\sset}\coloneqq\meanwidth{\sset\intersec t\ball[2][n]}.
\end{equation}
In particular, we call $\meanwidth[t]{\sset-\grtr}$ the \emph{local Gaussian width} of $\sset$ in $\grtr$ at scale $t>0$; see also Figure~\ref{fig:results:modelsetup:meanwidth:local}.
Note that the local Gaussian width is always bounded from above by its global counterpart:
\begin{equation}\label{eq:locmwlessglob}
	\meanwidth[t]{\sset}=\mean[\sup_{\x \in \sset\intersec t\ball[2][n]}\sp{\gaussian}{\x}] \leq \mean[\sup_{\x \in \sset}\sp{\gaussian}{\x}]=\meanwidth{\sset}.
\end{equation}

\begin{figure}[!t]
	\centering
	\begin{subfigure}[t]{0.4\textwidth}
		\centering
		\tikzstyle{blackdot}=[shape=circle,fill=black,minimum size=1mm,inner sep=0pt,outer sep=0pt]
		\begin{tikzpicture}[scale=2]
			\coordinate (K1) at (-.5,-.3);
			\coordinate (K2) at (.5,-1);
			\coordinate (K3) at (.3,-1.7);
			\coordinate (K4) at (-.2,-2.1);
			\coordinate (K5) at (-1,-1.2);
						
			\draw[fill=gray!20!white,thick] (K1) -- (K2) -- (K3) -- (K4) -- (K5) -- cycle;
			\node[blackdot,label=below :$\vnull$] (P0) at (barycentric cs:K1=0.9,K2=0.5,K3=0.5,K4=0.5,K5=0.8) {};
			\node[blackdot] at (K1) {};
			\node[below right=.6cm and .2cm of P0] {$\sset$};
			
			\path [name path=P0perp] ($(P0) + (110:1.2cm)$) -- ($(P0) + (290:1.2cm)$);
			\path [name path=xbarhyper] ($(K1) + (20:1.2cm)$) -- ($(K1) + (200:1.2cm)$);
			\path [name intersections={of=P0perp and xbarhyper,by={K1anchor}}];
			
			\draw[dashed] ($(P0) + (20:1.2cm)$) node (P0R) {} -- ($(P0) + (200:1.2cm)$) node (P0L) {};
			\draw[dashed] ($(K1anchor) + (20:1.2cm)$) node (K1R) {} -- ($(K1anchor) + (200:1.2cm)$) node (K1L) {};
			
			\draw[<->,>=stealth,thick,dashed] ($(P0L)!.1!(P0R)$) -- ($(K1L)!.1!(K1R)$) node[pos=.4,left] {$\tfrac{1}{\lnorm{\gaussian}} \displaystyle \sup_{\x \in \sset} \sp{\gaussian}{\x}$};
			\draw[->,>=stealth, thick] (P0) -- ($(P0)!.5!(K1anchor)$) node[midway,above right] {$\gaussian$};
		\end{tikzpicture}
		\caption{}
		\label{fig:results:modelsetup:meanwidth:global}
	\end{subfigure}%
	\qquad\qquad
	\begin{subfigure}[t]{0.4\textwidth}
		\centering
		\tikzstyle{blackdot}=[shape=circle,fill=black,minimum size=1mm,inner sep=0pt,outer sep=0pt]
		\begin{tikzpicture}[scale=2]
			\coordinate (K1) at (-.5,-.3);
			\coordinate (K2) at (.5,-1);
			\coordinate (K3) at (.3,-1.7);
			\coordinate (K4) at (-.2,-2.1);
			\coordinate (K5) at (-1,-1.2);
			\coordinate[below right=.2cm and .03 of K1] (P0);
			
			\draw[fill=gray!20!white, name path = K] (K1) -- (K2) -- (K3) -- (K4) -- (K5) -- cycle;
			\node[draw,circle,fill=gray!20!white,minimum size=1.2cm,label={[label distance=-2pt]30:$tB_2^n$}] at (P0) {}; 
			\begin{scope}
				\clip (K1) -- (K2) -- (K3) -- (K4) -- (K5) -- cycle;
				\node[draw,circle,fill=gray!60!white,minimum size=1.2cm] at (P0) {};
			\end{scope}
			\draw[thick] (K1) -- (K2) -- (K3) -- (K4) -- (K5) -- cycle;
			
			\node at (barycentric cs:K1=1,K2=1,K3=1,K4=1,K5=1) {$\sset - \grtr$};
			\node[blackdot,label=below :$\vnull$] at (P0) {};
		\end{tikzpicture}
		\caption{}
		\label{fig:results:modelsetup:meanwidth:local}
	\end{subfigure}%
	\caption{\subref{fig:results:modelsetup:meanwidth:global}~Visualization of the Gaussian width. If $\gaussian$ is fixed, $\sup_{\x \in \sset}\sp{\gaussian}{\x}$ measures the (scaled) spatial extent of $\sset$ in the direction of $\gaussian$. The expected value therefore computes the ``average'' width of $\sset$. \subref{fig:results:modelsetup:meanwidth:local}~Visualization of the local Gaussian width $\meanwidth[t]{\sset-\grtr}$ in $\grtr$: The dark gray region corresponds to the intersection $(\sset - \grtr) \intersec t\ball[2][n]$ over which the Gaussian width is computed. The scale $t$ determines the size (``resolution'') of the considered neighborhood of $\grtr$.}
	\label{fig:results:modelsetup:meanwidth}
\end{figure}

In order to relate the (local) Gaussian width to the number of samples $m$, it is very convenient to work with a scaling invariant complexity parameter.
This property is achieved by the notion of \emph{effective dimension},
\begin{equation}
\effdim{\sset}\coloneqq\frac{\meanwidth{\sset}^2}{\diam(\sset)^2},
\end{equation} 
where $\diam(\sset)$ denotes the Euclidean diameter of the subset $\sset \subset \R^n$.
If Assumption~\ref{model:signal} holds true and $\sset \subset \ball[2][n]$ (as in Subsection~\ref{subsec:results:unitball}), it particularly follows that $\diam(\sset) \asymp 1$ due to $\vnull, \grtr \in \sset$, so that we have $\effdim{\sset}\asymp \meanwidth{\sset}^2$.

Combining the above concepts of localization and scaling invariance (cf. \cite{mendelson2002improving,bartlett2005local,mendelson2007subgaussian}), we now introduce the local effective dimension. The following definition is adopted from the more recent works of \cite{plan2014highdim,plan2015lasso,genzel2016estimation} and in fact forms a crucial ingredient of our first recovery result in Theorem~\ref{thm:euclideanball}.
\enlargethispage{-\baselineskip}
\begin{definition}[Local Effective Dimension]
The \emph{local effective dimension} of a subset $\sset\subset \R^n$ in $\grtr$ at scale $t>0$ is given by
\begin{equation}\label{eq:localeffecversion}
	\effdim[t]{\sset-\grtr} \coloneqq \frac{\meanwidth[t]{\sset-\grtr}^2}{t^2}.
\end{equation}
\end{definition}
Since $\meanwidth[t]{\sset-\grtr} = \meanwidth{(\sset-\grtr) \intersec t\ball[2][n]}$, the scaling parameter $t$ essentially captures the diameter of the local neighborhood of $\grtr$ and therefore normalizes the local Gaussian width. 
More precisely, if Assumption~\ref{model:signal} is satisfied and $t \leq 1$, we have $\diam((\sset-\grtr) \intersec t\ball[2][n]) \asymp t$, which in turn implies that
\begin{equation}\label{eq:localeffdimeqformulation}
	\effdim[t]{\sset-\grtr} \asymp \effdim{(\sset-\grtr)\intersec t\ball[2][n]}.
\end{equation}
Thus, the local effective dimension is equivalent to the effective dimension of the shifted and localized signal set $(\sset-\grtr) \intersec t\ball[2][n]$.

To better understand the role of the scale $t$ in $\effdim[t]{\sset-\grtr}$, let us consider the limit case $t \to 0$, or with other words, what happens if the neighborhood of $\grtr$ becomes infinitesimal small.
For this purpose, we first rewrite the local effective dimension as follows:
\begin{equation}
	\effdim[t]{\sset-\grtr} = \tfrac{1}{t^2}\meanwidth{(\sset-\grtr) \intersec t \ball[2][n]}^2 = \meanwidth{\tfrac{1}{t}(\sset-\grtr) \intersec \ball[2][n]}^2.
\end{equation}
Hence, if $\sset$ is convex, it is not hard to see that
\begin{equation}\label{eq:effdimtocondim}
\meanwidth{\tfrac{1}{t}(\sset-\grtr) \intersec \ball[2][n]}^2\to \meanwidth{\cone{\sset}{\grtr} \intersec \ball[2][n]}^2,\quad  \text{ as $t\to 0$},
\end{equation}
where $\cone{\sset}{\grtr}$ is the descent cone of $\sset$ at $\grtr$.
The limit on the right-hand side of \eqref{eq:effdimtocondim} is well-known as conic effective dimension or statistical dimensional in the literature \cite{amelunxen2014edge}. 
Let us provide a formal definition of this important quantity because it will reappear in the hypotheses of Theorem~\ref{thm:results:generalset:global} and Theorem~\ref{thm:results:generalset:local} below:
\begin{definition}[Conic Effective Dimension]
The \emph{conic effective dimension} of a subset $\sset\subset \R^n$ in $\grtr$ is defined as
\begin{equation}
\effdim[\conic]{\sset - \grtr}\coloneqq\meanwidth{\cone{\sset}{\grtr} \intersec \ball[2][n]}^2.
\end{equation}
\end{definition}
From a geometric viewpoint, $\effdim[\conic]{\sset - \grtr}$ measures the size (narrowness) of the cone generated by $\sset - \grtr$.
While this complexity parameter is conceptually simple, it leads to problems if $\grtr$ lies in the interior of the signal set $\sset$ and not exactly on its boundary.
Supposed that $\spann\sset = \R^n$, we would then have $\cone{\sset}{\grtr} = \R^n$, which in turn implies $\effdim[\conic]{\sset - \grtr} \asymp n$.
In contrast, the local effective dimension does not suffer from such an ``unstable'' behavior, since it avoids to take the conic hull of $\sset - \grtr$.
Indeed, due to 
\begin{equation}\label{eq:efflessconic}
	\effdim[t]{\sset - \grtr} \leq \effdim[t]{\cone{\sset}{\grtr}} = \frac{\meanwidth[t]{\cone{\sset}{\grtr}}^2}{t^2} =\meanwidth[1]{\cone{\sset}{\grtr}}^2=\effdim[\conic]{\sset - \grtr},
\end{equation}
we conclude that the local effective dimension can be (much) smaller than its conic counterpart, e.g., in the situation of Figure~\ref{fig:results:modelsetup:meanwidth:local}.
With regard to our specific recovery problem, Theorem~\ref{thm:euclideanball} below shows that, by tolerating a reconstruction error of order $t$, the sampling rate is actually determined by $\effdim[t]{\sset - \grtr}$, rather than $\effdim[\conic]{\sset - \grtr}$.
For a more extensive discussion of the local and conic effective dimension, see also \cite[Sec.~III.D]{genzel2016estimation}.

\subsection{Recovery in Subsets of the Unit Ball}
\label{subsec:results:unitball}

This subsection investigates signal recovery under Assumption~\ref{model:measurements}~and~\ref{model:signal} but with the restriction that the signal set belongs to the Euclidean unit ball, i.e., $\sset \subset \ball[2][n]$.
The proofs of all results stated below are postponed to Subsection~\ref{sec:proof:unitball}.
Let us first recall the hinge loss estimator:
\begin{definition}[Hinge Loss Minimization]\label{def:results:unitball:estimator}
Let the model conditions of Assumption~\ref{model:measurements}~and~\ref{model:signal} hold true.
A \emph{hinge loss estimator} $\solu \in \R^n$ is defined as a solution of the convex program
\begin{equation}\label{eq:estimator}\tag{$P_{\losshng,\sset}$}
		\min_{\x \in \R^n} \tfrac{1}{m} \sum_{i = 1}^m \losshng(y_i\sp{\a_i}{\x}) \quad \text{subject to \quad $\x \in \sset$,}
\end{equation}
where $\losshng(v) \coloneqq \pospart{1 - v} = \max\{0, 1 - v\}$ denotes the \emph{hinge loss function}.
\end{definition}
Since \eqref{eq:estimator} is a particular instance of empirical risk minimization, we also introduce the notion of risk function, which is very common in statistical learning theory:
\begin{definition}[Empirical and Expected Risk Function]\label{def:results:unitball:risk}
The objective functional of \eqref{eq:estimator} is called \emph{empirical risk function} and denoted by 
\begin{equation}
\lossemp[](\x)\coloneqq\tfrac{1}{m} \sum_{i = 1}^m \losshng(y_i\sp{\a_i}{\x}).
\end{equation}
Its expectation (over all involved random variables) is given by
\begin{equation}
\lossexp(\x)\coloneqq\mean[\lossemp[](\x)]=\mean[\losshng(y \sp{\a}{\x})],
\end{equation}
which is referred to as the \emph{expected risk function}.
\end{definition}

For a better understanding of the formal arguments below, let us briefly outline the basic idea behind signal recovery via empirical risk minimization:
The first step is to verify that any minimizer of the expected risk,
\begin{equation}\label{eq:results:unitball:lossexp}
	\min_{\x\in \sset}\lossexp(\x) = \min_{\x\in \sset} \mean[\losshng(y \sp{\a}{\x})],
\end{equation} 
is well-behaved in the sense that it is sufficiently close to $\spann\{\grtr\}$. Indeed, Lemma~\ref{lem:minexpectedloss} below states that there exists an expected risk minimizer taking the form $\grtrmu$ for some $0 < \scalfac \leq 1$, which in turn is a consequence of the unit-ball assumption $\sset \subset \ball[2][n]$.
According to the law of large numbers, one can therefore expect that $\lossemp[](\grtrmu)\approx \lossemp[](\solu)$ as $m$ gets sufficiently large.
The second key step is then to exploit the convexity of $\lossemp(\cdot)$ to show that one even has $\grtrmu \approx \solu$.
More specifically, it will turn out that $\lossemp[](\cdot)$ satisfies \emph{restricted strong convexity} in a certain neighborhood of $\grtrmu$.
This finding is somewhat surprising, since the hinge loss function $\losshng(\cdot)$ is not even locally strongly convex. In fact, its second derivative just corresponds to a $\delta$-distribution shifted to $1$, so that intuitively speaking, $\losshng(\cdot)$ is locally strongly convex in an \emph{infinitesimal} neighborhood of $1$.
But nevertheless, this very mild form of convexity is still a sufficient criterion for restricted strong convexity of the associated empirical risk function $\lossemp[](\cdot)$, see Proposition~\ref{prop:quadprocesslowerbound} and Remark~\ref{rmk:proofs:unitball:rsc} for more technical details.
Finally, a simple normalization step leads to an error bound in the form of \eqref{eq:intro:errbound}.
The interested reader is referred to \cite{mendelson2014learninggeneral,genzel2016estimation} for a more detailed discussion of this strategy, which also applies to many other convex loss functions.

According to this roadmap, our first task is to study the expected risk minimizer of \eqref{eq:results:unitball:lossexp}:
\begin{lemma}\label{lem:minexpectedloss}
Let $\gaussianuniv \distributed \Normdistr{0}{1}$ be a standard Gaussian random variable. Moreover, let $\scalfac \in [0,1]$ be a minimizer of
\begin{equation}\label{eq:minexpectedloss:scalfac}
\min_{s\in [0,1]}\mean[\losshng(s f(\gaussianuniv)\gaussianuniv)], 
\end{equation}
where $f \colon \R \to \{-1,+1\}$ is the $1$-bit quantizer from Assumption~\ref{model:measurements}.
Assuming that $\mean[f(\gaussianuniv)\gaussianuniv]>0$, we have that $\scalfac>0$ and $\grtrmu\in \sset$ is a minimizer of \eqref{eq:results:unitball:lossexp}, i.e.,
\begin{equation}
\lossexp(\grtrmu)=\min_{\x\in \sset}\lossexp(\x).
\end{equation}
\end{lemma}
Let us emphasize that $\mean[f(\gaussianuniv)\gaussianuniv]>0$ is a reasonable assumption because it ensures that the linear signal $\gaussianuniv = \sp{\a}{\grtr}$ and the output variable $y = f(\gaussianuniv)$ are positively correlated. With this in mind, the correlation parameter $\lambda_f \coloneqq \mean[f(\gaussianuniv)\gaussianuniv]$ essentially captures the signal-to-noise ratio of the measurement rule \eqref{eq:results:measurements:prototype}.
In particular, a large value of $\lambda_f$ implies that the expected risk minimizer $\grtrmu$ is not too close to $\vnull$ (see Lemma~\ref{lem:lowerboundmu}).
On the other hand, if $\lambda_f = 0$, it can happen that all information on $\grtr$ is completely buried in noise and recovery is impossible, e.g., if $p = \tfrac{1}{2}$ in Example~\ref{ex:appl:noisepatterns}\ref{ex:bitflip}. But we point out that $\lambda_f > 0$ is still not a necessary condition for recovery in general. For instance, one could imagine a binary even output function, similar to real-valued phaseless measurements generated by $f = \abs{\cdot}$. However, one would have to deal with a sign ambiguity in that case and a simple convex estimator such as \eqref{eq:estimator} becomes useless.

Besides the correlation assumption $\lambda_f > 0$, we need a second mild condition on the quantizer $f$ in order to formulate our main recovery result Theorem~\ref{thm:euclideanball}:
\begin{assumption}[Correlation Conditions]\label{model:correlation}
Let $\gaussianuniv \distributed \Normdistr{0}{1}$ be a standard Gaussian random variable and let $f \colon \R \to \{-1,+1\}$ be the $1$-bit quantizer from Assumption~\ref{model:measurements}. Assume that the following two model conditions hold true:
\begin{enumerate}[label=(C\arabic*)]
\item \label{enum:cond_c1} $\lambda \coloneqq\lambda_f\coloneqq\mean[f(\gaussianuniv)\gaussianuniv]>0$,
\item \label{enum:cond_c2} $\mean[f(\gaussianuniv)\sign(\gaussianuniv)\suchthat \abs{\gaussianuniv}]\geq 0\quad\quad \text{(a.s.)}.$
\end{enumerate}
We call $\lambda$ the \emph{correlation parameter} of the quantizer $f$.
\end{assumption}
The statistical meaning of \ref{enum:cond_c1} and~\ref{enum:cond_c2} is studied in greater detail in Subsection~\ref{subsec:appl:examples}.
In this context, we will also show that both conditions are easily satisfied for two prototypical sources of model corruptions, namely \emph{random bit flip noise} (after quantization) and \emph{additive Gaussian noise} (before quantization).

We are now ready to state the main result of this subsection:
\begin{theorem}[Signal Recovery in Unit Ball]\label{thm:euclideanball} 
	Let the model conditions of Assumption~\ref{model:measurements},~\ref{model:signal},~and~\ref{model:correlation} be satisfied,
	assume that $\sset\subset \ball[2][n]$, and let $\scalfac$ be defined according to \eqref{eq:minexpectedloss:scalfac}.
	For every $t \in \intvop{0}{\scalfac}$ and $\eta\in \intvop{0}{\tfrac{1}{2}}$, the following holds true with probability at least $1 - \eta$:
	If the number of samples obeys
	\begin{equation}\label{eq:number_of_measurements}
		m\gtrsim \lambda^{-2}\cdot t^{-2} \cdot \max\{\effdim[t]{\sset - \grtrmu}, \log(\eta^{-1}) \},
	\end{equation}
	then any minimizer $\solu$ of \eqref{eq:estimator} satisfies
	\begin{equation}\label{eq:errorboundunitball} 
		\lnorm[auto]{\grtr - \frac{\solu}{\lnorm{\solu}}} \leq \frac{t}{\scalfac}\lesssim t \cdot \sqrt{\log(\lambda^{-1})} \ .
	\end{equation}
\end{theorem}

A remarkable feature of Theorem~\ref{thm:euclideanball} is that the impact of the underlying $1$-bit measurement model is completely controlled by the correlation parameter $\lambda$. 
Since $\lambda$ can be regarded as a constant scaling factor, recovery via \eqref{eq:estimator} is still possible when the specific output rule is unknown and the signal-to-noise ratio is very low.
With other words, under the hypothesis of Assumption~\ref{model:correlation}, $\solu / \lnorm{\solu}$ constitutes a consistent estimator of $\grtr$ even if $\lambda$ is relatively close to $0$.
Note that the above recovery statement strongly resembles the non-uniform results of \cite{plan2015lasso, genzel2016estimation}, where the uncertainty about the non-linear output function $f$ is also captured by a few model-dependent parameters.

We wish to emphasize that the error tolerance $t > 0$ needs to be fixed in advance. Thus, one may rewrite \eqref{eq:number_of_measurements} as a condition on $t$:
\begin{equation}\label{eq:minimalnumber}
	t\gtrsim \lambda^{-1}\cdot \Big(\frac{\max\{\effdim[t]{\sset - \grtrmu}, \log(\eta^{-1}) \}}{m}\Big)^{1/2},
\end{equation}
but note that this is still quite implicit, as the right-hand side also depends on $t$.
If $m$ is now adjusted such that \eqref{eq:minimalnumber} holds true with equality (up to numerical constants), the actual error bound \eqref{eq:errorboundunitball} can be directly related to the number of samples:
\begin{equation}\label{eq:effdimbound} 
	\lnorm[auto]{\grtr - \frac{\solu}{\lnorm{\solu}}} \lesssim \scalfac^{-1}  \cdot\lambda^{-1}\cdot \Big(\frac{\max\{\effdim[t]{\sset - \grtrmu}, \log(\eta^{-1}) \}}{m}\Big)^{1/2}.
\end{equation}
While this expression already exhibits the optimal oversampling rate of $\asympfaster{m^{-1/2}}$, the local effective dimension on the right-hand side still explicitly depends on the parameter $t$.
The following corollary shows that one can easily get rid of this dependence by applying \eqref{eq:locmwlessglob} or \eqref{eq:efflessconic}.
\begin{corollary}\label{coro:euclideanball}
The assertion of Theorem~\ref{thm:euclideanball} still holds true if the condition \eqref{eq:number_of_measurements} is replaced by 
\begin{equation}
m\gtrsim \lambda^{-2}\cdot t^{-2} \cdot \max\{\effdim[\conic]{\sset - \grtrmu}, \log(\eta^{-1})\},
\end{equation}
or by
\begin{equation}
m\gtrsim \lambda^{-2}\cdot t^{-4} \cdot \max\{\meanwidth{K}^2, \log(\eta^{-1})\}.
\end{equation}
\end{corollary}
Adjusting $m$ similarly to \eqref{eq:effdimbound}, we now obtain the non-asymptotic error bounds
\begin{equation}\label{eq:conicbound} 
\lnorm[auto]{\grtr - \frac{\solu}{\lnorm{\solu}}} \lesssim \Big(\frac{\max\{\effdim[\conic]{\sset - \grtrmu}, \log(\eta^{-1}) \}}{m}\Big)^{1/2},
\end{equation} 
and
\begin{equation}\label{eq:globalbound} 
\lnorm[auto]{\grtr - \frac{\solu}{\lnorm{\solu}}} \lesssim \Big(\frac{\max\{\meanwidth{K}^2, \log(\eta^{-1}) \}}{m}\Big)^{1/4},
\end{equation}
where the constants hide the dependence on the model parameters $\scalfac$ and $\lambda$.
The error estimate promoted by \eqref{eq:conicbound} is desirable due to the optimal approximation rate of $\asympfaster{m^{-1/2}}$, but as already pointed out at the end of Subsection~\ref{subsec:results:modelsetup}, using the conic effective dimension as complexity measure has several downsides.
This particularly concerns the issue of stable recovery in situations where $\grtr$ is only close to the boundary of $\sset$ or if the boundary of $\sset$ is (locally) smooth.
In either of these scenarios, the recovery statement of Theorem~\ref{thm:euclideanball} is significantly stronger because it relates the error accuracy to the local effective dimension at the right scale.

Compared to this, the bound of \eqref{eq:globalbound} does not even depend on the ground truth signal $\grtr$ and therefore, in principle, holds true for every $\grtr \in \sset$.\footnote{But note that our recovery results are non-uniform, i.e., the ground truth signal needs to be fixed in advance.}
Apart from that, there exist explicit upper bounds for the global Gaussian width in many cases of interest, see \cite{vershynin2014estimation} for example.
But these appealing features of \eqref{eq:globalbound} clearly come along with a much slower oversampling rate of $\asympfaster{m^{-1/4}}$.

\subsection{Recovery in General Convex Sets}
\label{subsec:results:generalset}

The crucial assumption of the previous subsection was that the signal set belongs to the Euclidean unit ball, meaning that $\sset \subset \ball[2][n]$. According to Lemma~\ref{lem:minexpectedloss}, we were able to ensure that the minimizer of the expected risk \eqref{eq:results:unitball:lossexp} takes the form $\grtrmu$ with $\scalfac \in \intvopcl{0}{1}$.
Consequently, Theorem~\ref{thm:euclideanball} states that the normalized empirical risk minimizer $\solu / \lnorm{\solu}$ of \eqref{eq:estimator} constitutes a consistent estimator of $\grtr$.
On the other hand, it can be computationally appealing to drop the unit-ball assumption and to allow for ``larger'' convex signal sets.
A very common example is an $\l{1}$-penalty, for which \eqref{eq:estimator} can be reformulated as a linear program (cf. \cite[Sec.~VI.A]{kolleck2015l1svm}).

This motivates us to investigate the recovery performance of hinge loss minimization under arbitrary convex constraints.
More precisely, we will make the following model assumptions throughout this subsection:\enlargethispage{-1.5\baselineskip}
\begin{assumption}[General Signal Sets]\label{model:results:generalset}
	Let $\grtr \in \S^{n-1}$ be a unit vector in $\R^n$ and let $\a \distributed \Normdistr{\vnull}{\I{n}}$ be a standard Gaussian.
	We consider \emph{perfect $1$-bit Gaussian measurements} 
	\begin{equation}
		y \coloneqq \sign(\sp{\a}{\grtr}) \in \{-1, +1\}.
	\end{equation}
	Each of the $m$ samples $\{(\a_i, y_i)\}_{i \in [m]}$ is then drawn as an independent copy from the random pair $(\a, y)$, implying that
	\begin{equation}\label{eq:results:generalset:meas}
		y_i = \sign(\sp{\a_i}{\grtr}), \quad i = 1, \dots, m.
	\end{equation}
	We assume that $\grtr \in \sset$ for a fixed signal set $\sset \subset \R^n$ which is convex, bounded, and closed.
\end{assumption}	
Note that this model includes a special case of Assumption~\ref{model:measurements} with $f = \sign$ and also revives the requirements of Assumption~\ref{model:signal}, except from $\vnull \in \sset$.
The restriction to perfect $1$-bit measurements in fact simplifies the argumentation significantly.
More precisely, the general noisy case would involve a very subtle case distinction depending on whether the expected risk minimizer lies in the interior or on the boundary of $\sset$. The former case actually corresponds to $\scalfac < 1$ in Lemma~\ref{lem:minexpectedloss} and one may in principle proceed similarly to the proof of Theorem~\ref{thm:euclideanball}. In the latter case --- which occurs for noiseless measurements --- the strategy of the next subsection could prove useful.
However, none of these extensions is straightforward and they would require a considerable technical effort.
Furthermore, the numerical analysis of Experiment~\hyperref[para:appl:numerics:exp2]{2} in Subsection~\ref{subsec:appl:numerics} indicates that this step would particularly require an adaption of the scalable estimator introduced below in Definition~\ref{def:estimatortuned}.
In order to highlight the geometric aspects of our recovery approach, we will therefore only investigate the noiseless case in this part.

\subsubsection{Recovery via Scalable Signal Sets}
\label{subsec:results:generalset:scalable}

One of the major difficulties in the general setup of Assumption~\ref{model:results:generalset} is that an expected risk minimizer of the hinge loss is not necessarily a scalar multiple of $\grtr$ anymore, i.e.,
\begin{equation}
	\argmin_{\x \in \sset} \lossexp(\x) \not\subset \spann\{\grtr\}.
\end{equation}
Indeed, by an orthogonal decomposition $\x = \proj_{\grtr}(\x) + \proj_{\orthcompl{\grtr}}(\x) = \sp{\x}{\grtr}\grtr + \proj_{\orthcompl{\grtr}}(\x)$, the expected risk at $\x \in \sset$ takes the form (cf. Definition~\ref{def:results:unitball:risk}):
\begin{align}
	\lossexp(\x) = \mean[\losshng(y \sp{\a}{\x})] &= \mean{}[\losshng(\sp{\x}{\grtr}\abs{\sp{\a}{\grtr}} + y\sp{\a}{\proj_{\orthcompl{\grtr}}(\x)})] \\ 
	&= \mean{}\big[\pospart[\big]{1 - \nu_{\x}\abs{\gaussianuniv} + \orthcompl{\nu}_{\x} \orthcompl{\gaussianuniv}} \big] \label{eq:results:generalset:exploss}
\end{align}
where $\nu_{\x} \coloneqq \sp{\x}{\grtr}$, $\orthcompl{\nu}_{\x} \coloneqq \lnorm{\proj_{\orthcompl{\grtr}}(\x)}$, and $\gaussianuniv, \orthcompl{\gaussianuniv} \distributed \Normdistr{0}{1}$ are independent.
Hence, we can conclude that, in order to minimize $\lossexp(\cdot)$, it is beneficial to select $\x \in \sset$ such that $\nu_{\x}$ is large and $\orthcompl{\nu}_{\x}$ is small at the same time.
As long as $\sset \subset \ball[2][n]$, Lemma~\ref{lem:minexpectedloss} simply states that this trade-off is satisfied for $\x = \grtr$, i.e., $\nu_{\x} = \sp{\grtr}{\grtr} = 1$ and $\orthcompl{\nu}_{\x} = 0$.\footnote{The proof of Lemma~\ref{lem:minexpectedloss} particularly makes use of the fact that $\abs{\sp{\x}{\grtr}} \leq 1$ for all $\x \in \sset \subset \ball[2][n]$.}
But without the unit-ball assumption, there might exist $\x \in \sset$ with $\nu_{\x} > 1$ and $\orthcompl{\nu}_{\x} > 0$ such that $\lossexp(\x) < \lossexp(\grtr)$. See Figure~\ref{fig:results:generalset:exprisk} for an illustration.
While this situation might appear somewhat artificial in two dimensions, it is actually characteristic for high-dimensional convex sets, which implies that a (normalized) minimizer $\solu$ of \eqref{eq:estimator} is not expected to be close to $\grtr$, even if $m \to \infty$. For numerical evidence of this claim, we refer to Experiment~\hyperref[para:appl:numerics:exp1]{1} in Subsection~\ref{subsec:appl:numerics}.

\begin{figure}
	\centering
		\tikzstyle{blackdot}=[shape=circle,fill=black,minimum size=1mm,inner sep=0pt,outer sep=0pt]
		\begin{tikzpicture}[scale=2,extended line/.style={shorten >=-#1,shorten <=-#1},extended line/.default=1cm]]
			\coordinate (K1) at (0,0);
			\coordinate (K2) at (.85,-1.13333);
			\coordinate (K3) at (.5,-1.75);
			\coordinate (K4) at (-.5,-1.5);
			\coordinate (K5) at (-.5,-.75);
			\coordinate (P0) at (0,-1);
			\coordinate (X0) at (.25,-.33333);
			\coordinate (X) at (.1,-.13333);
			
			\coordinate (Xproj) at ($(P0)!(X)!90:(X0)$);
						
			\draw[fill=gray!20!white,thick] (K1) -- (K2) -- (K3) -- (K4) -- (K5) -- cycle;
			\node at (barycentric cs:K1=1,K2=1,K3=4,K4=1,K5=1) {$\sset$};
			\node[draw,dashed,label={[label distance=-3pt]above left:$\S^{n-1}$}] at (P0) [circle through={(X0)}] {};
			\node[blackdot, label={[label distance=-1pt]0:$\vnull$}] at (P0) {};
			\node[blackdot, label=right:$\grtr$] at (X0) {};
			\node[blackdot, label={[label distance=-3pt]80:$\x$}] at (X) {};
			
			\draw[<->,>=stealth,shorten >=1pt] (Xproj) -- (P0) node[pos=.4,yshift=1pt,below] {\smaller$\orthcompl{\nu}_{\x}$};
			\draw[<->,>=stealth,shorten >=1pt] (Xproj) -- (X) node[pos=.5,left,xshift=2pt] {\smaller$\nu_{\x}$};
						
			\draw[dashed] ($(X0)!-.75!(P0)$) -- ($(X0)!2.75!(P0)$);
			
			\node[label={[label distance=-3pt]left:$\spann\{\grtr\}$}] at ($(X0)!2.35!(P0)$) {};
		\end{tikzpicture}
	\caption{If $\sset \not\subset \ball[2][n]$, there might exist $\x \in \sset$ such that $\nu_{\x} = \sp{\x}{\grtr} > 1$ and $\orthcompl{\nu}_{\x} \coloneqq \lnorm{\proj_{\orthcompl{\grtr}}(\x)} > 0$, as illustrated in the above figure. This particularly implies that $\x \not\in \spann\{\grtr\}$ and $\lnorm{\x} > 1$.}
	\label{fig:results:generalset:exprisk}
\end{figure}

In order to come up with an improved estimation strategy, let us first make an important observation about the expected risk minimizer if the signal set $\sset$ is upscaled by a factor $\scalfac \gtrsim 1$:\enlargethispage{-1.5\baselineskip}
\begin{proposition}\label{prop:asympexprisk}
	Let Assumption~\ref{model:results:generalset} be satisfied and assume that $\scalfac \gtrsim 1$.
	Then every expected risk minimizer $\x^\ast$ on $\scalfac\sset$ (i.e., $\lossexp(\x^\ast) = \min_{\x \in \scalfac\sset} \lossexp(\x)$) satisfies
	\begin{equation}\label{eq:asympexprisk:bound}
		\lnorm[auto]{\grtr - \frac{\x^\ast}{\lnorm{\x^\ast}}} \lesssim \frac{1}{\scalfac} \ .
	\end{equation}
\end{proposition}
See Subsection~\ref{subsec:proofs:asympexprisk} for a proof of this result.
The statement of Proposition~\ref{prop:asympexprisk} can be understood as a relaxation of Lemma~\ref{lem:minexpectedloss}: While the expected risk minimizer $\x^\ast$ on $\scalfac\sset$ may not exactly be a scalar multiple of $\grtr$, its direction $\x^\ast / \lnorm{\x^\ast}$ at least aligns well with $\grtr$, supposed that $\scalfac$ is sufficiently large --- or more geometrically speaking, the angle between $\x^\ast$ and $\grtr$ gets very small.
With this in mind, a key achievement of our main results below is that the assertion of Proposition~\ref{prop:asympexprisk} essentially remains valid for an \emph{empirical} risk minimizer on $\scalfac\sset$, leading to a consistent estimator of $\grtr$.
For this purpose, let us introduce an adapted version of \eqref{eq:estimator} that allows us to rescale the signal set:
\begin{definition}[Scalable Hinge Loss Minimization]\label{def:estimatortuned}
	Let Assumption~\ref{model:results:generalset} hold true and let $\scalfac > 0$ be a fixed \emph{scaling parameter}.\footnote{Note that the letter $\scalfac$ is already used in the course \eqref{eq:minexpectedloss:scalfac} in order to define an expected risk minimizer in subsets of the unit ball. It will be clear from the context which definition is meant. However, this ambiguous notation is no coincidence because we will analyze the excess risk functional with respect to $\grtrmu$ in both cases (see beginning of Section~\ref{sec:proofs}).} The estimator $\solu \in \R^n$ is defined as a solution of the convex program
	\begin{equation}\label{eq:estimatortuned}\tag{$P_{\losshng,\scalfac\sset}$}
		\min_{\x \in \R^n} \tfrac{1}{m} \sum_{i = 1}^m \losshng(y_i \sp{\a_i}{\x}) \quad \text{subject to \quad $\x \in \scalfac \sset$.}
	\end{equation}
	As before, the objective functional of \eqref{eq:estimatortuned} is denoted by $\lossemp[](\x)$ and called the \emph{empirical risk function} (see Definition~\ref{def:results:unitball:risk}).
\end{definition}
Note that the choice $\scalfac = 1$ exactly corresponds to the estimator \eqref{eq:estimator} considered in Subsection~\ref{subsec:results:unitball}.
By the law of large numbers, we may hope again that $\solu$ approximates an expected risk minimizer as $m$ grows. According to Proposition~\ref{prop:asympexprisk}, this would basically imply an error bound of the form
\begin{equation}
	\lnorm[auto]{\grtr - \frac{\solu}{\lnorm{\solu}}} \lesssim \frac{1}{\scalfac} \ ,
\end{equation}
which enables quantitative control over the accuracy of the hinge loss estimator \eqref{eq:estimatortuned} by means of $\scalfac$.
However, this approximation rate in $\scalfac$ cannot be independent of the sample count $m$:
As long as $m$ is fixed, one could simply enlarge $\scalfac$ such that $\lossemp(\grtrmu) = 0$. The solution set of \eqref{eq:estimatortuned} would then become highly non-unique and any $\solu \in \scalfac\sset$ with $\lossemp(\solu) = 0$ is a minimizer, no matter how distant it is from $\spann\{\grtr\}$. See also Subsection~\ref{subsec:results:generalset:geometry} for a geometric interpretation of this undesirable parameter configuration.
In conclusion, there must be a certain relationship between the values of $\scalfac$ and $m$ in order to turn our recovery task into a well-defined problem.
Such a condition is actually the key aspect of our next main result.
\begin{theorem}[Signal Recovery in Convex Sets]\label{thm:results:generalset:global}
	Let the model conditions of Assumption~\ref{model:results:generalset} be satisfied.
	For every fixed $\scalfac > 0$ and $\eta \in \intvop{0}{\tfrac{1}{2}}$, the following holds true with probability at least $1 - \eta$:
	If $\scalfac \gtrsim 1$ and the number of samples obeys
	\begin{equation}\label{eq:results:generalset:global:meas}
		m \gtrsim \scalfac^{4} \cdot \max\{\effdim[\conic]{\sset - \grtr}, \log(\eta^{-1})\},
	\end{equation}
	then any minimizer $\solu$ of \eqref{eq:estimatortuned} satisfies
	\begin{equation}\label{eq:results:generalset:global:bound}
		\lnorm[auto]{\grtr - \frac{\solu}{\lnorm{\solu}}} \lesssim \frac{1}{\scalfac}\ .
	\end{equation}	
	The same assertion holds true if \eqref{eq:results:generalset:global:meas} is replaced by
	\begin{equation}\label{eq:results:generalset:global:meas:global}
		m \gtrsim \scalfac^{4} \cdot \max\{\meanwidth{\sset}^2, \log(\eta^{-1})\}.
	\end{equation}
\end{theorem}
The proof of Theorem~\ref{thm:results:generalset:global} is postponed to Subsection~\ref{subsec:proofs:generalset}.
It relies on similar statistical tools as the proof of Theorem~\ref{thm:euclideanball}, but the actual arguments are quite different.
In fact, the role of the underlying empirical processes may change significantly if $\sset$ is not contained in the unit ball anymore.

This also indicates why the above recovery statement is conceptually somewhat different from Theorem~\ref{thm:euclideanball}: The hypothesis of  Theorem~\ref{thm:results:generalset:global} merely relies on the ``coarser'' complexity measure of conic efficient dimension (or Gaussian width), while the scaling parameter $\scalfac^{-1}$ now mimics the role of the scale $t$ of the local effective dimension in Theorem~\ref{thm:euclideanball}.
Indeed, $\scalfac$ can be regarded as an oversampling factor that controls the recovery accuracy, even though it also affects the size of the constraint set in \eqref{eq:estimatortuned}. Adjusting $\scalfac$ such that \eqref{eq:results:generalset:global:meas} holds true with equality (up to numerical constants), the error bound of \eqref{eq:results:generalset:global:bound} can be rewritten in terms of the sample count $m$:
\begin{equation}\label{eq:results:generalset:global:boundsample} 
	\lnorm[auto]{\grtr - \frac{\solu}{\lnorm{\solu}}} \lesssim \Big(\frac{\max\{\effdim[\conic]{\sset - \grtr}, \log(\eta^{-1}) \}}{m}\Big)^{1/4}.
\end{equation}
Thus, by appropriately upscaling the signal set $\sset$, the normalized minimizer $\solu / \lnorm{\solu}$ turns into a consistent estimator of $\grtr$ with an approximation rate of $\asympfaster{m^{-1/4}}$.

We wish to emphasize that using the conic effective dimension as complexity parameter is not as problematic as in the setting of Subsection~\ref{subsec:results:unitball}, where the expected risk minimizer $\grtrmu$ could lie in the interior of $\sset$.
Due to the flexible scaling parameter of \eqref{eq:estimatortuned}, one can basically assume that $\grtr$ lies on the boundary of $\sset$: 
If $\grtr$ would belong to the interior of $\sset$, there exists a factor $\nu \in \intvop{0}{1}$ such that $\grtr \in \boundary{(\nu\sset)}$. Then the constraint of \eqref{eq:estimatortuned} takes the form
\begin{equation}
	\x \in \scalfac \sset = \underbrace{\tfrac{\scalfac}{\nu}}_{= \scalfac'} \cdot \underbrace{\nu\sset}_{= \sset'} = \scalfac' \sset' \ ,
\end{equation}
and we may apply Theorem~\ref{thm:results:generalset:global} with $\scalfac'$ and $\sset'$ instead of $\scalfac$ and $\sset$, respectively. 
In that way, the error accuracy is just affected by an additional (unknown) scalar factor $\nu$. This particularly explains why the illustrations of Figure~\ref{fig:results:generalset:exprisk} and Figure~\ref{fig:results:generalset:cylinder} assume that $\grtr$ lies on the boundary of $\sset$.

However, there are still cases in which the conic effective dimension is badly behaved, e.g., if $\grtr$ is just compressible and not exactly sparse. Then we may replace the condition  \eqref{eq:results:generalset:global:meas} by \eqref{eq:results:generalset:global:meas:global} to use the Gaussian width as complexity measure, which does only take global features of $\sset$ into account.
For more details, see the discussion subsequent to Corollary~\ref{coro:euclideanball} as well as Subsection~\ref{subsec:appl:signalsets}.


\begin{remark}
\begin{rmklist}
\item\label{rmk:results:generalset:kvl1svm}
	Theorem~\ref{thm:results:generalset:global} with \eqref{eq:results:generalset:global:meas:global} contains a recent result from Kolleck and Vyb\'{\i}ral \cite{kolleck2015l1svm} as special case. Their work was the first one that theoretically studied the recovery performance of hinge loss minimization under $1$-bit measurements. These results however do only focus on a global analysis of $\l{1}$-constraints, adapting concentration inequalities from Plan and Vershynin in \cite{plan2013robust}.
	In contrast, our approach is rather local and relies on refined bounds for empirical processes.
	This eventually allows us to extend the restrictive model setup of \cite{kolleck2015l1svm} into various directions and to improve their approximation rates (with respect to $m$), as we will see below in Theorem~\ref{thm:results:generalset:local}.
\item\label{rmk:results:generalset:variancebias}
	\emph{The variance-bias problem.} This issue is well-known from statistical learning theory.
	In general, the variance-bias problem states that there is a fundamental trade-off between the sample error (variance) and the approximation error (bias) in empirical risk minimization if the size of the hypothesis set is varied (and the sample count $m$ remains fixed); see also \cite[Sec.~1.5]{cucker2007learning}.
	Interestingly, the above estimation strategy fits very well into this situation:
	The ``hypothesis set'' $\sset$ in \eqref{eq:estimatortuned} is enlarged by increasing the scaling parameter $\scalfac$, and \eqref{eq:asympexprisk:bound} indicates that an expected risk minimizer $\x^\ast$ aligns better with $\spann\{\grtr\}$ as $\scalfac$ grows.
	With other words, the approximation error gets smaller. But on the other hand, an empirical risk minimizer $\solu$ of \eqref{eq:estimatortuned} might be still a poor approximation of $\x^\ast$ if $\scalfac$ is too large, meaning that the sample error is large.
	This issue of \emph{overfitting} can be fortunately resolved by increasing the sampling rate according to \eqref{eq:results:generalset:global:meas}.
	In this light, Theorem~\ref{thm:results:generalset:global} just claims that the variance-bias problem can be handled by carefully balancing the scaling parameter $\scalfac$ and the sample size $m$. 
	\qedhere
\end{rmklist} \label{rmk:results:generalset}
\end{remark}

\subsubsection{Localized Approximation Rates}

A downside of Theorem~\ref{thm:results:generalset:global} is the relatively slow error decay of $\asympfaster{m^{-1/4}}$ in \eqref{eq:results:generalset:global:boundsample}.
In fact, this approximation rate is obviously much worse than the decay of $\asympfaster{m^{-1/2}}$ achieved by Theorem~\ref{thm:euclideanball} in \eqref{eq:effdimbound}, which in turn is essentially optimal (see \cite[Sec.~4]{plan2014highdim}).
Our third main result, Theorem~\ref{thm:results:generalset:local}, shows that the factor of $\scalfac^4$ in condition \eqref{eq:results:generalset:global:meas} can be replaced by $\scalfac^2 \cdot t_0^2$, where $t_0$ is an additional geometric parameter that depends on $\grtr$ and $\sset$.
Formally, it is defined by
\begin{equation}\label{eq:results:generalset:t0}
	t_0 \coloneqq \max\{1 , \rad((\boundary(\ssetmu) \intersec \cyl(\grtr,\scalfac)) - \grtrmu)\}
\end{equation}
where $\cyl(\grtr,\scalfac)$ denotes the following cylindrical tube around $\spann\{\grtr\}$:
\begin{equation}\label{eq:results:generalset:cylinder}
	\cyl(\grtr, \scalfac) \coloneqq \{ \x \in \R^n \suchthat \lnorm{\x - \sp{\x}{\grtr}\grtr} \leq 1, \sp{\x}{\grtr} \geq \tfrac{\scalfac}{2} \},
\end{equation}
see Figure~\ref{fig:results:generalset:cylinder} for a visualization.
\begin{figure}
	\centering
	\tikzstyle{blackdot}=[shape=circle,fill=black,minimum size=1mm,inner sep=0pt,outer sep=0pt]
		\begin{tikzpicture}[scale=2,extended line/.style={shorten >=-#1,shorten <=-#1},extended line/.default=1cm]]
			\coordinate (K1) at (0,0);
			\coordinate (K2) at (.85,-1.13333);
			\coordinate (K3) at (.5,-1.75);
			\coordinate (K4) at (-.5,-1.5);
			\coordinate (K5) at (-.5,-.75);
			\coordinate (P0) at (0,-1);
			\path[name path=K1--K2] (K1) -- (K2);
			\path[name path=P0--X0] (P0) -- ++(65:1);
			\path[name intersections={of=P0--X0 and K1--K2,by=X0}];
			\coordinate (X0muhalf) at ($(X0)!.5!(P0)$);
			\coordinate (SpanEnd) at ($(X0)!-1!(P0)$);
			
			\path (X0muhalf) -- ++(-25:.2) coordinate (CylR);
			\coordinate (CylL) at ($(X0muhalf)!-1!(CylR)$);
			\coordinate (CylLtop) at ($(CylL)!(SpanEnd)!90:(CylR)$);
			\coordinate (CylRtop) at ($(CylR)!(SpanEnd)!90:(CylL)$);
						
			\draw[fill=gray!20!white,thick] (K1) -- (K2) -- (K3) -- (K4) -- (K5) -- cycle;
			\draw[fill=gray!20!white] (CylLtop) -- (CylL) -- (CylR) -- (CylRtop);
			\begin{scope}
				\clip (K1) -- (K2) -- (K3) -- (K4) -- (K5) -- cycle;
				\draw[fill=gray!60!white] (CylLtop) -- (CylL) -- (CylR) -- (CylRtop);
			\end{scope}
			\draw[thick] (K1) -- (K2) -- (K3) -- (K4) -- (K5) -- cycle;
			\begin{scope}
				\clip (CylLtop) -- (CylL) -- (CylR) -- (CylRtop) -- cycle;
				\draw[red,ultra thick] (K1) -- (K2) -- (K3) -- (K4) -- (K5) -- cycle;
			\end{scope}
			\draw[<->,>=stealth] ($(CylL)!1.05!(CylLtop)$) -- ($(CylR)!1.05!(CylRtop)$) node[pos=.6,above] {\smaller{$2$}};
			
			\draw[dashed] (SpanEnd) -- ($(X0)!2.75!(P0)$);
			\node at (barycentric cs:K1=1,K2=1,K3=8,K4=1,K5=1) {$\scalfac\sset$};
			\node[blackdot, label=0:$\vnull$] at (P0) {};
			\node[blackdot, label=right:$\grtrmu$] at (X0) {};
			\node[label={[label distance=-3pt]left:$\spann\{\grtr\}$}] at ($(X0)!2.35!(P0)$) {};
			\node[blackdot,label={[xshift=2pt,yshift=-2pt]180:\smaller{$\tfrac{\scalfac}{2}\grtr$}}] at (X0muhalf) {};
			\node[below right=.2 and .0 of CylRtop] {$\cyl(\grtr,\scalfac)$};
		\end{tikzpicture}
	\caption{Illustration of the cylindrical tube $\cyl(\grtr, \scalfac)$: It starts from the base of diameter $2$ located at $\tfrac{\scalfac}{2}\grtr$ and then stretches beyond infinity along $\spann\{\grtr\}$. The parameter $t_0$ measures the radial extent of the red intersection (around $\grtrmu$). The dark gray region corresponds to the cylindrical intersection $\ssetmu \intersec \cyl(\grtr,\scalfac)$.}
	\label{fig:results:generalset:cylinder}
\end{figure}
We wish to emphasize that the intersection $\boundary(\ssetmu) \intersec \cyl(\grtr,\scalfac)$ can be significantly smaller than $\ssetmu \intersec \cyl(\grtr,\scalfac)$, whose radius around $\grtrmu$ is basically of order $\scalfac$.
Intuitively, if $\scalfac$ is sufficiently large, the boundary $\boundary(\ssetmu)$ does only intersect with the side of $\cyl(\grtr,\scalfac)$, which has diameter $2$. The value of $t_0$ then becomes (almost) independent of $\scalfac$ and is determined by the ``local'' geometry of $\sset$ in a neighborhood of $\grtr$ (of size $1/\scalfac$).
In this situation, the sampling rate would only scale quadratically in $\scalfac$, leading to the desired rate of $\asympfaster{m^{-1/2}}$. 

Before we further discuss the impact of $t_0$, let us state the actual recovery guarantee:
\begin{theorem}[Signal Recovery in Convex Sets -- Local Version]\label{thm:results:generalset:local}
	Let the model conditions of Assumption~\ref{model:results:generalset} be satisfied and let $t_0$ be defined according to \eqref{eq:results:generalset:t0}.
	For every fixed $\scalfac > 0$ and $\eta \in \intvop{0}{\tfrac{1}{2}}$, the following holds true with probability at least $1 - \eta$:
	If $\scalfac \gtrsim t_0$ and the number of samples obeys
	\begin{equation}\label{eq:results:generalset:local:meas}
		m \gtrsim \scalfac^{2} \cdot t_0^2 \cdot \max\{\effdim[\conic]{\sset - \grtr}, \log(\eta^{-1})\},
	\end{equation}
	then any minimizer $\solu$ of \eqref{eq:estimatortuned} satisfies
	\begin{equation}\label{eq:results:generalset:local:bound}
		\lnorm[auto]{\grtr - \frac{\solu}{\lnorm{\solu}}} \lesssim \frac{1}{\scalfac} \ .
	\end{equation}	
\end{theorem}
As before, we can adjust $\scalfac$ in \eqref{eq:results:generalset:local:meas} to obtain a convenient error bound depending on $m$:
\begin{equation}\label{eq:results:generalset:local:boundsample}
	\lnorm[auto]{\grtr - \frac{\solu}{\lnorm{\solu}}} \lesssim \Big(\frac{t_0^2 \cdot \max\{\effdim[\conic]{\sset - \grtr}, \log(\eta^{-1}) \}}{m}\Big)^{1/2}.
\end{equation}
While this oversampling rate certainly resembles the unit-ball case in \eqref{eq:conicbound}, let us point out an important difference:
Roughly speaking, the geometric conclusion of Theorem~\ref{thm:euclideanball} is that every minimizer of \eqref{eq:estimator} must lie within the spherical intersection $\sset \intersec (t\ball[2][n] + \grtrmu)$; see Figure~\ref{fig:proofs:fact}. In contrast, the proof of Theorem~\ref{thm:results:generalset:local} argues that every minimizer of \eqref{eq:estimatortuned} must belong to the cylindrical intersection $\ssetmu \intersec \cyl(\grtr,\scalfac)$; see Figure~\ref{fig:results:generalset:cylinder}.
By rescaling the latter set by a factor of $1/ \scalfac$ with $\scalfac = 1/t \gg 1$, it is not hard to see that the shapes of both intersections are fundamentally different.
More specifically, due to the anisotropy of $\cyl(\grtr,\scalfac)$, the conic effective dimension $\effdim[\conic]{\sset - \grtr}$ needs to be scaled by an extra factor of $t_0^2$ in \eqref{eq:results:generalset:local:boundsample}.

Without any further assumptions on the geometric arrangement of $\sset$ and $\grtr$, it is difficult to make precise statements about the order of $t_0$.
For example, if the boundary of $\sset$ is almost orthogonal to $\spann\{\grtr\}$ in a small neighborhood of $\grtr$, we can expect that $t_0 \approx 1$.
But as $\boundary{\sset}$ gets more ``tangent'' to $\spann\{\grtr\}$, as in Figure~\ref{fig:results:generalset:cylinder}, $t_0$ may become significantly larger.
If $t_0 \approx \scalfac$, the condition of \eqref{eq:results:generalset:local:meas} particularly degenerates to \eqref{eq:results:generalset:global:meas}. Moreover, the situation of $t_0 \gg \scalfac$ is excluded by Theorem~\ref{thm:results:generalset:local} in any case, whereas Theorem~\ref{thm:results:generalset:global} would still provide meaningful error bounds.
However, as long as $t_0$ is considered as a (possibly large) signal-dependent parameter, we can always achieve the optimal non-asymptotic rate of $\asympfaster{m^{-1/2}}$.
For a numerical simulation related to the impact of $t_0$ on the recovery performance of \eqref{eq:estimatortuned}, see Experiment~\hyperref[para:appl:numerics:exp3]{3} in Subsection~\ref{subsec:appl:numerics} below.

\subsubsection{Geometric Interpretation and Classification Margins}
\label{subsec:results:generalset:geometry}

We conclude this subsection with a brief discussion on the geometric aspects of the above recovery problem.
To this end, let us consider the following interpretation of the perfect $1$-bit model from Assumption~\ref{model:results:generalset}: The measurement vectors $\pcloud \coloneqq \{\a_1, \dots, \a_m\} \subset \R^n$ form a ``cloud'' of random points that is generated by the standard Gaussian distribution in $\R^n$. According to the observation rule \eqref{eq:results:generalset:meas}, each point is then endowed with a binary label depending on which side of the hyperplane $H(\grtr) \coloneqq \orthcompl{\{\grtr\}}$ the respective point resides. With other words, $H(\grtr)$ is a \emph{separating hyperplane} that divides $\pcloud$ into two classes, say $\pcloud^+$ and $\pcloud^-$; see also Figure~\ref{fig:results:generalset:geometry:sephyp:msmall}.

The actual recovery task can be now translated into finding the hyperplane $H(\grtr)$ --- or equivalently, its normal vector $\grtr$ --- where only the labeled point clouds $\pcloud^+$ and $\pcloud^-$ are available. It should be emphasized that this problem is somewhat different from traditional classification, where one is already satisfied with \emph{any} separating hyperplane $H(\x)$. If $m$ is sufficiently large, such as in Figure~\ref{fig:results:generalset:geometry:sephyp:mlarge}, one can expect that $\x$ and $\grtr$ are however close, in the sense that
\begin{equation}\label{eq:results:generalset:geometry:error}
	\lnorm[auto]{\frac{\x}{\lnorm{\x}} - \grtr} = \sqrt{2 - 2 \sp{\tfrac{\x}{\lnorm{\x}}}{\grtr}}
\end{equation}
is small.
From a statistical perspective, this conclusion essentially follows from the law of large numbers, which implies that the margin between $\pcloud^+$ and $\pcloud^-$ ``contracts'' to $H(\grtr)$ as $m \to \infty$.
On the other hand, if $m$ is too small, the margin between both classes is typically large, so that there are many separating hyperplanes and the error in \eqref{eq:results:generalset:geometry:error} is not necessarily small.

\begin{figure}[!t]
	\centering
	\begin{subfigure}[t]{0.3\textwidth}
		\centering
		\tikzstyle{blackdot}=[shape=circle,fill=black,minimum size=1mm,inner sep=0pt,outer sep=0pt]
		\begin{tikzpicture}[scale=2]
			\node[inner sep=0pt] (pic) at (0,0) {\includegraphics[width=\textwidth]{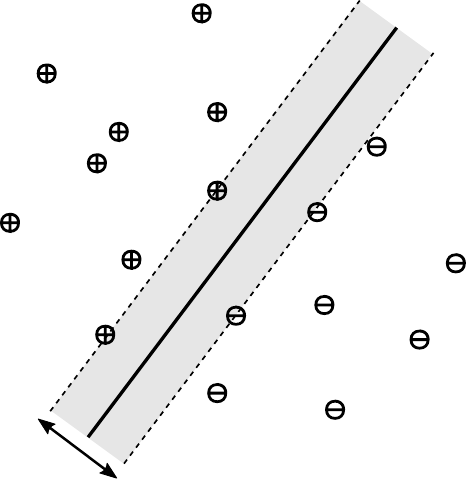}};
			\node[above right=-.35 and -0.9 of pic] {$H(\grtr)$};
		\end{tikzpicture}
		\caption{}
		\label{fig:results:generalset:geometry:sephyp:msmall}
	\end{subfigure}%
	\qquad\qquad\qquad\qquad
	\begin{subfigure}[t]{0.3\textwidth}
		\centering
		\begin{tikzpicture}[scale=2]
			\node[inner sep=0pt] (pic) at (0,0) {\includegraphics[width=\textwidth]{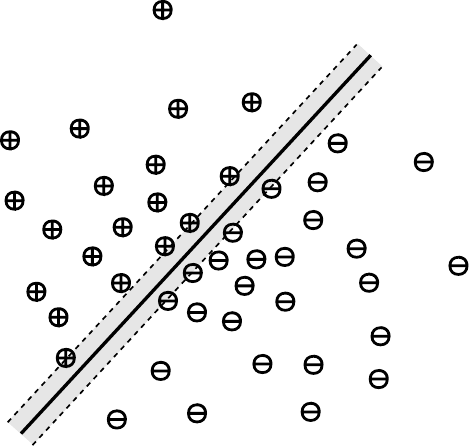}};
			\node[above right=-.65 and -1.1 of pic] {$H(\grtr)$};
		\end{tikzpicture}
		\caption{}
		\label{fig:results:generalset:geometry:sephyp:mlarge}
	\end{subfigure}%
	\caption{Visualization of the labeled data $\pcloud = \pcloud^+ \union \pcloud^-$ generated by \protect{\eqref{eq:results:generalset:meas}}. The width of the gray region corresponds to the classification margin between $\pcloud^+$ and $\pcloud^-$ with respect to the separating hyperplane $H(\grtr) = \orthcompl{\{\grtr\}}$. \subref{fig:results:generalset:geometry:sephyp:msmall}~The sample count $m$ is rather small, implying that the margin is quite large. \subref{fig:results:generalset:geometry:sephyp:mlarge}~The margin becomes increasingly smaller as $m$ grows.}
	\label{fig:results:generalset:geometry:sephyp}
\end{figure}

The major concern of this work is to face this challenge by minimizing the empirical hinge loss $\lossemp(\cdot)$ on a certain signal set (see Definition~\ref{def:results:unitball:risk}). It is well-known from the literature that the associated estimator \eqref{eq:estimator} can be identified with a support vector machine (SVM), whose actual purpose is to maximize the classification margin between labeled data points,\footnote{In general, there does not necessarily exist a separating hyperplane. For this reason, one rather speaks of \emph{soft margin classification} where a certain amount of misclassified points is permitted. The purpose of SVMs is therefore to find a hyperplane that maximizes the soft margin. This perspective is also taken when dealing with noisy $1$-bit measurements in Subsection~\ref{subsec:results:unitball}.} see \cite[Sec.~II]{kolleck2015l1svm} for example.
The scenario of Figure~\ref{fig:results:generalset:geometry:sephyp:mlarge} is therefore somewhat undesirable because the margin between $\pcloud^+$ and $\pcloud^-$ is very narrow.
Consequently, hinge loss minimization via \eqref{eq:estimator} can be inappropriate if $m$ is too large, even though all separating hyperplanes do almost align with $H(\grtr)$. This precisely corresponds to the observations made at the beginning of Subsection~\ref{subsec:results:generalset:scalable}, where we investigated the expected risk minimizer. Indeed, if $\sset$ does not belong to the unit ball, a vector $\x^\ast$ could minimize $\lossexp(\cdot)$ on $\sset$ but still induce a hyperplane $H(\x^\ast)$ that is not separating as $m \to \infty$.

Interestingly, we were able to resolve this issue by introducing a scalable hinge loss estimator \eqref{eq:estimatortuned} in Definition~\ref{def:estimatortuned}.
Due to the identity
\begin{equation}\label{eq:results:generalset:geometry:hingeloss}
	\min_{\x \in \scalfac\sset} \tfrac{1}{m} \sum_{i = 1}^m \losshng(y_i \sp{\a_i}{\x}) = \min_{\x \in \sset} \tfrac{1}{m} \sum_{i = 1}^m \losshng(y_i \sp{\scalfac\a_i}{\x}),
\end{equation}
solving \eqref{eq:estimatortuned} is equivalent to solving \eqref{eq:estimator} with measurement vectors upscaled by a factor of $\scalfac$.
In the situation of Figure~\ref{fig:results:generalset:geometry:sephyp:mlarge}, this implies that  the dense point cloud and the separating margin are both enlarged. The geometric arrangement of the transformed data then resembles Figure~\ref{fig:results:generalset:geometry:sephyp:msmall} and hinge loss minimization becomes feasible.
Our main results, Theorem~\ref{thm:results:generalset:global} and Theorem~\ref{thm:results:generalset:local}, confirm this heuristic reasoning and particularly show how $m$ and $\scalfac$ need to scale (non-asymptotically) in order to obtain a consistent estimator of $\grtr$.

Finally, it is worth mentioning that the geometric viewpoint of SVMs was also the initial motivation of \cite{kolleck2015l1svm} to study the performance of the scalable hinge loss estimator as it is stated on the right-hand side of \eqref{eq:results:generalset:geometry:hingeloss}.
While the upscaling strategy indeed allows us to mimic a well-posed classification problem for SVMs, this perspective nevertheless appears a bit artificial in the light of $1$-bit observation models.
In real-world applications, the data $\pcloud$ is very unlikely to follow a standard Gaussian distribution, but rather exhibits anisotropic features, as it is the case for Gaussian mixture models. For that reason, the focus of this work is mainly on $1$-bit compressed sensing, where Gaussian measurements often serve as a benchmark to analyze recovery algorithms.

\section{Examples and Applications}
\label{sec:appl}

In this section, we illustrate the general framework of Section~\ref{sec:results} by a few applications and examples. 
Subsection~\ref{subsec:appl:signalsets} continues our discussion on signal sets from Subsection~\ref{subsec:results:modelsetup} and conveys more intuition regarding how their complexity is measured by means of the Gaussian width and the effective dimension. 
A particular emphasis is on the case of sparse vectors and related convex relaxations, which will also be applied to our main results.
Subsection~\ref{subsec:appl:examples} then focuses on two specific noisy $1$-bit observation models in the unit-ball setup of Subsection~\ref{subsec:results:unitball}, namely random bit flips and additive Gaussian noise. In this context, we will demonstrate that hinge loss minimization even yields a consistent estimator if the signal-to-noise ratio is very low.
Our theoretical considerations are finally complemented by some numerical experiments in Subsection~\ref{subsec:appl:numerics}. Apart from investigating the aforementioned noise models (see Experiment~\hyperref[para:appl:numerics:exp2]{2}), we do also evaluate the empirical performance of the two hinge loss programs from Subsection~\ref{subsec:results:unitball} and Subsection~\ref{subsec:results:generalset} in the situation of perfect $1$-bit observations.

\subsection{The Gaussian Width and Sparse Recovery}
\label{subsec:appl:signalsets}

Let us begin with a few examples of structured signal sets $\sset$ that are widely used in signal estimation theory and satisfy Assumption~\ref{model:signal}:
\begin{example}
\begin{rmklist}
\item
	\label{rmk:signalset:effsparse}\emph{(Effectively) Sparse signals.} 
	The prototypical example of low-dimensional structures studied in compressed sensing is \emph{sparsity}. We denote the set of all $s$-sparse vectors on the Euclidean unit sphere by
	\begin{equation}
		\Sigma^n_s \coloneqq \{\x\in \R^n \suchthat \lnorm{\x}[0]\leq s, \lnorm{\x}= 1\}.
	\end{equation}
	Since $\Sigma^n_s$ is obviously non-convex, it falls out of the scope of our main results. 
	However, by the Cauchy-Schwarz inequality, each vector $\grtr \in \Sigma^n_s$ satisfies
	\begin{equation}\label{eq:appl:signalsets:cauchyschwarz}
		\lnorm{\grtr}[1]\leq \sqrt{\lnorm{\grtr}[0]} \cdot \lnorm{\grtr} \leq \sqrt{s} \cdot \lnorm{\grtr}=\sqrt{s},
	\end{equation}
	which implies that $\sset_{n,s}\coloneqq\sqrt{s}\ball[1][n]\intersec \ball[2][n]$ is a convex signal set that contains all normalized $s$-sparse vectors. Interestingly, this set essentially equals the convex hull of $\Sigma^n_s$ in the sense that (cf. \cite[Lem.~3.1]{plan2013onebit}) 
	\begin{equation}\label{eq:approximatelysparse}
		\convhull(\Sigma^n_s)\subset \sset_{n,s}\subset 2\convhull(\Sigma^n_s).
	\end{equation}
	Consequently, $\sset_{n,s}$ forms a natural convex relaxation of $\Sigma^n_s$.
	Since $\sset_{n,s}$ particularly contains compressible vectors that are almost $s$-sparse, it is also referred to as the set of \emph{effectively $s$-sparse vectors}.
\item
	\label{rmk:signalset:l1ball}\emph{Scaled $\l{1}$-balls.} Another straightforward convex relaxation of sparsity are scaled $\l{1}$-balls, i.e., $\sset=R\ball[1][n]$ with an appropriately chosen scaling parameter $R>0$. Indeed, according to \eqref{eq:appl:signalsets:cauchyschwarz}, we can immediately conclude that $\sset=\sqrt{s}\ball[1][n]$ is a superset of $\Sigma^n_s$. Such types of $\l{1}$-constraints became very popular in practice because they often allow for sparse approximation and linear programming at the same time. Unfortunately, this example does not meet the unit-ball assumption of Theorem~\ref{thm:euclideanball}, so that it is only admissible for Theorem~\ref{thm:results:generalset:global} and Theorem~\ref{thm:results:generalset:local}.
\item
	\label{rmk:signalset:subspace}\emph{Subspaces.} Perhaps the simplest example of a structured signal set is a linear subspace.
	If $\grtr \in \S^{n-1}$ belongs to a known subspace $E\subset \R^n$ of low dimension, one may just consider the signal set $\sset=E\intersec \ball[2][n]$ as prior. 
\item
	\label{rmk:signalset:polytope}\emph{Polytopes.} If $\grtr$ is known to be a convex combination of finitely many vectors $\{\x_1,\dots,\x_k\}\subset \R^n$, one could simply choose the polytope $\sset=\convhull\{\vnull,\x_1,\dots,\x_k\}$ as signal set. \qedhere
\end{rmklist}
\label{ex:signalset}
\end{example}

Next, we collect some basic yet important properties of the Gaussian width and effective dimension. The proofs are omitted, since most statements are direct consequences of the definition. For a more extensive discussion, see also \cite{plan2013robust,vershynin2014estimation,RomanHDP}.
\begin{proposition}\label{prop:gaussianwidth} The Gaussian width $\meanwidth{\sset}$ of a bounded subset $\sset\subset\R^n$ (see Definition~\ref{def:gaussianwidth}) and the effective dimension
\begin{equation}
\effdim{\sset}=\frac{\meanwidth{\sset}^2}{\diam(\sset)^2} \ ,
\end{equation}
satisfy the following properties:
\begin{thmproperties}
\item\label{prop:property:monotonicity} If $\sset\subset \sset' \subset \R^n$, it follows that $\meanwidth{\sset}\leq \meanwidth{\sset'}$.
\item\label{prop:property:difference} $\meanwidth{\sset}=\tfrac{1}{2}\meanwidth{\sset-\sset}$.
\item\label{prop:property:shift} $\meanwidth{\sset + \x}=\meanwidth{\sset}$ for every $\x \in \R^n$.
\item\label{prop:property:convexhull} $\meanwidth{\sset}=\meanwidth{\convhull\sset}$.
\item\label{prop:property:algdim} $\effdim{\sset}\leq \dim(\spann\sset)$. 
\item\label{prop:property:euclideanball} If $\sset$ is the Euclidean ball of a $d$-dimensional subspace of $\R^n$, then $\effdim{\sset}\asymp d$.
\item\label{prop:property:finite} If $\sset$ is finite, we have $\effdim{\sset}\lesssim \log(\abs{\sset})$.
\end{thmproperties}
\end{proposition}


Let us now focus on the important case of sparse recovery. More precisely, we intend to apply our main results from Section~\ref{sec:results} to the set of effectively $s$-sparse vectors as defined in Example~\ref{ex:signalset}\ref{rmk:signalset:effsparse} as well as to scaled $\l{1}$-balls as considered in Example~\ref{ex:signalset}\ref{rmk:signalset:l1ball}. 
In either of these cases, there exist (sharp) bounds on the Gaussian width in literature (see \cite[Sec.~2 and~3]{plan2013robust}):
\begin{align}
\meanwidth{\sqrt{s}\ball[1][n]} &\lesssim \sqrt{s\log(n)}, \label{eq:gaussianwidthscaled} \\ 
\meanwidth{\sqrt{s}\ball[1][n]\intersec \ball[2][n]} &\lesssim \sqrt{s\log(\tfrac{2n}{s})}. \label{eq:gaussianwidtheffsparse}
\end{align}
Using \eqref{eq:gaussianwidthscaled}, we now obtain the following sparse recovery result from Theorem~\ref{thm:results:generalset:global}:
Set $K=\sqrt{s}\ball[1][n]$ and assume that $\grtr \in \sqrt{s}\ball[1][n] \intersec \S^{n-1}$. Provided that $m\gtrsim \scalfac^4 \cdot s \log(n)$, any minimizer $\solu$ of \eqref{eq:estimatortuned} satisfies with high probability
\begin{equation}
\lnorm[auto]{\grtr - \frac{\solu}{\lnorm{\solu}}}\lesssim \frac{1}{\scalfac} \ . 
\end{equation}

Combined with \eqref{eq:gaussianwidtheffsparse}, Corollary~\ref{coro:euclideanball} yields a similar statement: Set $K=\sqrt{s}\ball[1][n] \intersec \ball[2][n]$ and assume that $\grtr \in \sqrt{s}\ball[1][n] \intersec \S^{n-1}$. If $m\gtrsim t^{-4} \cdot s  \log(2n/s)$,
then any minimizer $\solu$ of \eqref{eq:estimator} satisfies with high probability
\begin{equation}
\lnorm[auto]{\grtr - \frac{\solu}{\lnorm{\solu}}}\lesssim t. 
\end{equation}

We would like to point out that these assertions resemble the findings of \cite[Thm.~II.3, Thm.~IV.1]{kolleck2015l1svm}.
While the number of required measurements is essentially optimal with respect to the signal complexity in both cases (cf. \cite[Chap.~11]{foucart2013cs}), the dependence on the respective oversampling factor is clearly sub-optimal.
Indeed, the fourth power of $\scalfac^4$ and $t^{-4}$ leads to an error decay rate of $\asympfaster{m^{-1/4}}$.
In order to achieve the optimal oversampling rate of $\asympfaster{m^{-1/2}}$, we need to investigate the conic effective dimension of sparse vectors. 
Invoking a well-known bound on the conic effective dimension from \cite[Prop.~3.10]{chandrasekaran2012geometry}, it turns out that an $s$-sparse vector $\grtr$ that lies on the boundary of an $\l{1}$-ball satisfies
\begin{equation}\label{eq:coniceffdimssparse}
\effdim[\conic]{\lnorm{\grtr}[1]\ball[1][n] - \grtr}\lesssim s\log(\tfrac{2n}{s}).
\end{equation}

Together with Theorem~\ref{thm:results:generalset:local} (and \eqref{eq:results:generalset:local:boundsample}), this gives the following:
Set $\sset=R\ball[1][n]$ and assume that $\grtr$ is $s$-sparse with $\grtr \in \boundary{\sset} \intersec \S^{n-1}$. Then any minimizer $\solu$ of \eqref{eq:estimatortuned} satisfies with high probability at least $1- \eta$
\begin{equation}\label{eq:generalsparse}
	\lnorm[auto]{\grtr - \frac{\solu}{\lnorm{\solu}}} \lesssim \Big(\frac{t_0^2 \cdot \max\{ s\log(\tfrac{2n}{s}), \log(\eta^{-1}) \}}{m}\Big)^{1/2},
\end{equation}
where $t_0$ is the signal-dependent parameter defined in \eqref{eq:results:generalset:t0}. An analogous result follows from Corollary~\ref{coro:euclideanball} (and \eqref{eq:conicbound}) in the unit-ball case by considering the signal set $\sset = \lnorm{\grtrmu}[1]\ball[1][n] \intersec \ball[2][n]$.
In this situation, one can even remove the factor $t_0$ in \eqref{eq:generalsparse}, meaning that any solution $\solu$ to \eqref{eq:estimator} obeys
\begin{equation}\label{eq:appl:signalsets:generalsparseunitball}
\lnorm[auto]{\grtr - \frac{\solu}{\lnorm{\solu}}} \lesssim \Big(\frac{\max\{s\log(\tfrac{2n}{s}), \log(\eta^{-1}) \}}{m}\Big)^{1/2}.
\end{equation}

While the error bound of \eqref{eq:appl:signalsets:generalsparseunitball} is essentially optimal (see \cite[Sec.~4]{plan2014highdim}), it is however only of limited practical interest because the signal set $\sset = \lnorm{\grtrmu}[1]\ball[1][n] \intersec \ball[2][n]$ depends on the unknown vector $\grtr$.
Indeed, if we would just select $\sset = R\ball[1][n] \intersec \ball[2][n]$ for some $R > \lnorm{\grtrmu}[1]$, the conic effective dimension would drastically increase to $\effdim[\conic]{( R\ball[1][n] \intersec \ball[2][n] ) - \grtrmu} \asymp n$, which in turn leads to a very pessimistic sampling rate.
A similar problem would occur if $\grtr$ is just compressible rather than exactly $s$-sparse. In this situation, $\grtr$ may reside in a higher dimensional face of a scaled $\l{1}$-ball but is still relatively close to a low-dimensional face.

Fortunately, the refined concept of local effective dimension is able to resolve these issues and allows for \emph{stable recovery}.
Roughly speaking, the idea is as follows: If the ``anchor vector'' $\grtrmu \in K = \sqrt{s}\ball[1][n] \intersec \ball[2][n]$ in Theorem~\ref{thm:euclideanball} is not exactly $s$-sparse, select a nearby $s$-sparse vector $\x^{\ast} \in \boundary{\sset}$ and set $t^{\ast} \coloneqq \lnorm{\grtrmu - \x^{\ast}}$. Then, the local effective dimension $\effdim[t]{\sset - \grtrmu}$ at any scale $t \geq t^{\ast}$ behaves as if we would consider its conic counterpart in $\x^{\ast}$, i.e., $\effdim[\conic]{\sset - \x^{\ast}}$.
This geometric argument is formalized by the following proposition, which is a consequence of \cite[Lem.~A.2]{genzel2017cosparsity}:
\begin{proposition}\label{prop:effdimeffsparse}
	Let $\grtrmu \in  \sset \coloneqq \sqrt{s}\ball[1][n]\intersec \ball[2][n]$ and set $\tilde{\sset}_{n,s} \coloneqq \{\x \in \sset \suchthat \lnorm{\x}[0] \leq s, \,\lnorm{\x}[1]=\sqrt{s}\}$. Moreover, let $\x^{\ast}\in \R^n$ be a best $s$-term approximation of $\grtrmu$ in $\tilde{\sset}_{n,s}$, i.e., $\x^{\ast} = \argmin_{\x\in \tilde{\sset}_{n,s}}\lnorm{\grtrmu-\x}$. Then for every $t\geq t^{\ast}\coloneqq\lnorm{\grtrmu-\x^{\ast}}$, we have that
	\begin{equation}
		\effdim[t]{(\sqrt{s}\ball[1][n]\intersec \ball[2][n] ) - \grtrmu}\lesssim \effdim[\conic]{\sqrt{s}\ball[1][n] - \x^{\ast}} \stackrel{\eqref{eq:coniceffdimssparse}}{\lesssim} s\log(\tfrac{2n}{s}).
	\end{equation}
\end{proposition}
An application of Theorem~\ref{thm:euclideanball} now leads to stable sparse recovery in the following sense: Let $\grtr\in \S^{n-1}$ be effectively $s$-sparse, i.e., $\grtr\in K=\sqrt{s}\ball[1][n]\intersec \ball[2][n]$, and let $t^{\ast}$ be defined as in Proposition~\ref{prop:effdimeffsparse}. Fix any $t\geq t^{\ast}$ and assume that the number of samples obeys $m\gtrsim t^{-2}\cdot s\log(\tfrac{2n}{s})$.
Then any minimizer $\solu$ of \eqref{eq:estimator} satisfies with high probability
\begin{equation}
\lnorm[auto]{\grtr - \frac{\solu}{\lnorm{\solu}}} \lesssim t.
\end{equation}
The best possible accuracy is achieved for $t = t^\ast$, which basically reflects the degree of compressibility of $\grtr$. In particular, $t^\ast$ is supposed to be very small as long as there exists a good $s$-term approximation for $\grtr$.
On the other hand, if $t$ exceeds this ``base level'', the sampling rate precisely behaves as if $\grtr$ would be $s$-sparse and $\sset$ is perfectly tuned such that $\grtrmu \in \boundary{\sset}$.

\subsection{Correlation Conditions and Noisy Quantization Models}
\label{subsec:appl:examples}

Let us begin with a brief discussion on the correlation conditions of Assumption~\ref{model:correlation} that are required for Theorem~\ref{thm:euclideanball}.
For the remainder of this section, we assume that the hypotheses of Subsection~\ref{subsec:results:unitball} hold true, in particular, $\sset\subset \ball[2][n]$.
Moreover, we agree on the following terminology:
According to Assumption~\ref{model:measurements}, denote the linear and quantized sampling rules by $y^{\text{lin}} = \sp{\a}{\grtr}$ and $y=f(\sp{\a}{\grtr})$, respectively. Analogously, $y^{\sign} = \sign(\sp{\a}{\grtr})$ is associated with the perfect $1$-bit model, which arises from $f=\sign$.

As already mentioned in the course of Lemma~\ref{lem:minexpectedloss} in Subsection~\ref{subsec:results:unitball}, the parameter $\lambda$ in condition~\ref{enum:cond_c1} simply corresponds to the covariance between quantized and linear measurements:
\begin{align}
\cov(y,y^{\text{lin}})=\cov(f(\sp{\a}{\grtr}),\sp{\a}{\grtr}) =\mean[f(\sp{\a}{\grtr})\sp{\a}{\grtr}] =\mean[f(\gaussianuniv)\gaussianuniv]=\lambda,
\end{align}
where $g = \sp{\a}{\grtr} \distributed \Normdistr{0}{1}$ is due to $\lnorm{\grtr}=1$. Thus, demanding $\lambda > 0$ in \ref{enum:cond_c1}
means that the true output $y$ is positively correlated with the linear signal $y^{\text{lin}}$.

In contrast, the condition of \ref{enum:cond_c2} requires that the covariance between the noisy observation model $y$ and its ``perfect'' counterpart $y^{\sign}$ is non-negative while conditioning on the magnitude of the underlying linear signal $y^{\text{lin}}$. Indeed, it holds that
\begin{align}
\mean[y  y^{\sign}\suchthat \abs{y^{\text{lin}}}]
=\mean[f(\sp{\a}{\grtr})\sign(\sp{\a}{\grtr})\suchthat \abs{\sp{\a}{\grtr}}]
=\mean[f(\gaussianuniv)\sign(\gaussianuniv)\suchthat \abs{\gaussianuniv}],
\end{align}
and
\begin{equation}
\mean[y^{\sign}\suchthat \abs{y^{\text{lin}}}]=\mean[\sign(\gaussianuniv)\suchthat \abs{\gaussianuniv}]=\mean[\sign(\gaussianuniv)]=0\quad\quad \text{(a.s.)},
\end{equation}
where we have used that $\sign(\gaussianuniv)$ and $\abs{\gaussianuniv}$ are independent.
Consequently, \ref{enum:cond_c2} is equivalent to 
\begin{equation}
\cov(y, y^{\sign} \suchthat \abs{y^{\text{lin}}}) = \mean[y  y^{\sign}\suchthat \abs{y^{\text{lin}}}]-\mean[y\suchthat \abs{y^{\text{lin}}}]\,\mean[y^{\sign}\suchthat \abs{y^{\text{lin}}}]\geq 0\quad\quad \text{(a.s.)}.
\end{equation}
This indicates that both \ref{enum:cond_c1} and \ref{enum:cond_c2} are mild and natural conditions which are fulfilled for many types of noisy $1$-bit quantization models, e.g., random bit flips (after quantization) and additive Gaussian noise (before quantization).

Before studying these two prototypical examples in greater detail, let us make the following general observation:
An ideal scenario for $1$-bit compressed sensing is that the signal vector $\grtr$ is first linearly measured by $y^{\text{lin}} = \sp{\a}{\grtr}$ and then quantized in a noiseless fashion via $y^{\sign} = \sign(\sp{\a}{\grtr})$.
This heuristic is particularly reflected by the size of the correlation parameter $\lambda$: If $f \colon \R \to \{-1,+1\}$ is an arbitrary quantization function, we have $f(\gaussianuniv)\gaussianuniv\leq \abs{\gaussianuniv}$, implying that
\begin{equation}
	\lambda_f = \mean[f(\gaussianuniv)\gaussianuniv]\leq \mean[\sign(\gaussianuniv)\gaussianuniv]=\lambda_{\sign}=\sqrt{\tfrac{2}{\pi}} \ .
\end{equation}
Hence, $\lambda_f$ is maximized for $f = \sign$, which eventually leads to the best possible error bound and sampling rate in Theorem~\ref{thm:euclideanball}.
However, it is difficult to build perfect $1$-bit quantizers in practice, so that the measurement process is usually disturbed by \emph{noise}. 
In this context, one roughly distinguishes between two sources of measurement errors: firstly, noise that corrupts the linear signal $y^{\text{lin}} = \sp{\a}{\grtr}$ before quantization, and secondly, noise that affects the actual quantization step.
The following example considers each of these two scenarios:
\begin{example}
\begin{rmklist}
\item \label{ex:addGauss}
	\emph{Additive Gaussian noise.}
	A typical example of corruptions before quantization is additive Gaussian noise. More precisely, we consider Assumption~\ref{model:measurements} for a $1$-bit quantization function
	\begin{equation}
		f(v) = f_{\sigma}(v) \coloneqq \sign(v+ \tau)
	\end{equation}
	where $\tau\distributed \Normdistr{0}{\sigma^2}$ is independent from $\a$. Consequently, the samples in \eqref{eq:measurements} take the form
	\begin{equation}\label{eq:Gaussiannoiseobservations}
		y_i = \sign(\sp{\a_i}{\grtr} + \tau_i), \quad i = 1, \dots, m,
	\end{equation}
	with $\tau_i \distributed \Normdistr{0}{\sigma^2}$ being an independent copy of $\tau$. 
	If $\sigma=0$, we are in the situation of perfect $1$-bit measurements, whereas all information about the signal $\grtr$ is lost as $\sigma \to \infty$.
\item\label{ex:bitflip}
	\emph{Random bit flips.} An example of noise during quantization are independent sign flips of $y^{\sign} = \sign(\sp{\a}{\grtr})$. This can be easily modeled by setting 
	\begin{equation}
		f(v) = f_p(v) \coloneqq \eps \cdot\sign(v)
	\end{equation} 
	in Assumption~\ref{model:measurements} for an independent Bernoulli random variable  $\eps\in \{-1,+1\}$ with $\prob[\eps = 1] = p>\nobreak\tfrac{1}{2}$. 
	Then, \eqref{eq:measurements} looks as follows:
	\begin{equation}\label{eq:bitflipobservations}
		y_i = \eps_i \cdot \sign(\sp{\a_i}{\grtr}), \quad i = 1, \dots, m,
	\end{equation}
	where $\eps_i\in \{-1,+1\}$ is an independent copy of $\eps$.
	Note that $p=1$ again corresponds to perfect $1$-bit measurements, whereas $p \to \tfrac{1}{2}$ leads to a complete loss of information.
	\qedhere
\end{rmklist}\label{ex:appl:noisepatterns}
\end{example}

The next proposition verifies that both types of noise are compatible with Assumption~\ref{model:correlation}. Its proof is postponed to Subsection~\ref{subsec:proofs:noisemodel}. 
\begin{proposition}\label{prop:bitflipaddnoisecond}
The noisy $1$-bit models from Example~\ref{ex:appl:noisepatterns}\ref{ex:addGauss} and~\ref{ex:bitflip} do both satisfy Assumption~\ref{model:correlation}. Moreover, the associated correlation parameters are given by
\begin{alignat}{2}\label{eq:example_lambda}
\lambda_{f_{\sigma}} &\asymp \tfrac{1}{1+\sigma} \qquad && \text{ for additive Gaussian noise, and} \\
\lambda_{f_p} &= (2p-1) \sqrt{\tfrac{2}{\pi}} \qquad && \text{ for random bit flips.}
\end{alignat} 
\end{proposition}
Recalling the assertions of Theorem~\ref{thm:euclideanball} and Corollary~\ref{coro:euclideanball}, we obtain the following conditions on the number of required measurements:
\begin{alignat}{2}
m &\gtrsim (1+\sigma)^2 \cdot C_{t,\sset} \qquad &&\text{ for additive Gaussian noise in Example~\ref{ex:appl:noisepatterns}\ref{ex:addGauss}, and}\\ 
m &\gtrsim \tfrac{1}{(2p-1)^{2}} \cdot C_{t,\sset} \qquad &&\text{ for random bit flips in Example~\ref{ex:appl:noisepatterns}\ref{ex:bitflip},}
\end{alignat}
where the constant $C_{t,\sset} > 0$ hides the dependence on the oversampling factor $t$ and the signal complexity.
This indicates that signal recovery is still feasible in the presence of strong noise, but when approaching the limiting cases of $\sigma\to \infty$ or $p\to \tfrac{1}{2}$, we clearly need to take more and more samples to ensure accurate estimates of $\grtr$.
For a numerical simulation related to the noise models of Example~\ref{ex:appl:noisepatterns}, see Experiment~\hyperref[para:appl:numerics:exp2]{2} in the next subsection.

\subsection{Numerical Experiments}
\label{subsec:appl:numerics}

In this subsection, we provide some numerical evidence of our theoretical findings from Section~\ref{sec:results}.
The experiments below study the estimation performance of hinge loss minimization \eqref{eq:estimator} and its scalable version \eqref{eq:estimatortuned} within different parameter regimes, each of them investigating a different aspect of our statistical analysis.
The notational conventions of this part follow again those of Section~\ref{sec:results} and it is especially helpful to recall the observation model of Assumption~\ref{model:measurements}.

\paragraph{Implementation and General Setup.}

In all our experiments, we are solving the convex programs \eqref{eq:estimator} and \eqref{eq:estimatortuned} using the \texttt{Matlab} software package \texttt{cvx}~\cite{cvx1,cvx2} with the default settings in place. Moreover, Gaussian random
numbers are generated using the standard \texttt{Matlab} command
\texttt{randn}. Let us now provide more details about the general specifications that all numerical simulations below have in common:

\begin{listing}
\item
	\emph{Sparse signals.} Although there are certainly many other interesting examples of structural hypothesis, we are focusing on the benchmark case of (high-dimensional) \emph{sparse} signal vectors here, in conjunction with an $\l{1}$-relaxation.
	More specifically, an $s$-sparse signal $\grtr \in \R^n$ is drawn as follows: generate independent Gaussian random numbers in the first $s$ entries and set the rest to zero.\footnote{Since the measurement vectors $\a_1, \dots, \a_m$ are independent standard Gaussian random vectors and our hinge loss programs do not have any prior information about the support of $\grtr$, it is not necessary to draw a random support at this point.} Then normalize the resulting vector to obtain $\grtr \in \S^{n-1}$.
	Following the reasoning of Subsection~\ref{subsec:appl:signalsets}, we take a scaled $\l{1}$-ball of the form $R\ball[1][n]$ as basis to construct the signal set $\sset$, depending on which estimator is actually examined.
\item
	\emph{$1$-bit measurements.} For a selected output function $\fobs$, the measurements are generated according to Assumption~\ref{model:measurements}. Thus, the measurement vectors $\a_1, \dots, \a_m \in \R^n$ are independent instances of a standard Gaussian random vector.
	The specific choices of $n$, $m$, and $\fobs$ may be different in each case and we refer to the Figures~\ref{fig:appl:numerics:exp1}--\ref{fig:appl:numerics:exp3} for  the individual parameter configurations.
\item
	\emph{Recovery error.} The outcome of a hinge loss minimization is denoted by $\solu \in \R^n$ and the resulting estimation error is computed by $\lnorm[\big]{\grtr - \solu / \lnorm{\solu}}$, such as done in our main results in Section~\ref{sec:results}. In order to improve the statistical accuracy, each of the experiments is repeated for a certain number of trials and the reported recovery error then corresponds to the arithmetic mean over all iterations.
\end{listing}

\paragraph{Experiment~1 (Comparison in the noiseless case).} \label{para:appl:numerics:exp1}

This numerical experiment is motivated by the following issues:
\begin{highlight}
	How do the approaches from Subsection~\ref{subsec:results:unitball} and Subsection~\ref{subsec:results:generalset} compete with each other in the situation of perfect $1$-bit measurements, i.e., $\fobs = \sign$? Moreover, concerning the conclusion of Subsection~\ref{subsec:results:generalset:scalable}: Is upscaling the signal set really necessary in the non-unit-ball case?
\end{highlight}
For this purpose, we intend to compare the reconstruction performance of the following three estimators (as a function of the number of available samples $m$):
\begin{properties}[3em]{E}
\item\label{item:numerics:estimUBL1}
	Unit-ball + $\l{1}$: Solve \eqref{eq:estimator} from Definition~\ref{def:results:unitball:estimator} with $\sset = (\lnorm{\grtr}[1] \ball[1][n])\intersec\ball[2][n]$.
\item\label{item:numerics:estimScaledL1}
	Scaled $\l{1}$: Solve \eqref{eq:estimatortuned} from Definition~\ref{def:estimatortuned} with $\sset = \lnorm{\grtr}[1] \ball[1][n]$ and $\scalfac \coloneqq \sqrt{m} / 10$. 
\item\label{item:numerics:estimNonScaledL1}
	Non-scaled $\l{1}$: Solve \eqref{eq:estimatortuned} from Definition~\ref{def:estimatortuned} with $\sset = \lnorm{\grtr}[1] \ball[1][n]$ and $\scalfac \coloneqq 1$.
\end{properties}
The choice of the scaling parameter $\scalfac$ in \ref{item:numerics:estimScaledL1} is inspired by the condition \eqref{eq:results:generalset:local:meas} of Theorem~\ref{thm:results:generalset:local}. Since the sparsity of the ground truth signal is kept fixed ($s = 10$), the scaling parameter $\scalfac$ only needs to be adjusted with respect to $m$ at this point. The multiplicative factor of $1/10$, on the other hand, is heuristically selected and could be further optimized. Note that in Experiment~\hyperref[para:appl:numerics:exp3]{3}, we will investigate this relationship in a more systematic way.
Apart from that, we point out that the $\l{1}$-balls used in \ref{item:numerics:estimUBL1}--\ref{item:numerics:estimNonScaledL1} are always scaled such that $\grtr$ lies on the boundary of $\sset$. This ``perfect'' tuning allows for a fairer comparison of all methods and avoids the technical issue of stability, as discussed above in the course of Proposition~\ref{prop:effdimeffsparse}.

\begin{figure}[!t]
	\centering
	\begin{minipage}{.4\textwidth}
		\begin{tabularx}{\textwidth}{Z{.3\textwidth}|Z{.7\textwidth}}
		Parameter & Value \\ \hline\hline
		Dimension & $n = 512$ \\ \hline 
		Sparsity & $s = 10$ \\ \hline 
		Model & $y = \sign(\sp{\a}{\grtr})$ \\ \hline 
		Measurements & {$m = 100, \dots, 2000$ \newline\vspace{.25\baselineskip} {\smaller (in $30$ equidistant steps)}} \\ \hline 
		Iterations & $50$ 
		\end{tabularx}
	\end{minipage}
	\begin{minipage}{.55\textwidth}
		\includegraphics[width=\textwidth]{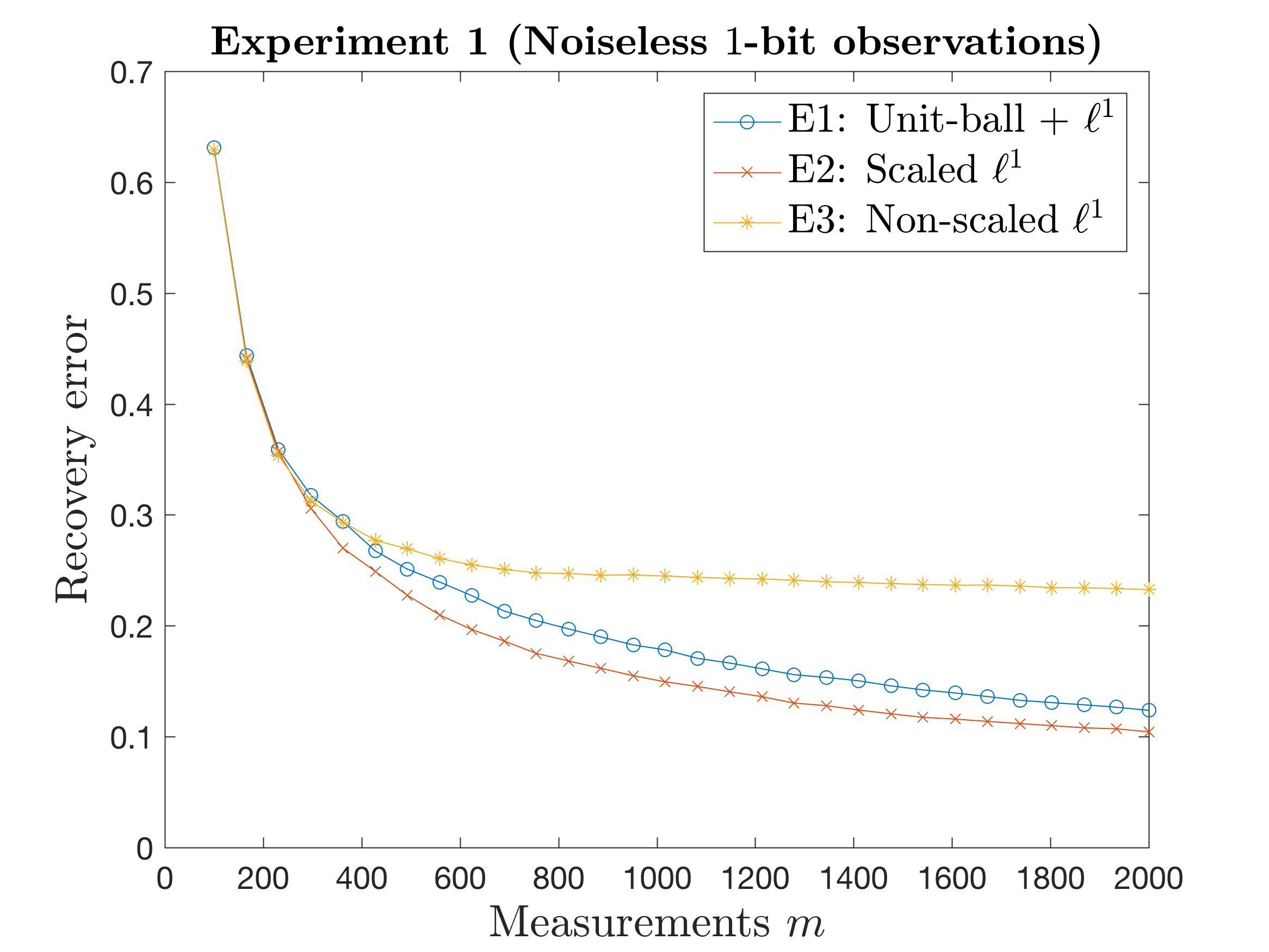}
	\end{minipage}
	\caption{Specifications and results of Experiment~\hyperref[para:appl:numerics:exp1]{1}.}
	\label{fig:appl:numerics:exp1}
\end{figure}

The parametric configuration and numerical results of Experiment~\hyperref[para:appl:numerics:exp1]{1} are reported in Figure~\ref{fig:appl:numerics:exp1}.
Here, we simply plot the recovery error of the estimators \ref{item:numerics:estimUBL1}--\ref{item:numerics:estimNonScaledL1} as a function of the sample size $m$.
A qualitative comparison of the behavior of the error curves allows us to draw the following conclusions:
\begin{listing}
\item
	The non-scalable estimator \ref{item:numerics:estimNonScaledL1} is clearly the worst, indicating that an accurate reconstruction of $\grtr$ is not possible in that way. This particularly confirms our initial motivation from Subsection~\ref{subsec:results:generalset:scalable}, namely that upscaling the signal set is crucial to obtain a consistent estimator.
\item
	The method of \ref{item:numerics:estimScaledL1} slightly outperforms \ref{item:numerics:estimUBL1}. While both approaches lead to a consistent estimator in theory, their ``success'' in fact relies on very different properties of the hinge loss.
	Informally speaking, the unit-ball estimator \ref{item:numerics:estimUBL1} benefits from a well-behaved multiplier term $\multiplterm{\cdot}{\grtrmu}$ in the noiseless case (cf. Lemma~\ref{lem:expnoisecor}), whereas the quadratic term $\quadrterm{\cdot}{\grtrmu}$ plays a fundamental role for the scalable estimator \ref{item:numerics:estimScaledL1}, to achieve a certain type of restricted strong convexity (cf. Subsection~\ref{subsec:proofs:generalset}).
	Consequently, both strategies work for quite different reasons, although the same loss function is used. This justifies once again our case distinction in Section~\ref{sec:results} between general convex signal sets and those contained in the unit ball.
\item
	All error curves behave very similarly for smaller values of $m$. This parameter region is however not of primary interest to us: The theoretical results of Section~\ref{sec:results} state that a reliable estimate of $\grtr$ requires considerable oversampling. Hence, the actual focus here is on the error decay rate when $m$ gets sufficiently large.
\end{listing}

\paragraph{Experiment~2 (Robustness against noise).} \label{para:appl:numerics:exp2}

Our second experiment takes again the estimators \ref{item:numerics:estimUBL1}--\ref{item:numerics:estimNonScaledL1} as a basis, but this time, we wish to address the following points:
\begin{highlight}
	How well does the unit-ball estimator \ref{item:numerics:estimUBL1} perform with noisy $1$-bit observations in practice?
	Since \ref{item:numerics:estimScaledL1} was only analyzed in the noiseless case in Subsection~\ref{subsec:results:generalset}: Is this estimator also robust against noise?
\end{highlight}
The parameter configuration and methodology of Experiment~\hyperref[para:appl:numerics:exp2]{2} is almost the same as the one of Experiment~\hyperref[para:appl:numerics:exp1]{1} in Figure~\ref{fig:appl:numerics:exp1}.
The only important difference is that we now examine noisy output functions: Experiment~\hyperref[para:appl:numerics:exp2]{2}a in Figure~\ref{fig:appl:numerics:exp2a} considers the \emph{random bit flip model} from Example~\ref{ex:appl:noisepatterns}\ref{ex:bitflip} with $p = 0.9$, i.e., there is a chance of $10$\% that the noiseless output $\sign(\sp{\a}{\grtr})$ is flipped.
Analogously, Experiment~\hyperref[para:appl:numerics:exp2]{2}b in Figure~\ref{fig:appl:numerics:exp2b} is based on the \emph{additive Gaussian noise model} studied in Example~\ref{ex:appl:noisepatterns}\ref{ex:addGauss} with $\sigma = 0.5$.

\begin{figure}[!t]
	\centering
	\begin{minipage}{.4\textwidth}
		\begin{tabularx}{\textwidth}{Z{.3\textwidth}|Z{.7\textwidth}}
		Parameter & Value \\ \hline\hline
		Dimension & $n = 512$ \\ \hline 
		Sparsity & $s = 10$ \\ \hline 
		Model & $y = \eps \cdot \sign(\sp{\a}{\grtr})$ \newline\vspace{.25\baselineskip} $p = \prob[\eps = 1] = 0.9$ \\ \hline 
		Measurements & {$m = 100, \dots, 2000$ \newline\vspace{.25\baselineskip} {\smaller (in $30$ equidistant steps)}} \\ \hline 
		Iterations & $50$ 
		\end{tabularx}
	\end{minipage}
	\begin{minipage}{.55\textwidth}
		\includegraphics[width=\textwidth]{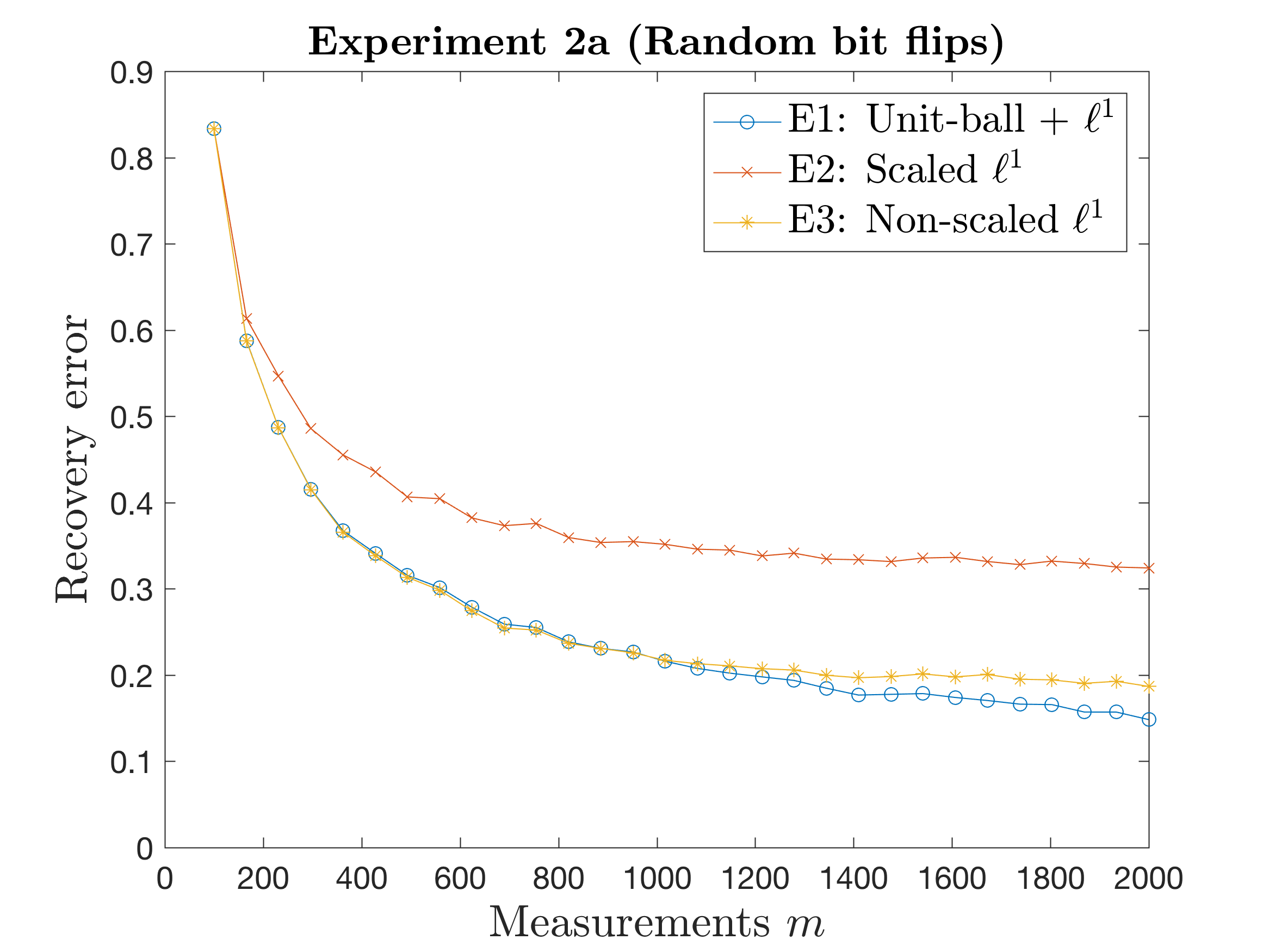}
	\end{minipage}
	\caption{Specifications and results of Experiment~\hyperref[para:appl:numerics:exp2]{2}a (Random bit flips, Example~\ref{ex:appl:noisepatterns}\ref{ex:bitflip} with $p = 0.9$).}
	\label{fig:appl:numerics:exp2a}
\end{figure}

\begin{figure}[!t]
	\centering
	\begin{minipage}{.4\textwidth}
		\begin{tabularx}{\textwidth}{Z{.3\textwidth}|Z{.7\textwidth}}
		Parameter & Value \\ \hline\hline
		Dimension & $n = 512$ \\ \hline 
		Sparsity & $s = 10$ \\ \hline 
		Model & $y = \sign(\sp{\a}{\grtr} + \tau)$ \newline\vspace{.25\baselineskip} $\tau\distributed \Normdistr{0}{(0.5)^2}$ \\ \hline 
		Measurements & {$m = 100, \dots, 2000$ \newline\vspace{.25\baselineskip} {\smaller (in $30$ equidistant steps)}} \\ \hline 
		Iterations & $50$ 
		\end{tabularx}
	\end{minipage}
	\begin{minipage}{.55\textwidth}
		\includegraphics[width=\textwidth]{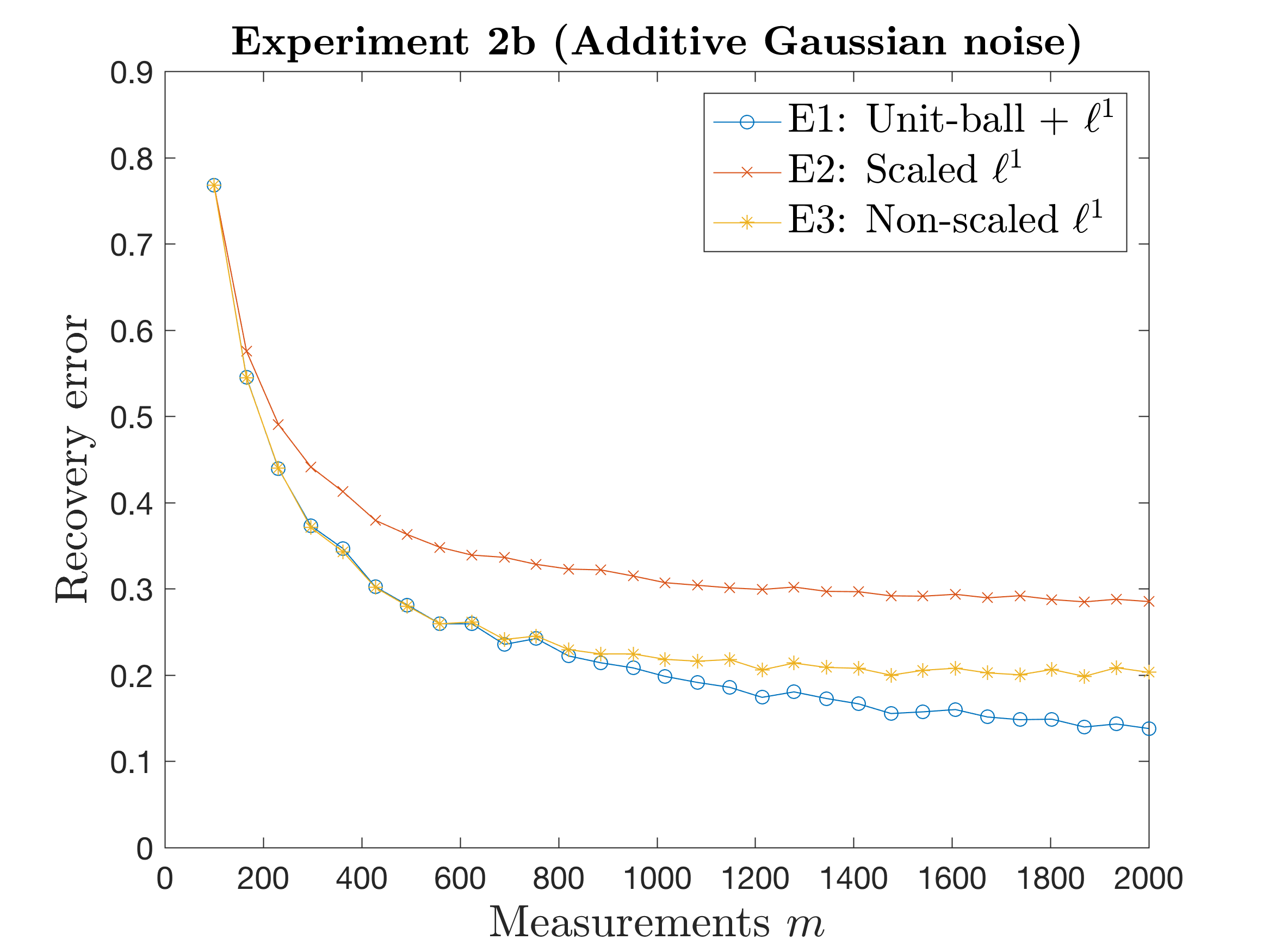}
	\end{minipage}
	\caption{Specifications and results of Experiment~\hyperref[para:appl:numerics:exp2]{2}b (Additive Gaussian noise, Example~\ref{ex:appl:noisepatterns}\ref{ex:addGauss} with $\sigma = 0.5$).}
	\label{fig:appl:numerics:exp2b}
\end{figure}

Since Figure~\ref{fig:appl:numerics:exp2a} and Figure~\ref{fig:appl:numerics:exp2b} exhibit error curves that are qualitatively very similar, let us discuss these numerical results jointly:
\begin{listing}
\item
	As predicted by the theoretical result of Theorem~\ref{thm:euclideanball}, the unit-ball estimator \ref{item:numerics:estimUBL1} is quite robust against noise, both before and after quantization.
\item
	The scalable estimator \ref{item:numerics:estimScaledL1} performs rather poorly. Indeed, our statistical analysis of this method is actually tailored to perfect $1$-bit observations and we have pointed out at the beginning of Subsection~\ref{subsec:results:generalset} that an extension to noisy observations is far from being obvious. This experiment provides numerical evidence of this claim.
\item
	At first sight, it is somewhat astonishing that the non-scaled estimator \ref{item:numerics:estimNonScaledL1} works much better than in the noiseless case considered in Figure~\ref{fig:appl:numerics:exp1} and that its recovery performance is comparable to \ref{item:numerics:estimUBL1}.
	In fact, both error curves would approach even closer as the noise level gets higher, i.e., $p$ smaller or $\sigma$ larger, respectively. An (informal) theoretical explanation of this phenomenon is that the expected risk minimizer would take the form $\scalfac\grtr$ with $\scalfac \ll 1$ in these cases, regardless of whether $\sset$ belongs to the unit ball or not.
	Thus, one would still obtain a consistent estimator, while the structure of the signal set $\sset$ is virtually not exploited anymore.
\end{listing}

\paragraph{Experiment~3 (The impact of sparsity).} \label{para:appl:numerics:exp3}

Our third experiment is inspired by the localized analysis of the scalable hinge loss estimator \eqref{eq:estimatortuned} from Subsection~\ref{subsec:results:generalset}. More specifically, we would like to focus on the following aspect of Theorem~\ref{thm:results:generalset:local}:
\begin{highlight}
	Is the presence of the additional geometric factor $t_0$ real or is it rather an artifact of the proof of Theorem~\ref{thm:results:generalset:local}?
	Can we expect that recovery of an $s$-sparse signal $\grtr \in \S^{n-1}$ via \eqref{eq:estimatortuned} already ``succeeds'' with $m \asymp \scalfac^{2} \cdot s\log(2n / s)$?
\end{highlight}
As pointed out in the discussion subsequent to Theorem~\ref{thm:results:generalset:local} as well as in Remark~\ref{rmk:proofs:generalset:complexity}, a general answer to these questions is unfortunately out of reach for us.
The following experiment is based on a simple numerical test that could at least allow for further insights into this problem.
For this purpose, let us pretend that the optimal sampling rate would be $m \asymp \scalfac^{2} \cdot s\log(2n / s)$.
In other words, we combine Theorem~\ref{thm:results:generalset:local} with the (sharp) upper bound on the conic effective dimension from \eqref{eq:coniceffdimssparse} and imagine that the factor $t_0$ can be omitted in \eqref{eq:results:generalset:local:meas}.
One possible indicator for the validity of such a hypothesis is now as follows: Fix the scaling parameter $\scalfac$ and verify if the recovery accuracy achieved by \eqref{eq:estimatortuned} remains (almost) constant as a function of the sparsity~$s$.\footnote{Note that in contrast to the case of linear observations (cf. \cite{amelunxen2014edge}), one cannot expect that the recovery error undergoes a sharp phase transition if $s$ or $m$ are increased.}
This consideration suggests to evaluate the recovery performance of the following estimator:
\begin{properties}[3em]{E} \setcounter{enumi}{3}
\item\label{item:numerics:estimScaledL1-sparsity}
	For $\scalfac > 0$ and $s \in [n]$, set $m \coloneqq \scalfac^{2} \cdot s\log(\tfrac{2n}{s})$ and solve \eqref{eq:estimatortuned} with $\sset = \lnorm{\grtr}[1] \ball[1][n]$.
\end{properties}
One clearly has to exercise some caution here, since the choice of $m$ in \ref{item:numerics:estimScaledL1-sparsity} disregards any numerical constants and the probability parameter $\eta$ involved in Theorem~\ref{thm:results:generalset:local}.
However, an empirical validation has shown that these parameters do only marginally affect the (qualitative) outcome of our numerical simulation.

\begin{figure}[!t]
	\centering
	\begin{minipage}{.42\textwidth}
		\begin{tabularx}{\textwidth}{Z{.3\textwidth}|Z{.7\textwidth}}
		Parameter & Value \\ \hline\hline
		Dimension & $n = 128$ \\ \hline 
		Sparsity & $s = 1,2,3, \dots, 128$ \\ \hline 
		Model & $y = \sign(\sp{\a}{\grtr})$ \\ \hline 
		Measurements & {$m = \scalfac^{2} \cdot s \log(\tfrac{2n}{s})$ \newline\vspace{.25\baselineskip} with $\scalfac^{-1} = 0.07, \dots, 0.3$ \newline\vspace{.25\baselineskip} \smaller{(in $15$ equidistant steps)}} \\ \hline 
		Iterations & $20$ 
		\end{tabularx}
	\end{minipage}
	\begin{minipage}{.55\textwidth}
		\includegraphics[width=\textwidth]{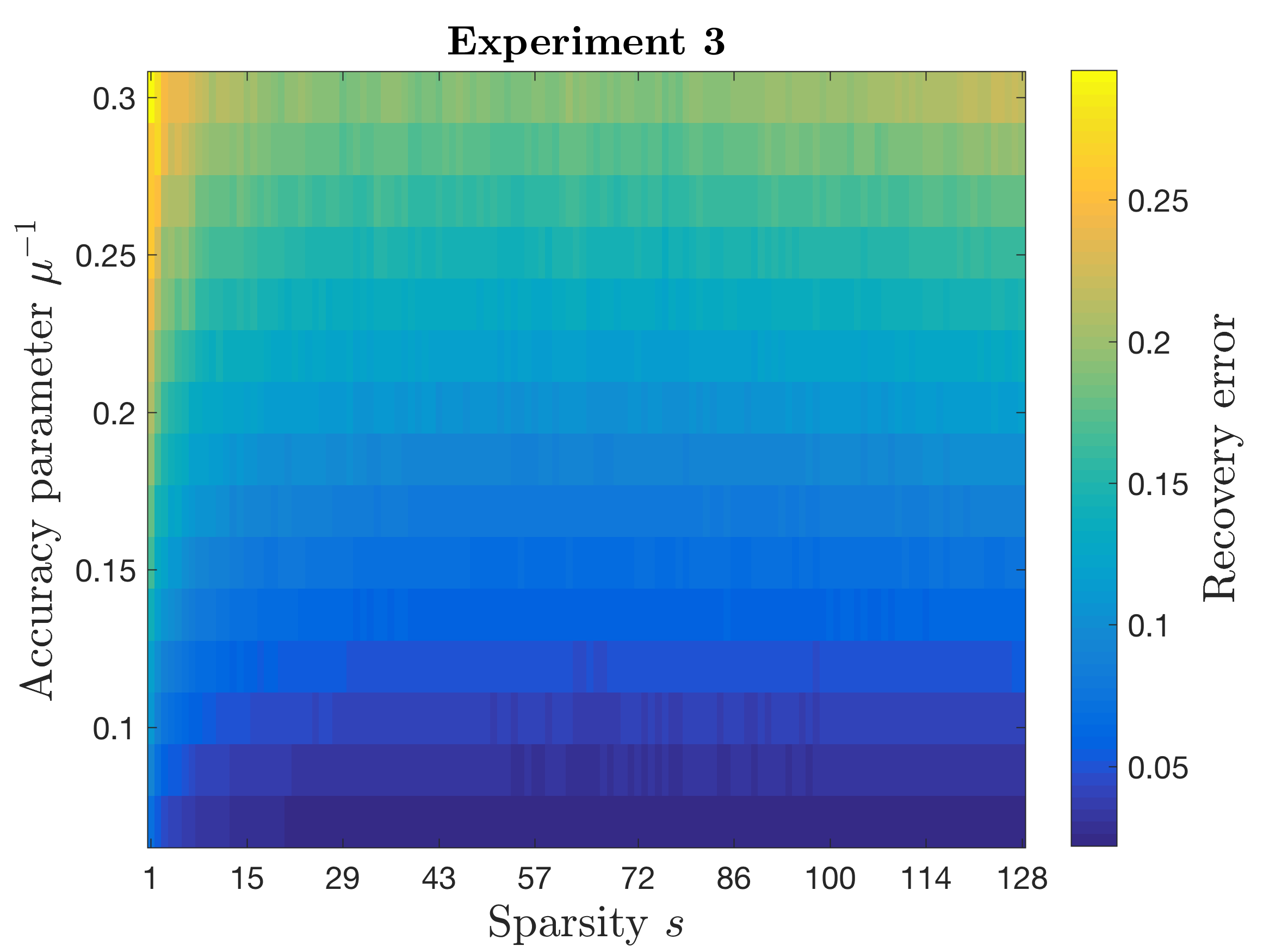}
	\end{minipage}
	\caption{Specifications and results of Experiment~\hyperref[para:appl:numerics:exp3]{3}. The recovery error refers to the scaled hinge loss estimator \ref{item:numerics:estimScaledL1-sparsity}.}
	\label{fig:appl:numerics:exp3}
\end{figure}

The parameter configuration and results of Experiment~\hyperref[para:appl:numerics:exp3]{3} are reported in Figure~\ref{fig:appl:numerics:exp3}, where the recovery error referred to \ref{item:numerics:estimScaledL1-sparsity} is plotted in color as a function of $s$ and $\scalfac$.
There are two interesting observations worth pointing out:
\begin{listing}
\item
	For any fixed $\scalfac$, the recovery error is almost constant in $s$, except for very small values of $s$ (corresponding to the vertical cross sections at the left end of the color plot in Figure~\ref{fig:appl:numerics:exp3}).
	In the latter case, the estimation behavior of \ref{item:numerics:estimScaledL1-sparsity} changes significantly, which provides evidence that a sampling rate of $m \asymp \scalfac^{2} \cdot s\log(2n / s)$ might be too optimistic in the very-sparse regime.
	One possible theoretical explanation of this phenomenon could be that the boundary of $\sset = \lnorm{\grtr}[1] \ball[1][n]$ is more ``tangent'' to $\spann\{\grtr\}$ for extremely small values of $s$. According to the discussion of Theorem~\ref{thm:results:generalset:local}, this implies a large value of $t_0$ and therefore also a higher sampling rate in \eqref{eq:results:generalset:local:meas}.
\item
	The true recovery error achieved by \ref{item:numerics:estimScaledL1-sparsity} in Figure~\ref{fig:appl:numerics:exp3} is not worse than the theoretical accuracy of $\scalfac^{-1}$ promoted by Theorem~\ref{thm:results:generalset:local}. While Theorem~\ref{thm:results:generalset:local} certainly involves several unspecified numerical constants, this empirical observation at least indicates that these constants are not too badly behaved in practice.
\end{listing}

%
%
%
%
%
%
%
%
%
%
%
%
%
%
%

\section{Related Literature}
\label{sec:literature}

In this part, we give a brief overview of some recent approaches from the literature that are closely related to the problem setup considered in this work.
While our focus is clearly on a challenge in (non-linear) compressed sensing, as discussed in Subsection~\ref{subsec:literature:1bit}, we will also point out connections to statistical learning in Subsection~\ref{subsec:literature:statlearn}, which is particularly useful to understand the proof strategy of our main results.

\subsection{Signal Processing and Compressed Sensing}
\label{subsec:literature:1bit}

As already mentioned before, our measurement model in Assumption~\ref{model:measurements} fits well into the framework of $1$-bit compressed sensing, or more generally, \emph{non-linear} compressed sensing.
In fact, there is an increasing interest in this subject in the recent literature; we refer the interested reader to \cite{boufounos2015quantization,dirksen2017partialcirc} and the references therein for an overview.
Perhaps the most related branch of research is by Plan, Vershynin, and collaborators \cite{plan2013robust,plan2013onebit,plan2014highdim,plan2014tessellation,plan2015lasso,baraniuk2017onebitdict}, whose model assumptions are very similar to ours.
Indeed, \cite{plan2015lasso} deals with the estimation of a structured index vector $\grtr \in \sset \subset \R^n$ from \emph{single-index observations}
\begin{equation}\label{generalquantization}
	y_i=f_i(\sp{\a_i}{\grtr}), \quad  i = 1, \dots, m,
\end{equation}
where the $f_i$ are independent copies of an unknown function $f \colon \R \to \R$ that could be non-linear and random. In particular, if $f$ is binary-valued, we precisely end up with the sampling rule of \eqref{eq:measurements}.
Under the hypothesis of i.i.d.\ Gaussian measurement vectors, \cite{plan2015lasso} investigates the performance of the \emph{generalized Lasso}
\begin{equation}\label{eq:squarelossestimator}\tag{$P_{\losssq,\sset}$}
	\min_{\x \in \R^n} \tfrac{1}{m} \sum_{i = 1}^m (\sp{\a_i}{\x} - y_i)^2 \quad \text{subject to \quad $\x \in \sset$,}
\end{equation}
which simply corresponds to \eqref{eq:generalestimator} using the square loss $\loss(v, v') \coloneqq \losssq(v-v') \coloneqq (v-v')^2$.
While the Lasso was originally designed to solve \emph{linear} regression problems, the recovery results of \cite{plan2015lasso} reveal that \eqref{eq:squarelossestimator} is surprisingly robust against non-linear distortions, even if the output function $f$ is completely unknown.
More technically, it turned out that, for an appropriate scaling parameter $\scalfac \in \R$, the model mismatch $y_i - \sp{\a_i}{\grtrmu}$ is uncorrelated to the measurement vector $\a_i$. With other words, the Lasso with non-linear inputs essentially works as well as if the inputs would follow a noisy linear model.

Despite the universal applicability of the Lasso, practitioners however often choose different types of loss functions for \eqref{eq:generalestimator}, which are specifically tailored to their model hypotheses, e.g., if the output variables $y_i$ are discrete.
This issue particularly motivated the first author in \cite{genzel2016estimation} to extend the framework of Plan and Vershynin to other choices of $\loss$. A key finding of \cite{genzel2016estimation} is that, in many situations of interest, \emph{restricted strong convexity} (\emph{RSC}) is a crucial property of an empirical risk function to ensure successful signal recovery via \eqref{eq:generalestimator}.
The criterion of RSC is indeed satisfied for a large class of loss functions, for instance, all those $\loss \colon \R \times \R \to \R$ which are twice differentiable in the first variable and locally strongly convex in a neighborhood of the origin (cf. \cite[Thm.~2.5]{genzel2016estimation}).
While this includes popular choices of $\loss$, such as the logistic loss, the hinge loss $\loss(v,v') = \losshng(v \cdot v') = \max\{0, 1 - v \cdot v'\}$ unfortunately does not meet these sufficient conditions at all.
Therefore, it is a somewhat surprising observation of this work that the hinge loss function still satisfies RSC (see Proposition~\ref{prop:quadprocesslowerbound} and Remark~\ref{rmk:proofs:unitball:rsc}) and similar recovery statements as in \cite{plan2015lasso,genzel2016estimation} remain valid. 
Apart from that, we wish to emphasize that the proofs of our results improve some of the techniques used in \cite{plan2015lasso,genzel2016estimation}.
For example, the linear multiplier term of the hinge loss is now handled by a more sophisticated concentration inequality due to Mendelson (see Theorem~\ref{thm:multiplierprocess}), which eventually leads to an enhanced probability of success in our reconstruction guarantees.

At this point, it is again worth mentioning the recent work of Kolleck and Vyb\'{\i}ral \cite{kolleck2015l1svm} whose analysis of hinge loss minimization is directly related to ours. We have already presented several details of their approach in Subsection~\ref{subsec:intro:overview} and Remark~\ref{rmk:results:generalset}\ref{rmk:results:generalset:kvl1svm}, as well as a link to support vector machines, which served as their original motivation (see Subsection~\ref{subsec:results:generalset:geometry}).
Let us briefly recap to what extent our results improve the recovery guarantees of \cite{kolleck2015l1svm}:
Firstly, we go beyond $\l{1}$-based signal sets and allow for arbitrary convex bodies as structural constraints. Secondly, the error bounds in \cite{kolleck2015l1svm} do only achieve an oversampling rate of $\asympfaster{m^{-1/4}}$, while Theorem~\ref{thm:euclideanball} and Theorem~\ref{thm:results:generalset:local} exhibit the (optimal) rate of $\asympfaster{m^{-1/2}}$. And finally, the noise patterns considered in \cite{kolleck2015l1svm} are far more restrictive than what is permitted by Assumption~\ref{model:measurements} and Assumption~\ref{model:correlation}.
The latter issue is probably an artifact of the proof techniques adapted from \cite{plan2013robust}, where only a constrained \emph{linear} estimator is investigated.
This is substantially different from our statistical framework, in which the ``quadratic'' part of the hinge loss is explicitly taken into account (see Proposition~\ref{prop:quadprocesslowerbound}).

\subsection{Statistical Learning Theory}
\label{subsec:literature:statlearn}

Since our proof strategy heavily relies on tools from statistical learning theory, let us also briefly discuss how our approach relates to this field of research.
In that context, the model of Assumption~\ref{model:measurements} is regarded as a \emph{sampling procedure} according to which the sample set $\{(\a_i, y_i)\}_{i \in [m]}$ is independently drawn from a random pair $(\a,y)$ that obeys a (partially) unknown probability distribution on $\R^n \times \{-1,+1\}$.
The measurement vectors $\a_i$ do usually play the role of data (or feature) vectors, whereas $y_i$ denotes a class label that depends on these features in some way.

One of the key goals in statistical learning is then to specify a (deterministic) \emph{prediction function} $F \colon \R^n \to \R$ that minimizes the risk of wrongly predicting the true label $y$ by $F(\a)$. Since only a finite collection of samples is given instead of $(\a,y)$, this is in fact a challenging problem and one typically restricts the set of predictors to a convex \emph{hypothesis set} $\hypospace$ (a subset of measurable functions), encoding one's beliefs in the underlying observation model.
Due to the specific form of our output rule, i.e., $y = f(\sp{\a}{\grtr})$, it is quite natural to consider a linear hypothesis set
\begin{equation}
	\hypospace = \hypospace_\sset = \{ \vec{v} \mapsto \sp{\vec{v}}{\x} \suchthat \x \in \sset \},
\end{equation}
where $\sset \subset \R^n$ is convex. Identifying $\hypospace_\sset$ with $\sset$, this precisely reflects Assumption~\ref{model:signal}.

The purpose of our main results in Section~\ref{sec:results} is to study the capability of the associated empirical risk minimizer $\solu$ of \eqref{eq:estimator} to approximate the ground truth vector $\grtr$. In the literature, such types of statements are often referred to as \emph{estimation}, but note that, somewhat unusually, the expected risk minimizer does not necessarily belong to $\spann\{ \grtr \}$, see Subsection~\ref{subsec:results:generalset:scalable}.
This stands in contrast to the above problem of \emph{prediction} in which one is rather interested in controlling the so-called \emph{sample error}
\begin{equation}\label{eq:literature:statlearn:sampleerror}
	\mean[\losshng(y \sp{\a}{\solu})] - \min_{\x \in \sset} \mean[\losshng(y \sp{\a}{\x})].
\end{equation}
Indeed, a small sample error does not automatically imply that the normalized minimizer $\solu / \lnorm{\solu}$ is also close to $\grtr$.
For more details on estimation and prediction, we refer to \cite{mendelson2014learninggeneral} and the references therein; for a comprehensive overview of statistical learning theory, one may also consider \cite{vapnik1998learning,cucker2007learning,hastie2009elements,shalev2014understanding}.

Of particular relevance to our approach are the works of Mendelson \cite{mendelson2014learning,mendelson2014learninggeneral} on learning without concentration.
His estimation results for empirical risk minimization establish very general principles that relate geometric properties of the hypothesis set to the sampling rate.
While these statements bear a certain resemblance to ours, the actual goals of Mendelson are somewhat different: \cite{mendelson2014learning,mendelson2014learninggeneral} consider a very abstract model setting, where the output variable $y$ is left unspecified and the hypothesis set is not just restricted to linear functions.
The key concern of \emph{Mendelson's small ball method} developed in \cite{mendelson2014learning,mendelson2014learninggeneral} is to allow for heavy tailed feature variables, for which concentration inequalities fail to hold true. Following this strategy, it is still possible to prove powerful estimation guarantees under very mild assumptions on the underlying probability measure.

In contrast, we investigate a specific $1$-bit output rule with Gaussian data.
This enables us to prove much more explicit error bounds and to precisely quantify the recovery behavior of the hinge loss estimator.
But let us emphasize that the results of this work are not implicitly contained in the framework of \cite{mendelson2014learning,mendelson2014learninggeneral} because the hinge loss does by far not satisfy the required assumptions, especially local strong convexity (cf. Remark~\ref{rmk:proofs:unitball:rsc}). 
The very recent work of \cite{pierre2017estimation} follows an alternative path to tackle this issue: Based on regularized empirical risk minimization,\footnote{Compared to the program \eqref{eq:generalestimator}, a regularized estimator takes the form ${\displaystyle\min_{\x \in \sset}} \tfrac{1}{m} \sum_{i = 1}^m \loss(\sp{\a_i}{\x},y_i) + \lambda \norm{\x}$.} the authors prove estimation bounds for Lipschitz loss functions, which in principle also includes the hinge loss.
Their theoretical findings again hold true in a fairly general learning setting, but the actual statements rely on an abstract \emph{Bernstein condition} that the loss function needs to fulfill.
While the lower bound on the excess risk in \eqref{eq:excessriskpositiveunitball} actually resembles such a \emph{Bernstein condition}, it is still unclear whether the framework of \cite{pierre2017estimation} would apply to our setup. Indeed, the proof of \eqref{eq:excessriskpositiveunitball} turns out to be highly non-trivial, so that verifying the general assumptions of \cite{pierre2017estimation} might take a lot of effort in a specific model situation.
This observation manifests once more that, despite obvious overlaps, the fields of statistical learning and signal processing address different types of problems.

\section{Conclusion and Outlook}
\label{sec:conclusion}

Our main results show that $1$-bit compressed sensing via hinge loss minimization is indeed feasible under fairly general model conditions.
This particularly includes a wide class of noisy bit flip patterns (see Theorem~\ref{thm:euclideanball}) as well as arbitrary convex constraint sets that encode structural hypotheses, such as sparsity (see Theorem~\ref{thm:results:generalset:global} and Theorem~\ref{thm:results:generalset:local}).
While comparable recovery guarantees were recently established for different loss functions \cite{genzel2016estimation,mendelson2014learninggeneral}, it is somewhat astonishing that these assertions essentially remain valid for the hinge loss, since it is neither differentiable nor locally strongly convex.
The proofs of our results however strongly rely on the specific form of $1$-bit observations and require several sophisticated adaptions of previous arguments.
For this reason, we do not expect that empirical hinge loss minimization is as universally applicable as the Lasso (cf. \cite{plan2015lasso,genzel2017msense,mendelson2014learning}).
On the other hand, the special ability of the hinge loss to deal with binary outputs also implies computational advantages. For example, the estimator \eqref{eq:estimator} can be recasted as a linear program in the case of $\l{1}$-constraints (cf. \cite[Sec.~VI.A]{kolleck2015l1svm}), which in turn is appealing for practical purposes.

Let us conclude with some potential extensions and open issues that might be investigated in future works:
\begin{listing}
\item
	\emph{Relaxing the model assumptions.}
	Although the technical details are not elaborated here, we suppose that --- by adapting known proof strategies --- the following points are relatively straightforward generalizations of our model setup: adversarial bit flips (see \cite{genzel2016estimation}), unnormalized signal vectors, and anisotropic sub-Gaussian measurements (see \cite{genzel2017msense,mendelson2014learning,genzel2016fs}).
	
	A probably more challenging problem is to unify the respective hypotheses of our main results from Section~\ref{sec:results}, ultimately leading to a recovery guarantee that allows for general convex signal sets and noisy observations at the same time.
	In fact, there are several significant differences in the argumentation of Subsection~\ref{sec:proof:unitball} and Subsection~\ref{subsec:proofs:generalset}.
	This particularly concerns the slightly different role of signal complexity in both parts, which is not even fully understood in the situation of general convex constraints (see Remark~\ref{rmk:proofs:generalset:complexity}).
	Apart from that, the results of Experiment~\hyperref[para:appl:numerics:exp1]{1} and Experiment~\hyperref[para:appl:numerics:exp2]{2} in Subsection~\ref{subsec:appl:numerics} provide evidence that our two reconstruction methods do indeed perform quite differently in practice.
\item
	\emph{Different loss functions.}
	The hinge loss is actually a prototypical example of a piecewise linear loss function.
	Since our analysis shows that the associated empirical risk function satisfies restricted strong convexity under certain conditions (cf. Remark~\ref{rmk:proofs:unitball:rsc}), one could expect that this important property holds true for a larger class of convex piecewise linear losses.
	However, the proofs in Section~\ref{sec:proofs} do crucially rely on the specific form of the quadratic term $\quadrterm{\x}{\grtrmu}$ that is associated with the hinge loss (cf. \eqref{eq:quadrterm}). While this term would take a similar algebraic form for general piecewise linear loss functions (involving step functions for each component), an adaption of the individual proof steps would certainly require a lot of care and technical effort.
	Nevertheless, we expect that our techniques can at least serve as a template in this context.
\item
	\emph{Optimal choice of the loss.}
	An issue that is closely related to the previous one is the following: Supposed we have (partial) knowledge of the true observation model, what is a good or even optimal choice of loss in empirical risk minimization \eqref{eq:generalestimator}? More specifically, when is a loss function, e.g., the hinge loss, superior to others? What practical rules-of-thumb can be derived from this study?
	These questions are of course quite vaguely formulated. One of the major difficulties is to come up with a quantitative measure to assess the recovery performance of a loss. Such a benchmark would also involve sharp lower bounds on the recovery error, which we consider as a very challenging problem on its own.
\item
	\emph{Regularized estimation.}
	From an algorithmic perspective, it can be very useful to solve a \emph{regularized} optimization problem of the form
	\begin{equation}
		\min_{\x \in \R^n} \tfrac{1}{m} \sum_{i = 1}^m \loss(\sp{\a_i}{\x},y_i) + \lambda \norm{\x}
	\end{equation}
	instead of \eqref{eq:generalestimator}. Here, the norm $\norm{\cdot}$ encourages structured solutions, similarly to the constraint $\x \in \sset$ in \eqref{eq:generalestimator}.
	An adaption of our results to such types of estimators is by far not obvious and might rely on rather different arguments in the proofs. See also \cite{pierre2017estimation,lecue2016regularization,lecue2017regularization} for recent achievements for regularized empirical risk minimization in statistical learning.
\end{listing}

\section{Proofs of the Main Results}
\label{sec:proofs}

Let us start with a brief roadmap of the proof strategy pursued in this section. The common recovery approach of Subsection~\ref{subsec:results:unitball} and Subsection~\ref{subsec:results:generalset} is to estimate the ground truth signal $\grtr \in \S^{n-1}$ via constrained empirical risk minimization. More specifically, we invoke the program of \eqref{eq:estimator} if the signal set $\sset \subset \R^n$ is contained in the Euclidean unit ball, whereas the scalable estimator \eqref{eq:estimatortuned} is used for general convex constraints.
In both cases, it will turn out that, with high probability, the respective minimizer $\solu$ resides in a certain (local) neighborhood of $\grtrmu$ for an appropriately chosen scaling parameter $\scalfac > 0$.
Such a localization argument is in fact widely used in estimation theory and dates back to classical works in geometric functional analysis and statistical learning (e.g., see \cite{milman1986banach,pajor1986gelfand,mendelson2002improving,bartlett2005local,mendelson2007subgaussian}). In order to make this idea more precise in our specific setup, let us introduce the \emph{excess risk functional}
\begin{equation}
\excessloss(\x) \coloneqq \lossemp[](\x)-\lossemp[](\grtrmu), \quad \x \in \R^n,
\end{equation}
where $\lossemp[](\x)= \tfrac{1}{m} \sum_{i = 1}^m \losshng(y_i \sp{\a_i}{\x})$ denotes the empirical risk (cf. Definition~\ref{def:results:unitball:risk}), which serves as objective functional in both recovery programs. 
The following simple observation shows that positivity of the excess risk allows us to reduce the set of potential minimizers:
\begin{fact}\label{fact}
Let $\excessloss(\x) > 0$ for some $\x \in \R^n$ and consider $\operatorname{Ray}(\x) \coloneqq \{ \grtrmu + \tau(\x - \grtrmu) \suchthat \tau \geq 0 \}$, which is the ray starting at $\grtrmu$ and passing through $\x$. Then, we have $\excessloss(\grtrmu + \tau(\x - \grtrmu)) > 0$ for all $\tau \geq 1$.
Moreover, if a minimizer $\solu$ of \eqref{eq:estimator} or \eqref{eq:estimatortuned} belongs to $\operatorname{Ray}(\x)$, it must be contained in the line segment between $\grtrmu$ and $\x$, i.e., $\solu \in \convhull\{ \grtrmu, \x \}$. This particularly implies $\lnorm{\grtrmu - \solu} \leq \lnorm{\grtrmu - \x}$.
\end{fact}
This claim directly follows from the convexity of the excess risk and constraint set as well as the fact that $\excessloss(\grtrmu) = 0$ and $\excessloss(\solu)\leq 0$. Note that the last inequality holds true because $\solu$ minimizes the empirical risk on a certain signal set and $\grtrmu$ is a feasible vector. Figure~\ref{fig:proofs:fact} demonstrates how one can use Fact~\ref{fact} to derive an error bound for $\solu$ in the case of spherical intersections.

\begin{figure}
	\centering
	\tikzstyle{blackdot}=[shape=circle,fill=black,minimum size=1mm,inner sep=0pt,outer sep=0pt]
		\begin{tikzpicture}[scale=2]
			\coordinate (K1) at (-.5,-.3);
			\coordinate (K2) at (.5,-1);
			\coordinate (K3) at (.3,-1.7);
			\coordinate (K4) at (-.2,-2.1);
			\coordinate (K5) at (-1,-1.2);
			\coordinate[below right=.25cm and .05 of K1] (muX0);
			
			\draw[fill=gray!20!white, name path = K] (K1) -- (K2) -- (K3) -- (K4) -- (K5) -- cycle;
			\node[draw,circle,fill=gray!20!white,minimum size=1.4cm,label={[label distance=-2pt]30:$tB_2^n + \grtrmu$}] at (muX0) {}; 
			\begin{scope}
				\clip (K1) -- (K2) -- (K3) -- (K4) -- (K5) -- cycle;
				\node[draw=red,circle,fill=gray!60!white,minimum size=1.4cm,ultra thick] at (muX0) {};
			\end{scope}
			\draw[thick] (K1) -- (K2) -- (K3) -- (K4) -- (K5) -- cycle;
			
			\node at (barycentric cs:K1=1,K2=1,K3=1,K4=1,K5=1) {$\sset$};
			\node[blackdot,label=below :$\grtrmu$] at (muX0) {};
		\end{tikzpicture}
	\caption{If $\excessloss(\cdot)$ is positive on a spherical intersection $\sset \intersec (t\S^{n-1} + \grtrmu)$ (red arc), then Fact~\protect{\ref{fact}} implies that every minimizer $\solu$ of \protect{\eqref{eq:estimator}} must belong to $\sset \intersec (t\ball[2][n] + \grtrmu)$ (dark gray region), that means, we have $\lnorm{\grtrmu - \solu} \leq t$.}
	\label{fig:proofs:fact}
\end{figure}

Consequently, the actual key challenge is to verify that the excess risk is positive on the boundary of an appropriate neighborhood of $\grtrmu$, e.g., a small Euclidean ball.
For this purpose, it is useful to consider the first order Taylor expansion of $\x \mapsto \lossemp[](\x)$ at $\grtrmu$. The approximation error is then given by
\begin{equation}\label{eq:basicequality:hinge}
\quadrterm{\x}{\grtrmu} \coloneqq  \lossemp(\x) - \lossemp(\grtrmu) - \underbrace{\tfrac{1}{m} \sum_{i = 1}^m z_i \sp{\a_i}{\x - \grtrmu}}_{\eqqcolon\multiplterm{\x}{\grtrmu}}, 
\end{equation}
where $\multiplterm{\cdot}{\grtrmu}$ is the ``linearization'' of $\lossemp(\cdot)$ at $\grtrmu$ with
\begin{equation}
	z_i \coloneqq y_i \cdot [\losshng]'(y_i \sp{\a_i}{\grtrmu}) = - y_i \cdot \indset{\intvopcl{-\infty}{1}}(y_i \sp{\a_i}{\grtrmu}), \quad i = 1, \dots, m.
\end{equation}
We will call $\multiplterm{\cdot}{\grtrmu}$ the \emph{multiplier term} in the following because the mapping $\x \mapsto \multiplterm{\x}{\grtrmu}$ indeed forms a multiplier empirical process (cf. \cite{mendelson2014learning,mendelson2014learninggeneral,mendelson2016upper}). In contrast, $\quadrterm{\cdot}{\grtrmu}$ is referred to as the \emph{quadratic term} of the excess risk $\excessloss(\cdot)$.\footnote{This is a slight abuse of terminology because $\losshng$ is not twice differentiable, so that $\quadrterm{\x}{\grtrmu}$ is actually not a quadratic function. However, $\x \mapsto \quadrterm{\x}{\grtrmu}$ mimics the role of a quadratic empirical process, as considered in \cite{mendelson2014learninggeneral,genzel2016estimation} for example, and we will prove similar quadratic lower bounds for it.}
Note that the convexity of $\lossemp(\cdot)$ implies that the quadratic term $\quadrterm{\cdot}{\grtrmu}$ is always non-negative.
Hence, in order to achieve
\begin{equation}\label{eq:multquad}
	\excessloss(\x) = \lossemp[](\x)-\lossemp[](\grtrmu)=\multiplterm{\x}{\grtrmu}+\quadrterm{\x}{\grtrmu} > 0
\end{equation}
for all $\x$ in a fixed subset of $\R^n$, it suffices to show that $\quadrterm{\cdot}{\grtrmu}$ uniformly dominates $\multiplterm{\cdot}{\grtrmu}$ on that specific set. 
To this end, we will treat both terms independently and apply different tools from empirical process theory. The multiplier term can be easily handled by a recent result of Mendelson \cite{mendelson2016upper}, Theorem~\ref{thm:multiplierprocess}, which concerns the uniform deviation of multiplier empirical processes from their mean. For the quadratic process on the other hand, we will employ \emph{Mendelson's small ball method} as stated by Theorem~\ref{thm:msbm} below.

We conclude this overview part with deriving a lower bound for $\quadrterm{\cdot}{\grtrmu}$ that is more convenient to work with.
For this, let us rewrite the quadratic term as follows:
\begin{align}
	\quadrterm{\x}{\grtrmu} &= \begin{aligned}[t] 
		&\tfrac{1}{m} \sum_{i = 1}^m \Big[ (1 - y_i \sp{\a_i}{\x}) \indset{\intvopcl{-\infty}{1}}(y_i \sp{\a_i}{\x}) \\
		&  - (1 - y_i \sp{\a_i}{\grtrmu}) \indset{\intvopcl{-\infty}{1}}(y_i \sp{\a_i}{\grtrmu}) \\
		&  + (y_i\sp{\a_i}{\x} - y_i \sp{\a_i}{\grtrmu}) \indset{\intvopcl{-\infty}{1}}(y_i \sp{\a_i}{\grtrmu}) \Big]
	\end{aligned} \\
	&= \tfrac{1}{m} \sum_{i = 1}^m (1 - y_i \sp{\a_i}{\x}) [\indset{\intvopcl{-\infty}{1}}(y_i \sp{\a_i}{\x}) - \indset{\intvopcl{-\infty}{1}}(y_i \sp{\a_i}{\grtrmu})] \\
	&= \begin{aligned}[t] 
		& \tfrac{1}{m} \sum_{i = 1}^m \Big[ (1 - y_i \sp{\a_i}{\x}) \cdot \indset{\intvclop{1}{\infty}}(y_i \sp{\a_i}{\grtrmu}) \cdot \indset{\intvopcl{-\infty}{1}}(y_i \sp{\a_i}{\x}) \\
		& + (y_i \sp{\a_i}{\x} - 1) \cdot \indset{\intvopcl{-\infty}{1}}(y_i \sp{\a_i}{\grtrmu}) \cdot \indset{\intvclop{1}{\infty}}(y_i \sp{\a_i}{\x}) \Big].
	\end{aligned}\label{eq:quadrterm}
\end{align}
Setting $\h \coloneqq \x - \grtrmu$, the following estimates hold true for every $\xi>0$:
\begin{align}
	& (1 - y_i \sp{\a_i}{\x}) \cdot \indset{\intvclop{1}{\infty}}(y_i \sp{\a_i}{\grtrmu}) \cdot \indset{\intvopcl{-\infty}{1}}(y_i \sp{\a_i}{\x}) \\
	={} & (- y_i \sp{\a_i}{\h} -( y_i \sp{\a_i}{\grtrmu}-1)) \cdot \indset{\intvclop{1}{\infty}}(y_i \sp{\a_i}{\grtrmu}) \cdot \indset{\intvopcl{-\infty}{1}}(y_i \sp{\a_i}{\h} + y_i \sp{\a_i}{\grtrmu}) \\
	\geq{} & (- y_i \sp{\a_i}{\h} -( y_i \sp{\a_i}{\grtrmu}-1)) \cdot \indset{\intvcl{1}{1+2\xi}}(y_i \sp{\a_i}{\grtrmu}) \cdot \indset{\intvopcl{-\infty}{-2\xi}}(y_i \sp{\a_i}{\h}) \\
	\geq{} & (-y_i \sp{\a_i}{\h} - 2\xi)\cdot \indset{\intvcl{1}{1+2\xi}}(y_i \sp{\a_i}{\grtrmu}) \cdot \indset{\intvopcl{-\infty}{-2\xi}}(y_i \sp{\a_i}{\h}) \eqqcolon \empproc_-(\a_i, \h)\label{eq:empneg}
\end{align}
and
\begin{align}
	& (y_i \sp{\a_i}{\x} - 1) \cdot \indset{\intvopcl{-\infty}{1}}(y_i \sp{\a_i}{\grtrmu}) \cdot \indset{\intvclop{1}{\infty}}(y_i \sp{\a_i}{\x}) \\
	={} & (y_i \sp{\a_i}{\h} - ( 1 - y_i \sp{\a_i}{\grtrmu} )) \cdot \indset{\intvopcl{-\infty}{1}}(y_i \sp{\a_i}{\grtrmu}) \cdot \indset{\intvclop{1}{\infty}}(y_i \sp{\a_i}{\h} + y_i \sp{\a_i}{\grtrmu}) \\
	\geq{} & (y_i \sp{\a_i}{\h} - ( 1 - y_i \sp{\a_i}{\grtrmu} )) \cdot \indset{\intvcl{1-2\xi}{1}}(y_i \sp{\a_i}{\grtrmu}) \cdot \indset{\intvclop{2\xi}{\infty}}(y_i \sp{\a_i}{\h}) \\
	\geq{} & (y_i \sp{\a_i}{\h} - 2\xi) \cdot \indset{\intvcl{1-2\xi}{1}}(y_i \sp{\a_i}{\grtrmu}) \cdot \indset{\intvclop{2\xi}{\infty}}(y_i \sp{\a_i}{\h}) \eqqcolon \empproc_+(\a_i, \h).\label{eq:emppos}
\end{align}
Therefore, the resulting \emph{non-negative empirical process} 
\begin{equation}\label{eq:empproc}
\h\mapsto \tfrac{1}{m} \sum_{i = 1}^m \empproc(\a_i, \h) \quad \text{with} \quad \empproc(\a_i, \h)\coloneqq\empproc_+(\a_i, \h) + \empproc_-(\a_i, \h)
\end{equation}
satisfies 
\begin{equation}\label{eq:quadempproc}
	\quadrterm{\x}{\grtrmu}\geq\tfrac{1}{m} \sum_{i = 1}^m \empproc(\a_i, \h)
\end{equation}
for all $\h = \x - \grtrmu$ and $\xi > 0$.

	
	

\subsection{Tools From Empirical Process Theory}

This subsection provides two important tools from empirical process theory which we will apply to control the multiplier and the quadratic term of the excess risk in \eqref{eq:multquad}.
The following concentration inequality by Mendelson investigates the uniform deviation of multiplier processes.
Note that this result even holds true in a more general setting, see \cite[Thm.~4.4]{mendelson2016upper}.
\begin{theorem}[{\cite{mendelson2016upper}}]\label{thm:multiplierprocess}
	Let $L \subset t \ball[2][n]$. For every $i \in [m]$, assume that $\a_i$ is an independent copy of a standard Gaussian random vector $\a \distributed \Normdistr{\vnull}{\I{n}}$, and $z_i$ is an independent copy of a sub-Gaussian random variable $z$ which is not necessarily independent of $\a$.
	There exist numerical constants $C_1, C_2 > 0$ such that for every $u > 0$, the following holds true with probability at least $1 - 2 \cdot e^{-C_1 \cdot u^2} - 2 \cdot e^{-C_1 \cdot m}$:
	\begin{equation}
		\sup_{\h \in L} \abs[\Big]{\tfrac{1}{m}\sum_{i = 1}^m \big(z_i \sp{\a_i}{\h} - \mean[z_i \sp{\a_i}{\h}] \big)} \leq C_2 \cdot \normsubg{z} \cdot \frac{\meanwidth{L} + u \cdot t}{ \sqrt{m}} \ .
	\end{equation}
\end{theorem}

Our second ingredient is \emph{Mendelson's small ball method}, which is a powerful concept to establish lower bounds for non-negative empirical processes.
We state an adaption of Tropp's version in \cite[Prop.~5.1]{tropp2014recovery} below, but it should be emphasized that the original idea is due to Mendelson \cite[Thm.~5.4]{mendelson2014learning}.
\begin{theorem}[Mendelson's small ball method]\label{thm:msbm}
Let $L\subset \R^n$ be a subset, $\a \distributed \Normdistr{\vnull}{\I{n}}$ be a standard Gaussian random vector, and
$F\colon\R \to \R$ be a non-negative (random) contraction that fixes the origin.\footnote{That means, $F$ is $1$-Lipschitz and $F(0) = 0$.}
For every $i \in [m]$, assume that $\a_i \distributed \Normdistr{\vnull}{\I{n}}$ is an independent copy of $\a$ and 
$F_i$ is an independent copy of $F$. Then, for every $\xi > 0$ and $u>0$, the following holds true with probability at least $1-e^{-\tfrac{u^2}{2}}$:
\begin{align}
	\inf_{\h \in L} \tfrac{1}{m} \sum_{i = 1}^m F_i(\sp{\a_i}{ \h})
	&\geq \xi \cdot \Big( Q_{2\xi}(L) - \frac{\tfrac{2}{\xi} \cdot \meanwidth{L}+ u}{\sqrt{m}} \Big), \label{eq:msbm:lowerbound}
\end{align}
where 
\begin{equation}
	Q_{2\xi}(L) \coloneqq \inf_{\h\in L}\prob[F(\sp{\a}{\h})\geq 2\xi]
\end{equation}
denotes the \emph{small ball function} that is associated with $F$.	
\end{theorem}

\subsection{Proof of Theorem~\ref{thm:euclideanball} (Subsets of the Unit Ball)}
\label{sec:proof:unitball}
Throughout this subsection, we assume that the hypotheses of Theorem~\ref{thm:euclideanball} are satisfied, in particular $\sset\subset \ball[2][n]$ and the model conditions of Assumption~\ref{model:measurements},~\ref{model:signal},~and~\ref{model:correlation}.
Following our proof sketch from the beginning of Section~\ref{sec:proofs}, our goal is to show that the excess risk functional $\excessloss(\cdot)=\lossemp[](\cdot)-\lossemp[](\grtrmu)$ is uniformly positive on the boundary of a small Euclidean ball centered at $\grtrmu$, i.e.,
\begin{equation}\label{eq:toshow}
	\inf_{\x \in \sset\intersec (t\S^{n-1}+ \grtrmu)}\excessloss(\x)>0
\end{equation}
for $t > 0$ small enough.
For the sake of readability, we denote by
\begin{equation}\label{eq:localizedsignalset}
	\sset_t\coloneqq\sset\intersec (t\S^{n-1}+ \grtrmu)
\end{equation}
the set of all points in $\sset$ with distance $t$ to $\grtrmu$, and by
\begin{equation}\label{eq:shiftedlocalizedsignalset}
	L_t\coloneqq\sset_t-\grtrmu=(\sset-\grtrmu)\intersec t\S^{n-1}
\end{equation}
its counterpart that arises from a parallel shift of $\grtrmu$ to the origin.
Thus, every point $\x \in \sset_t$ is associated with a directional vector $\h \coloneqq \x-\grtrmu \in L_t$.

We start by proving Lemma~\ref{lem:minexpectedloss}. For this purpose, let us define the convex function
\begin{equation}\label{eq:L}
	R: \R\to \R, \quad s\mapsto R(s)=\mean[\losshng(s f(\gaussianuniv)\gaussianuniv)], \quad \gaussianuniv \distributed \Normdistr{0}{1},
\end{equation}
which corresponds to the expected risk function restricted to the span of $\grtr$. Indeed, for every $s\in \R$, it holds that (cf. Definition~\ref{def:results:unitball:risk})
\begin{equation}\label{eq:expriskonspan}
	\lossexp(s\grtr) = \mean[\losshng(y \sp{\a}{s\grtr})]=\mean{}[\losshng(s y \underbrace{\sp{\a}{\grtr}}_{\mathclap{= g \distributed \Normdistr{0}{1}}})]=\mean[\losshng(s f(\gaussianuniv)\gaussianuniv)]=R(s).
\end{equation} 
\begin{proof}[Proof of Lemma~\ref{lem:minexpectedloss}]
We first show that, by convexity of the hinge loss, there exists an expected risk minimizer that belongs to the span of $\grtr$: For $\x\in \sset$, the orthogonal decomposition
\begin{equation}
	\x= \proj_{\grtr}(\x) + \proj_{\orthcompl{\grtr}}(\x)=\sp{\x}{\grtr}\grtr + \proj_{\orthcompl{\grtr}}(\x),
\end{equation}
allows us to rewrite the expected risk as follows:
\begin{equation}
	\lossexp(\x)=\mean[\losshng(y \sp{\a}{\x})] = \mean[\losshng(y \sp{\x}{\grtr}\sp{\a}{\grtr}+ y\sp{\a}{\proj_{\orthcompl{\grtr}}(\x)})].
\end{equation}
Since the projections of a standard Gaussian random vector onto orthogonal vectors are independent,
we conclude that $\sp{\a}{\proj_{\orthcompl{\grtr}}(\x)}$ is in fact independent of both $y$ and $\sp{\a}{\grtr}$. 
Therefore, Fubini's theorem and Jensen's inequality imply\footnote{A sub-index at the expected value means that the expectation is only taken with respect to this variable.}
\begin{align}
	\lossexp(\x)
		&= \mean_{y, \sp{\a}{\grtr}}\mean_{\sp{\a}{\proj_{\orthcompl{\grtr}}(\x)}}[\losshng(y \sp{\x}{\grtr}\sp{\a}{\grtr}+ y \sp{\a}{\proj_{\orthcompl{\grtr}}(\x)})]\\
		&\geq \mean_{y, \sp{\a}{\grtr}}[\losshng(y \sp{\x}{\grtr}\sp{\a}{\grtr}+ y
		\underbrace{\mean_{\sp{\a}{\proj_{\orthcompl{\grtr}}(\x)}}[\sp{\a}{\proj_{\orthcompl{\grtr}}(\x)}])}_{=0}]\\
		&= \mean[\losshng(y\sp{\x}{\grtr} \sp{\a}{\grtr})]=\mean[\losshng(y \sp{\a}{\sp{\x}{\grtr}\grtr})]=\lossexp(\sp{\x}{\grtr}\grtr).
\end{align}
Since $\sset\subset \ball[2][n]$, it holds that $\abs{\sp{\x}{\grtr}}\leq 1$ for all $\x\in \sset$. Consequently, the minimum of the expected risk function on $\sset$ is bounded from below by the minimum of $R$ on the compact interval $\intvcl{-1}{1}$, i.e.,
\begin{equation}
	\min_{\x\in \sset}\lossexp(\x)\geq \min_{s\in \intvcl{-1}{1}}\lossexp(s\grtr) \stackrel{\eqref{eq:expriskonspan}}{=} \min_{s\in \intvcl{-1}{1}}R(s).
\end{equation}
Computing the first (weak) derivative of $R$,
\begin{align}\label{eq:Lderiv}
	R'(s)=\mean[[\losshng]'(s f(\gaussianuniv)\gaussianuniv)  f(\gaussianuniv)\gaussianuniv]=-\mean[\indset{\intvopcl{-\infty}{1}}(s f(\gaussianuniv)\gaussianuniv) f(\gaussianuniv)\gaussianuniv ],
\end{align} 
and using the correlation assumption $\mean[f(\gaussianuniv)\gaussianuniv]>0$ in \ref{enum:cond_c1}, we observe that
\begin{equation}
	R'(0)=-\mean[f(\gaussianuniv)\gaussianuniv] < 0.
\end{equation}
The convexity of $R$ therefore implies that $R$ attains its minimum on the interval $\intvopcl{0}{1}$, which yields
\begin{equation}
	\min_{\x\in \sset}\lossexp(\x)\geq\min_{s\in \intvopcl{0}{1}}R(s)=\min_{s\in \intvopcl{0}{1}}\lossexp(s\grtr). 
\end{equation}
Since $\convhull\{\vnull, \grtr\} \subset \sset$ by Assumption~\ref{model:signal}, it even follows that
\begin{equation}
	\min_{\x\in \sset}\lossexp(\x)=\min_{s\in \intvopcl{0}{1}}\lossexp(s\x).
\end{equation}
Hence, if $\scalfac > 0$ is a minimizer of $R$ on $\intvopcl{0}{1}$, then we have
\begin{equation}
\min_{\x\in \sset}\lossexp(\x)=\lossexp(\grtrmu).
\end{equation}
\end{proof}

Our next auxiliary result states the relationship between $\scalfac$ and $\lambda = \mean[f(\gaussianuniv)\gaussianuniv]$ that was used in the error bound \eqref{eq:errorboundunitball} of Theorem~\ref{thm:euclideanball}.
\begin{lemma}\label{lem:lowerboundmu} 
We have the following upper bound on $\scalfac^{-1}$:
\begin{equation}
\scalfac^{-1}\lesssim \sqrt{\log(\lambda^{-1})}. 
\end{equation}
More specifically, if $R'(1)\geq 0$, it holds that
\begin{equation}
\scalfac^{-1}\leq \sqrt{2\log(\lambda^{-1})}.
\end{equation}
\end{lemma}
\begin{proof}
Set $X \coloneqq f(\gaussianuniv)\gaussianuniv$ and $s_{\ast} \coloneqq \sqrt{2 \log(\lambda^{-1})}$. Since $\abs{X}\leq \abs{\gaussianuniv}$, we observe that
\begin{equation}\label{eq:lowerboundmu:lambdawelldef} 
	0 \stackrel{\ref{enum:cond_c1}}{<} \lambda=\mean[X]\leq \mean[\abs{\gaussianuniv}]= \sqrt{\tfrac{2}{\pi}}<1,
\end{equation}
and in particular, that $s_{\ast}$ is well-defined and positive.
Moreover, we have
\begin{align}
\mean[\indset{\intvop{s_{\ast}}{\infty}}(X)\cdot X] \leq\mean[\indset{\intvop{s_{\ast}}{\infty}}(\abs{\gaussianuniv})\cdot \abs{\gaussianuniv}]
=\integ[s_{\ast}][\infty]{\tfrac{2x}{\sqrt{2\pi}} e^{-x^2/2}}{dx}
=\sqrt{\tfrac{2}{\pi}}e^{-s_{\ast}^2/2}=\sqrt{\tfrac{2}{\pi}} \lambda.
\end{align}
Hence,
\begin{align}
	\mean[\indset{\intvopcl{-\infty}{s_{\ast}}}(X)\cdot X]=\mean[X]- \mean[\indset{\intvop{s_{\ast}}{\infty}}(X)\cdot X]\geq \lambda-\sqrt{\tfrac{2}{\pi}} \lambda> 0,
\end{align}
which gives
\begin{equation}\label{eq:Lstar}
	R'(\tfrac{1}{s_{\ast}})=-\mean[\indset{\intvopcl{-\infty}{s_{\ast}}}(X)\cdot X]<0.
\end{equation}
Let us make a case distinction for the sign of $R'(1)$:
If $R'(1) \geq 0$, we have $R'(\scalfac) = 0$ because $\scalfac$ is a minimizer of $R$ on $[0,1]$ and $R'(0) < 0$.
In particular, $R'(\scalfac) > R'(\tfrac{1}{s_{\ast}})$ due to \eqref{eq:Lstar}. Since $R'$ is non-decreasing, this implies $\scalfac\geq \tfrac{1}{s_{\ast}}$, which is the claim of Lemma~\ref{lem:lowerboundmu}. 
Finally, if $R'(1)<0$, the convexity of $R$ implies that $\scalfac=1$ is the minimizer of $R$ on $[0,1]$. 
And from $\lambda\leq \sqrt{\frac{2}{\pi}}$, it follows that
\begin{equation}
	\sqrt{\log(\lambda^{-1})}\geq \sqrt{\log(\sqrt{\tfrac{\pi}{2}})} \gtrsim 1=\scalfac^{-1}.
\end{equation}
\end{proof}
The proof of Lemma~\ref{lem:lowerboundmu} reveals the following two important facts about the minimizer $\scalfac$, depending on the sign of $R'(1)$:
\begin{align}
\text{if } R'(1)< 0, &\text{ it follows that } \scalfac=1, \text{ and}\label{eq:valueLneg}\\
\text{if } R'(1)\geq 0, &\text{ it follows that } R'(\scalfac)=0.\label{eq:valueLpos}
\end{align}

The next lemma shows that, as long as $\sset \subset \ball[2][n]$, the multiplier term is always non-negative in expectation. This is very different from the case of general signal sets, where the expected value could become negative (see Remark~\ref{rmk:proofs:unitball}\ref{rmk:proofs:unitball:comparison}).
\begin{lemma}\label{lem:expnoisecor} 
For every $t>0$, it holds that
\begin{equation}
	\inf_{\x\in \sset_t}\mean[\multiplterm{\x}{\grtrmu}]\geq 0.
\end{equation}
According to \eqref{eq:valueLneg} and \eqref{eq:valueLpos}, we can distinguish between two cases:
\begin{itemize}
\item If $R'(1)<0$, we have
	\begin{equation}
		\inf_{\x\in \sset_t} \mean[\multiplterm{\x}{\grtrmu}]\geq -\frac{R'(1)}{2} \cdot t^2>0.
	\end{equation}
\item If $R'(1)\geq 0$, we have
	\begin{equation}
		\inf_{\x\in \sset_t} \mean[\multiplterm{\x}{\grtrmu}]=0.
	\end{equation}
\end{itemize}
\end{lemma}

\begin{proof} 
Let $\x \in \sset_t$. By the orthogonal decomposition
\begin{equation}
\x= \proj_{\grtr}(\x) + \proj_{\orthcompl{\grtr}}(\x)=\sp{\x}{\grtr}\grtr + \proj_{\orthcompl{\grtr}}(\x),
\end{equation}
we compute
\begin{align}
\mean[\multiplterm{\x}{\grtrmu}] &=\mean[\indset{\intvopcl{-\infty}{1}}(y \sp{\a}{\grtrmu}) \cdot y\sp{\a}{\grtrmu-\x} ]\\
								   &=\mean[\indset{\intvopcl{-\infty}{1}}(y \sp{\a}{\grtrmu})\cdot y\sp{\a}{\grtrmu-\sp{\x}{\grtr}\grtr-\proj_{\orthcompl{\grtr}}(\x)}]\\
     								   &=(\scalfac-\sp{\x}{\grtr}) \cdot \mean[\indset{\intvopcl{-\infty}{1}}(y \sp{\a}{\grtrmu}) \cdot y\sp{\a}{\grtr}]\\
								   &\stackrel{\eqref{eq:Lderiv}}{=} -R'(\scalfac) \cdot (\scalfac-\sp{\x}{\grtr})\label{eq:meanofmultiplierterm},
\end{align}
where we have again used the fact that $\sp{\a}{\proj_{\orthcompl{\grtr}}(\x)}$ is independent from $\sp{\a}{\grtrmu}$ and $y$.
As before, we now make a case distinction for the sign of $R'(1)$:
\begin{itemize}
\item 
	$R'(1)<0$: By \eqref{eq:valueLneg}, it holds that $\scalfac=1$, and therefore
	\begin{equation}\label{eq:proofs:unitball:expnoisecor:valueLneg}
		\mean[\multiplterm{\x}{\grtr}]= -R'(1) \cdot (1-\sp{\x}{\grtr}) \geq -\frac{R'(1)}{2} \cdot \underbrace{\lnorm{\x- \grtr}^2}_{= t^2} = -\frac{R'(1)}{2} \cdot t^2> 0
	\end{equation}
	for all $\x \in \sset_t$, where the first inequality is due to $\lnorm{\grtr} = 1$ and $\lnorm{\x} \leq 1$. 
\item 
	$R'(1)\geq 0$: Combining \eqref{eq:valueLpos} and \eqref{eq:meanofmultiplierterm}, we immediately obtain that $\mean[\multiplterm{\x}{\grtrmu}]=0$ for all $\x \in \sset_t$.
\end{itemize}
\end{proof}

The following proposition shows that the quadratic term $\quadrterm{\x}{\grtrmu}$ is not only non-negative but can be uniformly bounded from below on $\sset_t \subset t\S^{n-1}+ \grtrmu$.
\begin{proposition}\label{prop:quadprocesslowerbound} There exist numerical constants $C,C'>0$ such that for every $t\leq \scalfac$ and $\eta \in \intvop{0}{\tfrac{1}{2}}$, the following holds true with probability at least $1-\eta$:
\begin{equation}\label{eq:RSChinge}
\inf_{\x \in \sset_t}\quadrterm{\x}{\grtrmu}\geq 
C' \cdot \lambda \cdot t^2 -  C \cdot t \cdot \frac{\sqrt{\effdim[t]{\sset - \grtrmu}}+ \sqrt{\log(\eta^{-1})}}{\sqrt{m}} \ . 
\end{equation}
\end{proposition}
\begin{remark}\label{rmk:proofs:unitball:rsc}
	Combining \eqref{eq:RSChinge} with the condition \eqref{eq:number_of_measurements}, we obtain a lower bound of the form
	\begin{equation}
		\quadrterm{\x}{\grtrmu} \gtrsim t^2 = \lnorm{\x - \grtrmu}^2 \quad \text{ for all $\x \in \sset_t = \sset\intersec (t\S^{n-1}+ \grtrmu)$.}
	\end{equation}
	According to \cite[Def.~2.2]{genzel2016estimation}, this means that the empirical risk function $\lossemp(\cdot)$ satisfies \emph{restricted strong convexity} (\emph{RSC}) at scale $t$ with respect to $\grtrmu$ and $\sset$.
	This observation is quite remarkable because the hinge loss $\losshng$ does by far not fulfill sufficient criteria known from the literature, e.g., in \cite[Sec.~II.C]{genzel2016estimation} or \cite[Sec.~4]{mendelson2014learninggeneral}.
	More specifically, since $\losshng$ is piecewise linear, its second derivative does only exist in a distributional sense. Proposition~\ref{prop:quadprocesslowerbound} therefore indicates that one can even expect RSC (at certain scales $t$) if the ``curvature energy'' of the loss function is concentrated in a single point.
\end{remark}
\begin{proof}
According to \eqref{eq:quadempproc}, the quadratic term can be uniformly bounded from below by a simplified non-negative empirical process \eqref{eq:empproc}, i.e.,
\begin{equation}
\inf_{\x \in \sset_t}\quadrterm{\x}{\grtrmu}\geq \inf_{\h \in L_{t}} \tfrac{1}{m} \sum_{i = 1}^m \empproc(\a_i, \h),
\end{equation}
where $L_t=\sset_t-\grtrmu$.
We now apply Theorem~\ref{thm:msbm} with $\xi=\tfrac{t}{6}$, $L=L_t$, $u=\sqrt{2\log(\eta^{-1})}$, and
\begin{align}
F(s)= {} &(-y s - \tfrac{t}{3})\cdot \indset{\intvcl{1}{1+\tfrac{t}{3}}}(y \sp{\a}{\grtrmu}) \cdot \indset{\intvopcl{-\infty}{-\tfrac{t}{3}}}(y s) \\
& \quad + (y s - \tfrac{t}{3}) \cdot \indset{\intvcl{1-\tfrac{t}{3}}{1}}(y \sp{\a}{\grtrmu}) \cdot \indset{\intvclop{\tfrac{t}{3}}{\infty}}(y s).
\end{align}
It is not hard to see that $F$ is indeed a non-negative contraction with $F(0)=0$. Since $t^{-1}\meanwidth{L_t}\leq \sqrt{\effdim[t]{\sset - \grtrmu}}$ and 
$F_i(\sp{\a_i}{\h})=\empproc(\a_i, \h)$, the assertion of Theorem~\ref{thm:msbm} implies that, with probability at least $1-\eta$, it holds that
\begin{equation}
\inf_{\h \in L_{t}} \tfrac{1}{m} \sum_{i = 1}^m \empproc(\a_i, \h)
\geq \tfrac{t}{6} \cdot \Big( Q_{t/3}(L_{t}) - \frac{12\sqrt{\effdim[t]{\sset - \grtrmu}}+ \sqrt{2\log(\eta^{-1})}}{\sqrt{m}} \Big).
\end{equation}

It remains to show that the small ball function associated with $F$ satisfies the lower bound
\begin{equation}
Q_{t/3}(L_{t}) = \inf_{\h\in L_{t}} \prob[\empproc(\a, \h)\geq \tfrac{t}{3}] \gtrsim t\cdot \lambda.
\end{equation}
For this purpose, we divide $L_t$ into two disjoint subsets in the following way:
\begin{equation}
L_t =( \sset -\grtrmu) \intersec t \S^{n-1} = L_{t}^{+}\union  L_{t}^{-},
\end{equation}
where $L_{t}^{+} \coloneqq \{\h\in L_{t} \suchthat \sp{\h}{\grtr}\geq 0 \}$ and $L_{t}^{-} \coloneqq \{\h\in L_{t} \suchthat \sp{\h}{\grtr}< 0 \}$. 
Let us first consider the case of $\h \in L_{t}^{+}$ and bound $\prob[\empproc(\a, \h)\geq \tfrac{t}{3}]$ from below:
There exist $d,e\in \intvcl{0}{t}$ and $\x' \in \R^n$ such that
\begin{equation}
	\h=\sp{\h}{\grtr}\grtr + \proj_{\orthcompl{\grtr}}(\h)=d\grtr+e\x'
\end{equation}
with $d^2+e^2=t^2$, $\sp{\grtr}{\x'}=0$, and $\lnorm{\grtr}=\lnorm{\x'}=1$. 
Hence, we obtain 
\begin{align}
	\prob[\empproc_+(\a, \h) \geq \tfrac{t}{3}]&=\prob[y\sp{\a}{\h}\geq \tfrac{2t}{3}, y\sp{\a}{\grtrmu}\in \intvcl{1-\tfrac{t}{3}}{1}]\\
		&=\prob[d y\sp{\a}{\grtr}+ ey\sp{\a}{\x'}\geq \tfrac{2t}{3}, y\sp{\a}{\grtr}\in \intvcl{\tfrac{1}{\scalfac}-\tfrac{t}{3\scalfac}}{\tfrac{1}{\scalfac}}]\\
		&\geq \prob[d (\tfrac{1}{\scalfac}-\tfrac{t}{3\scalfac})+ ey\sp{\a}{\x'}\geq \tfrac{2t}{3}, y\sp{\a}{\grtr}\in \intvcl{\tfrac{1}{\scalfac}-\tfrac{t}{3\scalfac}}{\tfrac{1}{\scalfac}}] \\
		&= \prob[ey\sp{\a}{\x'}\geq \tfrac{2t}{3}- d (\tfrac{1}{\scalfac}-\tfrac{t}{3\scalfac})] \cdot \prob[y\sp{\a}{\grtr}\in \intvcl{\tfrac{1}{\scalfac}-\tfrac{t}{3\scalfac}}{\tfrac{1}{\scalfac}}],
\end{align}
where we have again used that the components of an orthogonal decomposition of a standard Gaussian vector are independent.
Next, we show that
\begin{align}
\prob[ey\sp{\a}{\x'}\geq \tfrac{2t}{3}- d (\tfrac{1}{\scalfac}-\tfrac{t}{3\scalfac})]\geq \prob[\gaussianuniv \geq 1] \quad \text{for $\gaussianuniv\distributed \Normdistr{0}{1}$.}
\label{eq:constpart}
\end{align}
From $t\leq \scalfac\leq 1$, it follows that $\tfrac{2\scalfac}{3}+ \tfrac{t}{3}\leq 1$, which is equivalent to $1-\tfrac{1}{\scalfac}+ \tfrac{t}{3\scalfac}\leq \tfrac{1}{3}$. Since $0 \leq d \leq t$, this leads to
$d-d(\tfrac{1}{\scalfac} - \tfrac{t}{3\scalfac})\leq \tfrac{t}{3}$, or equivalently, $\tfrac{2t}{3}- d (\tfrac{1}{\scalfac}-\tfrac{t}{3\scalfac})\leq t-d$. Due to $t-d\leq \sqrt{t^2-d^2}=e$, we therefore obtain
$\tfrac{2t}{3}- d (\tfrac{1}{\scalfac}-\tfrac{t}{3\scalfac})\leq e$, so that
\begin{align}
\prob[ey\sp{\a}{\x'} \geq \tfrac{2t}{3}- d (\tfrac{1}{\scalfac}-\tfrac{t}{3\scalfac})]&\geq
\prob[ey\sp{\a}{\x'}\geq e]\geq \prob[y\sp{\a}{\x'}\geq 1]\\
&=\prob[\sp{\a}{\x'}\geq 1]=\prob[\gaussianuniv \geq 1],
 \end{align}
where we have particularly used that the binary variable $y \in \{-1,+1\}$ is independent of $\gaussianuniv = \sp{\a}{\x'} \distributed \Normdistr{0}{1}$.
In conclusion, we have
\begin{equation}
	\prob[\empproc(\a, \h)\geq \tfrac{t}{3}]\geq \prob[\empproc_+(\a, \h) \geq \tfrac{t}{3}] \geq \prob[\gaussianuniv \geq 1] \cdot \prob[y\sp{\a}{\grtr}\in \intvcl{\tfrac{1}{\scalfac}-\tfrac{t}{3\scalfac}}{\tfrac{1}{\scalfac}}]
\end{equation}
for all $\h \in L_{t}^{+}$.

To handle the case of $\h \in L_{t}^{-}$, let us consider the decomposition
\begin{equation}
-\h=\sp{-\h}{\grtr}\grtr + \proj_{\orthcompl{\grtr}}(-\h)=d\grtr+e\x'
\end{equation} 
with $d,e\in \intvcl{0}{t}$ such that $d^2+e^2=t^2$, $\sp{\grtr}{\x'}=0$, and $\lnorm{\grtr}=\lnorm{\x'}=1$. We proceed as before:
\begin{align}
\prob[\empproc_-(\a, \h) \geq \tfrac{t}{3}]&=\prob[-y\sp{\a}{\h}\geq \tfrac{2t}{3}, y\sp{\a}{\grtrmu}\in \intvcl{1}{1+ \tfrac{t}{3}}]\\
&=\prob[d y\sp{\a}{\grtr}+ ey\sp{\a}{\x'}\geq \tfrac{2t}{3}, y\sp{\a}{\grtr}\in \intvcl{\tfrac{1}{\scalfac}}{\tfrac{1}{\scalfac}+ \tfrac{t}{3\scalfac}}]\\
&\geq \prob[ \tfrac{d}{\scalfac}+ ey\sp{\a}{\x'}\geq \tfrac{2t}{3}, y\sp{\a}{\grtr}\in \intvcl{\tfrac{1}{\scalfac}}{\tfrac{1}{\scalfac}+ \tfrac{t}{3\scalfac}}]\\
&=\prob[ ey\sp{\a}{\x'}\geq \tfrac{2t}{3}- \tfrac{d}{\scalfac}] \cdot \prob[y\sp{\a}{\grtr}\in \intvcl{\tfrac{1}{\scalfac}}{\tfrac{1}{\scalfac}+ \tfrac{t}{3\scalfac}}].
\end{align}
By \eqref{eq:constpart} again, it holds that
\begin{align}
\prob[ ey\sp{\a}{\x'}\geq \tfrac{2t}{3}- \tfrac{d}{\scalfac}]\geq  
\prob[ey\sp{\a}{\x'}\geq \tfrac{2t}{3}- d (\tfrac{1}{\scalfac}-\tfrac{t}{3\scalfac})]\geq \prob[\gaussianuniv \geq 1],
\end{align}
and consequently
\begin{equation}
\prob[\empproc(\a, \h)\geq \tfrac{t}{3}]\geq \prob[\empproc_-(\a, \h) \geq \tfrac{t}{3}]
\geq \prob[\gaussianuniv \geq 1] \cdot \prob[y\sp{\a}{\grtr}\in \intvcl{\tfrac{1}{\scalfac}}{\tfrac{1}{\scalfac}+ \tfrac{t}{3\scalfac}}]
\end{equation} 
for every $\h \in L_{t}^{-}$.

To this point, we have shown that
\begin{equation}
	Q_{t/3}(L_{t}) \geq\prob[\gaussianuniv \geq 1] \cdot \min \Big\{  \prob[f(g)g\in \intvcl{\tfrac{1}{\scalfac}-\tfrac{t}{3\scalfac}}{\tfrac{1}{\scalfac}}], \prob[f(g)g\in \intvcl{\tfrac{1}{\scalfac}}{\tfrac{1}{\scalfac}+ \tfrac{t}{3\scalfac}}]  \Big\},
\end{equation}
where $\gaussianuniv = \sp{\a}{\grtr} \distributed \Normdistr{0}{1}$.
Note that the correlation condition \ref{enum:cond_c2} of Assumption~\ref{model:correlation} is equivalent to 
\begin{equation}
\prob[f(\gaussianuniv)\gaussianuniv\geq 0\suchthat \abs{\gaussianuniv}]\geq \prob[f(\gaussianuniv)\gaussianuniv\leq 0\suchthat \abs{\gaussianuniv}]\quad\text{ (a.s.).}
\end{equation}
In particular, this means that
\begin{equation}
\prob[f(\gaussianuniv)\gaussianuniv\geq 0, \abs{\gaussianuniv}\in \mathcal{I}]\geq \prob[f(\gaussianuniv)\gaussianuniv\leq 0, \abs{\gaussianuniv}\in \mathcal{I}]
\end{equation}
for every interval $\mathcal{I}\subset \R$.
Selecting $\mathcal{I}=\intvcl{\tfrac{1}{\scalfac}-\tfrac{t}{3\scalfac}}{\tfrac{1}{\scalfac}}$, this implies
\begin{align}
\prob[f(g)g\in \intvcl{\tfrac{1}{\scalfac}-\tfrac{t}{3\scalfac}}{\tfrac{1}{\scalfac}}] &= \prob[f(\gaussianuniv)\gaussianuniv\geq 0, \abs{\gaussianuniv}\in \intvcl{\tfrac{1}{\scalfac}-\tfrac{t}{3\scalfac}}{\tfrac{1}{\scalfac}}] \\
&\geq \prob[f(\gaussianuniv)\gaussianuniv\leq 0, \abs{\gaussianuniv}\in \intvcl{\tfrac{1}{\scalfac}-\tfrac{t}{3\scalfac}}{\tfrac{1}{\scalfac}}] = \prob[-f(g)g\in \intvcl{\tfrac{1}{\scalfac}-\tfrac{t}{3\scalfac}}{\tfrac{1}{\scalfac}}],
\end{align}
and therefore
\begin{equation}
\prob[ \abs{\gaussianuniv} \in \intvcl{\tfrac{1}{\scalfac}-\tfrac{t}{3\scalfac}}{\tfrac{1}{\scalfac}}]=
\prob[ \abs{f(\gaussianuniv)\gaussianuniv} \in \intvcl{\tfrac{1}{\scalfac}-\tfrac{t}{3\scalfac}}{\tfrac{1}{\scalfac}}]
\leq 2 \prob[f(g)g\in \intvcl{\tfrac{1}{\scalfac}-\tfrac{t}{3\scalfac}}{\tfrac{1}{\scalfac}}].
\end{equation}
Similarly, for $\mathcal{I}=\intvcl{\tfrac{1}{\scalfac}}{\tfrac{1}{\scalfac}+ \tfrac{t}{3\scalfac}}$, we obtain
\begin{equation}
\prob[ \abs{\gaussianuniv} \in \intvcl{\tfrac{1}{\scalfac}}{\tfrac{1}{\scalfac}+ \tfrac{t}{3\scalfac}}]
\leq 2 \prob[f(g)g\in \intvcl{\tfrac{1}{\scalfac}}{\tfrac{1}{\scalfac}+ \tfrac{t}{3\scalfac}}].
\end{equation}
Elementary estimates now lead to
\begin{align}
\prob[f(g)g\in \intvcl{\tfrac{1}{\scalfac}-\tfrac{t}{3\scalfac}}{\tfrac{1}{\scalfac}}]&\geq \tfrac{1}{2} \prob[\abs{\gaussianuniv}\in \intvcl{\tfrac{1}{\scalfac}-\tfrac{t}{3\scalfac}}{\tfrac{1}{\scalfac}}]= \integ[\tfrac{1}{\scalfac}-\tfrac{t}{3\scalfac}][\tfrac{1}{\scalfac}] {\tfrac{1}{\sqrt{2\pi}} e^{-s^2/2}}{ds}\\
&\geq \tfrac{1}{\sqrt{2\pi}} \cdot \frac{t}{3\scalfac} \cdot e^{-1/(2\scalfac^2)}\geq  \tfrac{1}{\sqrt{2\pi}}  \cdot \frac{t}{3} \cdot e^{-1/(2\scalfac^2)}
\end{align}
and
\begin{align}
\prob[f(\gaussianuniv)\gaussianuniv\in \intvcl{\tfrac{1}{\scalfac}}{\tfrac{1}{\scalfac}+ \tfrac{t}{3\scalfac}}]
&\geq \tfrac{1}{2}\prob[ \abs{\gaussianuniv} \in \intvcl{\tfrac{1}{\scalfac}}{\tfrac{1}{\scalfac}+ \tfrac{t}{3\scalfac}}]=\integ[\tfrac{1}{\scalfac}][\tfrac{1}{\scalfac}+\tfrac{t}{3\scalfac}] {\tfrac{1}{\sqrt{2\pi}} e^{-s^2/2}}{ds}\\
\geq \integ[\tfrac{1}{\scalfac}][\tfrac{1}{\scalfac}+\tfrac{t}{3}] {\tfrac{1}{\sqrt{2\pi}} e^{-s^2/2}}{ds}&\geq \tfrac{1}{\sqrt{2\pi}}  \cdot \frac{t}{3} \cdot e^{-\left(\tfrac{1}{\scalfac}+\tfrac{t}{3}\right)^2/2}\geq \tfrac{1}{\sqrt{2\pi}} \cdot \frac{t}{3} \cdot e^{-1/(2\scalfac^2)} \cdot e^{-1/2},
\end{align}
where the assumption $t\leq \scalfac$ was used in the last step. 
Finally, we show that $e^{-1/(2\scalfac^2)}\gtrsim \lambda$, which would imply the claim $Q_{t/3}(L_{t})\gtrsim t\cdot \lambda$. 
If $R'(1)\geq 0$, this bound directly follows from Lemma~\ref{lem:lowerboundmu}. On the other hand, if $R'(1)<0$, \eqref{eq:valueLneg} states that $\scalfac=1$. Combined with the estimate of \eqref{eq:lowerboundmu:lambdawelldef}, it again follows that $e^{-1/(2\scalfac^2)}\gtrsim \sqrt{\tfrac{2}{\pi}}\geq \lambda$.
\end{proof}

We are ready to prove Theorem~\ref{thm:euclideanball}.
\begin{proof}[Proof of Theorem~\ref{thm:euclideanball}]
Using the triangle inequality and the decomposition of \eqref{eq:multquad}, we can derive the following lower bound for the excess risk:
\begin{align}
	\excessloss(\x)\geq \mean[\multiplterm{\x}{\grtrmu}] - \abs[\big]{\multiplterm{\x}{\grtrmu}- \mean[\multiplterm{\x}{\grtrmu}]} + \quadrterm{\x}{\grtrmu}.
\end{align}
Next, we take the infimum over the localized signal set $\sset_t=\sset\intersec (t\S^{n-1}+ \grtrmu)$ and obtain a uniform lower bound: 
\begin{equation}\label{eq:threesummandsunitball}
\inf_{\x \in \sset_t}\excessloss(\x)\geq \inf_{\x \in \sset_t}\mean[\multiplterm{\x}{\grtrmu}]
- \sup_{\x \in \sset_t} \abs[\big]{\multiplterm{\x}{\grtrmu}- \mean[\multiplterm{\x}{\grtrmu}]}
+\inf_{\x \in \sset_t}\quadrterm{\x}{\grtrmu}.
\end{equation}
Let us treat each of the three summands in \eqref{eq:threesummandsunitball} separately. According to Lemma~\ref{lem:expnoisecor}, the first term in \eqref{eq:threesummandsunitball} is non-negative.
Recalling that $z_i=-y_i \cdot \indset{\intvopcl{-\infty}{1}}(y_i \sp{\a_i}{\grtrmu})$ and $L_t=\sset_t-\grtrmu \subset t \ball[2][n]$, the second term takes the form
\begin{align}
\sup_{\x \in \sset_t} \abs[\big]{\multiplterm{\x}{\grtrmu}- \mean[\multiplterm{\x}{\grtrmu}]}
=\sup_{\h \in L_t} \abs[\Big]{ \tfrac{1}{m} \sum_{i = 1}^m \big[z_i \sp{\a_i}{\h} - \mean[z_i \sp{\a_i}{\h}]\big]}.
\end{align}
Since $\abs{z_i} \leq 1$, the multipliers $z_i$ are sub-Gaussian random variables with $\normsubg{z_i}\leq (\log(2))^{-1/2}$.
Hence, Theorem~\ref{thm:multiplierprocess} can be applied with $L = L_t$, and due to $m\gtrsim \log(\eta^{-1})$, it states that
\begin{align}
	\sup_{\h \in L_t} \abs[\Big]{ \tfrac{1}{m} \sum_{i = 1}^m \big[z_i \sp{\a_i}{\h} - \mean[z_i \sp{\a_i}{\h}]\big] }&\leq \tilde{C} \cdot
	\frac{\meanwidth{L_t} + t \cdot \sqrt{\log(\eta^{-1})}}{ \sqrt{m}}\\
	&\leq \tilde{C} \cdot t \cdot \frac{ \sqrt{d_t(\sset-\grtrmu)} + \sqrt{\log(\eta^{-1})}}{ \sqrt{m}}
\end{align}
with probability at least $1-\tfrac{\eta}{2}$, where $\tilde{C}>0$ is an appropriate numerical constant.
In order to bound the third term in \eqref{eq:threesummandsunitball}, we simply invoke Proposition~\ref{prop:quadprocesslowerbound}, after which
\begin{equation}
\inf_{\x \in \sset_t}\quadrterm{\x}{\grtrmu}\geq 
C' \cdot \lambda \cdot t^2 -  C'' \cdot t \cdot \frac{\sqrt{\effdim[t]{\sset - \grtrmu}}+ \sqrt{\log(\eta^{-1})}}{\sqrt{m}}
\end{equation}
with probability at least $1-\tfrac{\eta}{2}$ for appropriate numerical constants $C', C'' > 0$.

In total, the following holds true with probability at least $1-\eta$:
\begin{equation}\label{eq:excessriskpositiveunitball}
\inf_{\x \in \sset_t}\excessloss(\x)\geq C' \cdot \lambda \cdot t^2    -  (C''+\tilde{C}) \cdot t \cdot \frac{\sqrt{\effdim[t]{\sset - \grtrmu}}+ \sqrt{\log(\eta^{-1})}}{\sqrt{m}}>0,
\end{equation}
where the positivity directly follows from the initial assumption \eqref{eq:number_of_measurements}, that is,
\begin{equation}
	m\gtrsim \lambda^{-2}\cdot t^{-2} \cdot \max\{\effdim[t]{\sset - \grtrmu}, \log(\eta^{-1}) \}.
\end{equation}
The lower bound of \eqref{eq:excessriskpositiveunitball} now implies that every minimizer $\solu$ of \eqref{eq:estimator} satisfies $\lnorm{\grtrmu - \solu} \leq t$.
Indeed, if we would have $\lnorm{\grtrmu - \solu} > t$, the line segment between $\grtrmu$ and $\solu$ must intersect $\sset_t$, since the signal set $\sset$ is convex.
But this would contradict the conclusion of Fact~\ref{fact}.

Finally, we observe that $\abs{\lnorm{\solu}-\scalfac}=\abs{\lnorm{\solu}-\lnorm{\grtrmu}}\leq \lnorm{\solu-\grtrmu}\leq t$. In particular, $0<\scalfac-t\leq\lnorm{\solu}$, so that we can estimate
\begin{align}
\lnorm[auto]{\grtr - \frac{\solu}{\lnorm{\solu}}}\leq \lnorm[auto]{\grtr-\frac{\solu}{\scalfac}}+
\lnorm[auto]{\frac{\solu}{\scalfac}-\frac{\solu}{\lnorm{\solu}}}\leq \frac{t}{\scalfac}+\frac{\abs{\lnorm{\solu}-\scalfac}}{\scalfac \lnorm{\solu}} \cdot \lnorm{\solu}\leq \frac{2t}{\scalfac} \ .
\end{align}
Rescaling in $t$ precisely yields the claim of Theorem~\ref{thm:euclideanball}.
\end{proof}

\begin{remark}
\begin{rmklist}
\item \label{rmk:proofs:unitball:comparison}
	A careful study of the above proofs reveals that we have used the unit-ball assumption $\sset \subset \ball[2][n]$ at two points:
	firstly, in the proof of Lemma~\ref{lem:minexpectedloss} to show that the expected risk minimizer lies on the span of $\grtr$, and secondly, in \eqref{eq:proofs:unitball:expnoisecor:valueLneg} to verify the non-negativity of the multiplier term on $\sset_t$ in expectation.
	Especially the latter conclusion was crucial to ensure that the excess risk is positive on $\sset_t$, since the first term of \eqref{eq:threesummandsunitball} can be ignored in this way.
	
	Interestingly, as long as $\scalfac = 1$, one could even apply the quadratic lower bound of \eqref{eq:proofs:unitball:expnoisecor:valueLneg} to establish positivity of the excess risk on $\sset_t$, while completely disregarding the quadratic term in \eqref{eq:threesummandsunitball}.
	With other words, the linear part of the hinge loss is ``sufficient'' for signal recovery and RSC is not needed at all (cf. Remark~\ref{rmk:proofs:unitball:rsc}).
	A similar observation was already made by Plan and Vershynin in \cite{plan2013robust}, where they analyzed the performance of a simple linear estimator in $1$-bit compressed sensing.
	
	Such a simple argumentation however does not work out below in the case of general convex signal sets.
	Indeed, $\mean[\multiplterm{\x}{\grtrmu}]$ may become negative at certain points $\x$. 
	Consequently, in order to still ensure that $\excessloss(\x) > 0$, the adverse impact of the multiplier process needs to be ``compensated'' by the size of the quadratic term $\quadrterm{\x}{\grtrmu}$, which in turn requires a refined error analysis.
%
\item \label{rmk:proofs:unitball:starshaped}
	The above proof only requires that $\sset$ contains every line segment from $\grtrmu$ to any other point in $\sset$. Hence, similar to the results of Plan and Vershynin in \cite{plan2015lasso}, Theorem~\ref{thm:euclideanball} remains valid if $\sset - \grtrmu$ is just a \emph{star-shaped} set, i.e., $\lambda(\sset - \grtrmu) \subset \sset - \grtrmu$ for all $0 \leq \lambda \leq 1$.
	While this allows us to drop the convexity of $\sset$ in Assumption~\ref{model:signal}, such a relaxation is only of limited practical interest, as the ``anchor point'' $\grtrmu$ is still unknown.
	\qedhere
\end{rmklist} \label{rmk:proofs:unitball}
\end{remark}

\subsection{Proofs of Theorem~\ref{thm:results:generalset:global} and Theorem~\ref{thm:results:generalset:local} (General Convex Sets)}
\label{subsec:proofs:generalset}

Throughout this part, we assume that the model hypotheses of Subsection~\ref{subsec:results:generalset} hold true, in particular Assumption~\ref{model:results:generalset}.
According to our roadmap from the beginning of Section~\ref{sec:proofs}, our proof builds upon analyzing the excess risk functional
\begin{equation}\label{eq:proofs:generalset:excesslossdecomp}
	\x \mapsto \excessloss(\x) = \lossemp[](\x)-\lossemp[](\grtrmu)=\multiplterm{\x}{\grtrmu}+\quadrterm{\x}{\grtrmu}
\end{equation}
in a certain local neighborhood of $\grtrmu$.
In contrast to the strategy of Subsection~\ref{sec:proof:unitball}, an exclusion of a spherical intersection $\tilde{\sset} \coloneqq \ssetmu \intersec (\S^{n-1} + \grtrmu)$ as potential minimizers of \eqref{eq:estimatortuned} turns out to be too restrictive.
More specifically, if $\x = \grtrmu + \h \in \tilde{\sset}$ with $\abs{\sp{\grtr}{\h}} \approx 1$, the quadratic term $\quadrterm{\x}{\grtrmu}$ would become too small compared to $\abs{\multiplterm{\x}{\grtrmu}}$ (cf. Remark~\ref{rmk:proofs:unitball}\ref{rmk:proofs:unitball:comparison}). 
Therefore, we cannot expect that $\excessloss(\x)$ is strictly positive in these cases.

This issue can be resolved by carefully enlarging the set of ``feasible'' points in $\ssetmu$, i.e., those points at which the excess risk is not guaranteed to be positive.
More precisely, we will show that, with high probability, every minimizer $\solu$ of \eqref{eq:estimatortuned} belongs to the cylindrical tube defined in \eqref{eq:results:generalset:cylinder}, namely
\begin{equation}\label{eq:proofs:generalset:cylinder}
	\cyl(\grtr, \scalfac) = \{ \x \in \R^n \suchthat \lnorm{\x - \sp{\x}{\grtr}\grtr} \leq 1, \sp{\x}{\grtr} \geq \tfrac{\scalfac}{2} \}.
\end{equation}
See Figure~\ref{fig:results:generalset:cylinder} for more details on the geometric shape of $\cyl(\grtr, \scalfac)$.
The following lemma collects some basic properties of $\cyl(\grtr, \scalfac)$.
\begin{lemma}
	\begin{rmklist}
	\item\label{lem:proofs:generalset:cylinder:cylconvex}
		$\cyl(\grtr, \scalfac)$ is a convex set and $\grtrmu$ is an interior point.
	\item\label{lem:proofs:generalset:cylinder:errorbound}
		Every $\x \in \cyl(\grtr, \scalfac)$ satisfies
		\begin{equation}
			\lnorm[auto]{\grtr - \frac{\x}{\lnorm{\x}}} \leq \frac{3}{\scalfac} \ .
		\end{equation}
	\item\label{lem:proofs:generalset:cylinder:cylboundary}
		The boundary of $\cyl(\grtr, \scalfac)$ can be written as the union of its side and base:
		\begin{align}
		\boundary \cyl(\grtr, \scalfac) = {} & \underbrace{\{ \x \in \R^n \suchthat \lnorm{\x - \sp{\x}{\grtr}\grtr} = 1, \sp{\x}{\grtr} \geq \tfrac{\scalfac}{2} \}}_{\eqqcolon \cyl_{1}(\grtr, \scalfac) \quad \text{``side''}} \\
 		{} & \union \underbrace{\{ \x \in \R^n \suchthat \lnorm{\x - \sp{\x}{\grtr}\grtr} \leq 1, \sp{\x}{\grtr} = \tfrac{\scalfac}{2} \}}_{\eqqcolon \cyl_2(\grtr, \scalfac) \quad \text{``base''}}.
	\end{align}
	\end{rmklist}
	\label{lem:proofs:generalset:cylinder}
\end{lemma}
\begin{proof}
	The statements of \ref{lem:proofs:generalset:cylinder:cylconvex} and \ref{lem:proofs:generalset:cylinder:cylboundary} are elementary geometric properties.
	Let us verify \ref{lem:proofs:generalset:cylinder:errorbound}.
	For every $\x \in \cyl(\grtr, \scalfac)$, the definition of $\cyl(\grtr, \scalfac)$ implies that $\lnorm{\x}^2 \leq 1 + \sp{\x}{\grtr}^2$ as well as $\lnorm{\x} \geq \sp{\x}{\grtr} \geq \scalfac / 2$.
	Hence,
	\begin{align}
		\lnorm[auto]{\grtr - \frac{\x}{\lnorm{\x}}}^2 &= 2 \Big(1 - \frac{\sp{\x}{\grtr}}{\lnorm{\x}}\Big) = 2 \Big( \frac{\lnorm{\x} - \sp{\x}{\grtr}}{\lnorm{\x}}\Big) \\
		&\leq 2 \Big( \frac{\sqrt{1+\sp{\x}{\grtr}^2} - \sp{\x}{\grtr}}{\sp{\x}{\grtr}}\Big) \leq \frac{2}{\sp{\x}{\grtr}^2} \leq \frac{8}{\scalfac^2}.
	\end{align}
\end{proof}


According to Lemma~\ref{lem:proofs:generalset:cylinder}\ref{lem:proofs:generalset:cylinder:errorbound}, every minimizer $\solu$ of \eqref{eq:estimatortuned} that is contained in $\cyl(\grtr, \scalfac)$ already satisfies the error bound of Theorem~\ref{thm:results:generalset:global} and Theorem~\ref{thm:results:generalset:local}.
Therefore, it suffices to show that every point $\x \in \ssetmu$ that does not belong to $\cyl(\grtr, \scalfac)$ cannot solve \eqref{eq:estimatortuned}.

The next lemma forms the main technical component of our proof.
It provides a uniform lower bound for the excess risk $\excessloss(\cdot)$ on cylindrical intersections with $\boundary \cyl(\grtr, \scalfac)$.
\begin{lemma}\label{lem:proofs:generalset:excessloss}
	Let $L \subset \R^n$ be a bounded subset. Moreover, we define
	\begin{equation}
		\tilde{\sset} \coloneqq L \intersec \boundary \cyl(\grtr, \scalfac) \quad \text{and} \quad \tilde{L} \coloneqq \tilde{\sset} - \grtrmu \subset t_0 \ball[2][n]
	\end{equation}
	where $t_0 \coloneqq \max\{1 , \rad((L \intersec \cyl(\grtr, \scalfac)) - \grtrmu)\}$.
	There exist numerical constants $C_1, C_2 > 0$ such that for every $\eta \in \intvop{0}{\tfrac{1}{2}}$, the following holds true with probability at least $1 - \eta$:
	If $m \gtrsim \log(\eta^{-1})$ and $\scalfac \gtrsim 1$, then the excess risk satisfies
	\begin{align}
		\excessloss(\x) &\gtrsim t_0 \cdot \Big( \frac{1}{t_0 \scalfac} - C_1 \cdot \frac{ \max\{0,\sign(\theta_{\x})\}}{\scalfac^2} - C_2 \cdot \frac{\tfrac{1}{t_0}\meanwidth{\tilde{L}} + \sqrt{\log(\eta^{-1})}}{\sqrt{m}} \Big) \label{eq:proofs:generalset:excesslossbound}
	\end{align}
	for every $\x \in \tilde{\sset}$, where $\theta_{\x} \coloneqq \sp{\x - \grtrmu}{\grtr}$.
\end{lemma}

\begin{proof}
Let us first analyze the multiplier term of the decomposition in \eqref{eq:proofs:generalset:excesslossdecomp}.
Analogously to the proof of Theorem~\ref{thm:euclideanball} in Subsection~\ref{sec:proof:unitball}, we apply Theorem~\ref{thm:multiplierprocess} with $z=-y \cdot  \indset{\intvopcl{-\infty}{1}}(y \sp{\a}{\grtrmu})$ and $L = \tilde{L} \subset t_0 \ball[2][n]$.
Combined with the assumption $m \gtrsim \log(\eta^{-1})$, Theorem~\ref{thm:multiplierprocess} then implies that the following bound holds true with probability at least $1 - \tfrac{\eta}{2}$:
\begin{align}
	\multiplterm{\x}{\grtrmu} &\geq \mean[\multiplterm{\x}{\grtrmu}] - C_2 \cdot \normsubg{z} \cdot \frac{\meanwidth{\tilde{L}} + t_0 \cdot \sqrt{\log(\eta^{-1})}}{\sqrt{m}} \\
	&= t_0 \cdot  \Big( \tfrac{1}{t_0} \mean[\multiplterm{\x}{\grtrmu}] - C_2 \cdot \normsubg{z} \cdot \frac{\tfrac{1}{t_0}\meanwidth{\tilde{L}} + \sqrt{\log(\eta^{-1})}}{\sqrt{m}} \Big)
\end{align}
for every $\x \in \tilde{\sset}$.
Since $y = \sign(\sp{\a}{\grtr})$, the expected multiplier term simplifies as follows (cf. \eqref{eq:meanofmultiplierterm}):
\begin{align}
\mean[\multiplterm{\x}{\grtrmu}] &=\mean[\indset{\intvopcl{-\infty}{1}}(y \sp{\a}{\grtrmu}) \cdot y\sp{\a}{\grtrmu-\x} ]\\
	&=\mean[\indset{\intvopcl{-\infty}{1}}(y \sp{\a}{\grtrmu}) \cdot y\sp{\a}{\grtrmu-\sp{\x}{\grtr}\grtr-P_{\orthcompl{\grtr}}(\x)}]\\
	&=\sp{\grtrmu - \x}{\grtr} \cdot \mean[\indset{\intvopcl{-\infty}{1}}(\abs{\sp{\a}{\grtrmu}}) \cdot \abs{\sp{\a}{\grtr}}] \\
	&=\sp{\grtrmu - \x}{\grtr} \cdot \sqrt{\tfrac{2}{\pi}} \cdot \big(1 - e^{-\tfrac{1}{2\scalfac^2}}\big), \label{eq:proofs:generalset:excessloss:expmultipl} 
\end{align}
where we have again used the decomposition $\x = \sp{\x}{\grtr}\grtr + \proj_{\orthcompl{\grtr}}(\x)$.
Due to $\grtrmu - \x \in - \tilde{L} \subset t_0 \ball[2][n]$, it holds that $\sp{\x - \grtrmu}{\grtr} \leq t_0 \cdot \max\{0, \sign(\theta_{\x})\}$, so that we finally end up with 
\begin{align}
	\multiplterm{\x}{\grtrmu} &\geq - t_0 \cdot \Big( \frac{\sp{\x - \grtrmu}{\grtr}}{t_0} \cdot \underbrace{\sqrt{\tfrac{2}{\pi}} \cdot \big(1 - e^{-\tfrac{1}{2\scalfac^2}}\big)}_{\lesssim 1/\scalfac^2} + C_2 \cdot \underbrace{\normsubg{z}}_{\lesssim 1} \cdot \frac{\tfrac{1}{t_0}\meanwidth{\tilde{L}} + \sqrt{\log(\eta^{-1})}}{\sqrt{m}} \Big) \\
	&\geq - t_0 \cdot C_{\multipl} \cdot \Big( \frac{\max\{0,\sign(\theta_{\x})\}}{\scalfac^2} + \frac{\tfrac{1}{t_0}\meanwidth{\tilde{L}} + \sqrt{\log(\eta^{-1})}}{\sqrt{m}} \Big)\label{eq:proofs:generalset:excessloss:multiplterm}
\end{align}
for an appropriate numerical constant $C_{\multipl} > 0$.

Next, we derive a lower bound for the quadratic term in \eqref{eq:proofs:generalset:excesslossdecomp}.
For this purpose, we apply Mendelson's small ball method similarly to the proof of Proposition~\ref{prop:quadprocesslowerbound}: For $L = \tilde{L}$ and $u=\sqrt{2\log((\tfrac{\eta}{2})^{-1})}$, Theorem~\ref{thm:msbm} implies that 
\begin{align}
\inf_{\x \in \tilde{\sset}} \quadrterm{\x}{\grtrmu} \stackrel{\eqref{eq:quadempproc}}{\geq}
\inf_{\h \in \tilde{L}} \tfrac{1}{m} \sum_{i = 1}^m \empproc(\a_i, \h)
&\geq \xi \cdot \Big( Q_{2\xi}(\tilde{L}) - \frac{\tfrac{2}{\xi} \cdot \meanwidth{\tilde{L}}+ \sqrt{2\log((\tfrac{\eta}{2})^{-1})}}{\sqrt{m}} \Big) 
\end{align}
with probability at least $1-\tfrac{\eta}{2}$, where the constant $\xi > 0$ is specified later on.

Thus, it remains to find an appropriate lower bound for the small ball functional $Q_{2\xi}(\tilde{L})$.
Let $\h = \x - \grtrmu \in \tilde{L}$ and consider the orthogonal decomposition
\begin{equation}
	\h= \sp{\h}{\grtr}\grtr + \underbrace{\proj_{\orthcompl{\grtr}}(\h)}_{\eqqcolon e \x'}
\end{equation}
with $\x' \in \S^{n-1}$ and $e \in \R$.
First, assume that $\h \in \cyl_1(\grtr,\scalfac) - \grtrmu$ and $\sp{\h}{\grtr} \geq 0$ (see also Lemma~\ref{lem:proofs:generalset:cylinder}\ref{lem:proofs:generalset:cylinder:cylboundary}). Then, $e = 1$ and therefore
\begin{align}
\prob[\empproc_+(\a, \h) \geq 2\xi]&=\prob[y\sp{\a}{\h}\geq 4\xi, \abs{\sp{\a}{\grtrmu}}\in [1-2\xi, 1]]\\
&=\prob[\sp{\h}{\grtr} \cdot \abs{\sp{\a}{\grtr}}+ ey\sp{\a}{\x'}\geq 4\xi, \abs{\sp{\a}{\grtr}} \in [\tfrac{1}{\scalfac}-\tfrac{2\xi}{\scalfac}, \tfrac{1}{\scalfac}]]\\
&\geq \prob[y\sp{\a}{\x'}\geq 4\xi, \abs{\sp{\a}{\grtr}}\in [\tfrac{1}{\scalfac}-\tfrac{2\xi}{\scalfac}, \tfrac{1}{\scalfac}]].
\end{align}
We now set $\xi \coloneqq 1 / 16$ and use the independence of $\sp{\a}{\x'}$ and $\sp{\a}{\grtr}$ to obtain
\begin{align}
\prob[\empproc_+(\a, \h) \geq 2\xi]&\geq \prob[y\sp{\a}{\x'}\geq \tfrac{1}{4}, \abs{\sp{\a}{\grtr}}\in [\tfrac{7}{8\scalfac}, \tfrac{1}{\scalfac}]] \\
&= \prob[y\sp{\a}{\x'}\geq \tfrac{1}{4}] \cdot \prob[\abs{\sp{\a}{\grtr}}\in [\tfrac{7}{8\scalfac}, \tfrac{1}{\scalfac}]] \\
&= \prob[\gaussianuniv \geq \tfrac{1}{4}] \cdot \prob[\abs{\gaussianuniv} \in [\tfrac{7}{8\scalfac}, \tfrac{1}{\scalfac}]] \gtrsim \frac{1}{\scalfac} \ ,
\end{align}
where $\gaussianuniv \distributed \Normdistr{0}{1}$, and the last estimate is due to the assumption $\scalfac \gtrsim 1$. 
Similarly, if $\h \in \cyl_1(\grtr,\scalfac) - \grtrmu$ and $\sp{\h}{\grtr} \leq 0$, we have
\begin{align}
\prob[\empproc_-(\a, \h) \geq 2\xi]&=\prob[-y\sp{\a}{\h}\geq 4\xi, \abs{\sp{\a}{\grtrmu}}\in [1,1+ 2\xi]]\\
&=\prob[-\sp{\h}{\grtr} \cdot \abs{\sp{\a}{\grtr}} - ey\sp{\a}{\x'}\geq 4\xi, \abs{\sp{\a}{\grtr}}\in [\tfrac{1}{\scalfac},\tfrac{1}{\scalfac}+ \tfrac{2\xi}{\scalfac}]]\\
&\geq \prob[-y\sp{\a}{\x'}\geq \tfrac{1}{4}, \abs{\sp{\a}{\grtr}} \in [\tfrac{1}{\scalfac}, \tfrac{9}{8\scalfac}]]\\
&=\prob[\gaussianuniv \geq \tfrac{1}{4}] \cdot \prob[\abs{\gaussianuniv} \in [\tfrac{1}{\scalfac}, \tfrac{9}{8\scalfac}]] \gtrsim \frac{1}{\scalfac} \ , \quad \gaussianuniv \distributed \Normdistr{0}{1}.
\end{align}
Finally, we need to handle the base of $\cyl(\grtr, \scalfac)$: For every $\h \in \cyl_2(\grtr,\scalfac) - \grtrmu$, we have that $\sp{\h}{\grtr} = - \tfrac{\scalfac}{2}$ and $\abs{e} \leq 1$, which implies
\begin{align}
\prob[\empproc_-(\a, \h) \geq 2\xi]&=\prob[-y\sp{\a}{\h}\geq 4\xi, \abs{\sp{\a}{\grtrmu}}\in [1,1+ 2\xi]]\\
&=\prob[\tfrac{\scalfac}{2} \abs{\sp{\a}{\grtr}}- ey\sp{\a}{\x'}\geq 4\xi, \abs{\sp{\a}{\grtr}}\in [\tfrac{1}{\scalfac},\tfrac{1}{\scalfac}+ \tfrac{2\xi}{\scalfac}]]\\
&=\prob[\tfrac{\scalfac}{2} \abs{\sp{\a}{\grtr}}- ey\sp{\a}{\x'}\geq \tfrac{1}{4}, \abs{\sp{\a}{\grtr}}\in [\tfrac{1}{\scalfac}, \tfrac{9}{8\scalfac}]]\\
&\geq \prob[\abs{e y\sp{\a}{\x'}}\leq \tfrac{1}{8}, \tfrac{\scalfac}{2} \abs{\sp{\a}{\grtr}} \geq \tfrac{3}{8} , \abs{\sp{\a}{\grtr}} \in [\tfrac{1}{\scalfac},\tfrac{9}{8\scalfac}]]\\
&=\prob[\abs{ey\sp{\a}{\x'}}\leq \tfrac{1}{8}] \cdot \prob[\abs{\sp{\a}{\grtr}}\in [\tfrac{1}{\scalfac},\tfrac{9}{8\scalfac}]] \\
&\geq\prob[\abs{\gaussianuniv} \leq \tfrac{1}{8}] \cdot \prob[\abs{\gaussianuniv}\in [\tfrac{1}{\scalfac},\tfrac{9}{8\scalfac}]] \gtrsim \frac{1}{\scalfac} \ , \quad \gaussianuniv \distributed \Normdistr{0}{1}.
\end{align}
By a simple union bound argument (similarly to the proof of Proposition~\ref{prop:quadprocesslowerbound}), we conclude that
\begin{equation}
	Q_{1/8}(\tilde{L}) \gtrsim \frac{1}{\scalfac} \ .
\end{equation}
Hence, with probability at least $1 - \tfrac{\eta}{2}$, the quadratic term satisfies the lower bound
\begin{align}
	\quadrterm{\x}{\grtrmu} &\geq \frac{C}{\scalfac} - C_{\quadr} \cdot \frac{\meanwidth{\tilde{L}} + \sqrt{\log(\eta^{-1})}}{\sqrt{m}} \\
	&\geq t_0 \cdot \Big(\frac{C}{t_0 \scalfac} - C_{\quadr} \cdot \frac{\tfrac{1}{t_0}\meanwidth{\tilde{L}} + \sqrt{\log(\eta^{-1})}}{\sqrt{m}} \Big) \label{eq:proofs:generalset:excessloss:quadrterm}
\end{align}
for all $\x \in \tilde{\sset}$ and appropriate numerical constants $C, C_{\quadr} > 0$; note that we have also used that $\eta \leq \tfrac{1}{2}$ and $t_0 \geq 1$ here. 

Combining our lower bounds from \eqref{eq:proofs:generalset:excessloss:multiplterm} and \eqref{eq:proofs:generalset:excessloss:quadrterm}, the following holds true with probability at least $1 - \eta$ for all $\x \in \tilde{\sset}$:
\begin{align}
\excessloss(\x) &= \multiplterm{\x}{\grtrmu} + \quadrterm{\x}{\grtrmu} \\
&\geq t_0 \cdot \Big( \frac{C}{t_0 \scalfac} - C_{\multipl} \cdot \frac{ \max\{0,\sign(\theta_{\x})\}}{\scalfac^2} - (C_{\multipl} + C_{\quadr}) \cdot \frac{\tfrac{1}{t_0}\meanwidth{\tilde{L}} + \sqrt{\log(\eta^{-1})}}{\sqrt{m}} \Big), 
\end{align}
which is the claim of Lemma~\ref{lem:proofs:generalset:excessloss}.
\end{proof}

We are now ready to prove Theorem~\ref{thm:results:generalset:global}. For this purpose, we will apply Lemma~\ref{lem:proofs:generalset:excessloss} to different subsets of $\ssetmu$ and show that the excess risk is positive on their respective cylindrical intersections with $\boundary \cyl(\grtr,\scalfac)$.
Applying Fact~\ref{fact} separately to each of these subsets then yields the desired error bound.


\begin{proof}[Proof of Theorem~\ref{thm:results:generalset:global}]
	\textbf{Part 1:}
	We first apply Lemma~\ref{lem:proofs:generalset:excessloss} to
	\begin{equation}
		L^+ \coloneqq L \coloneqq \Big\{ \x = \grtrmu + \h \in \ssetmu \suchthat 0 \leq \sp{\grtr}{\tfrac{\h}{\lnorm{\h}}} \leq \sqrt{1 - \tfrac{1}{D^2 \scalfac^2}} \Big\},
	\end{equation}
	where $D > 0$ is a numerical constant that is specified later on. Note that this set is well-defined due to the assumption $\scalfac \gtrsim 1$, or more precisely, $\scalfac \geq D^{-1}$.
	Let us now estimate the radius $t_0 = \max\{1 , \rad((L^+ \intersec \cyl(\grtr, \scalfac)) - \grtrmu)\}$. To this end, let $\h \in(L^+ \intersec \cyl(\grtr, \scalfac)) - \grtrmu$. Since $\x = \grtrmu + \h \in \cyl(\grtr,\scalfac)$, we particularly have
	\begin{align}
		\lnorm{\h}^2 &= \lnorm{\grtrmu - \x}^2 = \lnorm{\underbrace{(\scalfac - \sp{\x}{\grtr})}_{= - \sp{\h}{\grtr}}\grtr - (\x - \sp{\x}{\grtr}\grtr)}^2 \\
		&= \abs{\sp{\h}{\grtr}}^2 + \lnorm{\x - \sp{\x}{\grtr}\grtr}^2 \leq \abs{\sp{\h}{\grtr}}^2 + 1. \label{eq:proofs:generalset:global:cylboundh}
	\end{align}
	And by the definition of $L^+$, it holds that
	\begin{align}
		&\lnorm{\h}^2 \leq \abs{\sp{\h}{\grtr}}^2 + 1 \leq \lnorm{\h}^2 \cdot \Big(1 - \tfrac{1}{D^2 \scalfac^2}\Big) + 1 \\
		&\implies \quad \lnorm{\h}^2 \leq D^2 \scalfac^2 \quad \implies \quad \lnorm{\h} \leq D \scalfac,
	\end{align}
	which implies $t_0 \leq D \scalfac$.
	From \eqref{eq:proofs:generalset:excesslossbound}, we can conclude that, with probability at least $1 - \tfrac{\eta}{2}$, the excess risk satisfies
	\begin{align}
		\excessloss(\x) &\gtrsim t_0 \cdot \Big( \frac{1}{t_0 \scalfac} - \frac{C_1}{\scalfac^2} - C_2 \cdot \frac{\tfrac{1}{t_0}\meanwidth{\tilde{L}} + \sqrt{\log(\eta^{-1})}}{\sqrt{m}} \Big) \\
		&\geq t_0 \cdot \Big( \frac{1}{D \scalfac^2} - \frac{C_1}{\scalfac^2} - C_2 \cdot \frac{\tfrac{1}{t_0}\meanwidth{\tilde{L}} + \sqrt{\log(\eta^{-1})}}{\sqrt{m}} \Big) \label{eq:proofs:generalset:global:excesslosslower}
	\end{align}
	for all $\x \in \tilde{\sset} = L^+ \intersec \boundary \cyl(\grtr, \scalfac)$. Adjusting $D$ (depending on $C_1$) and observing that
	\begin{align}
		\tfrac{1}{t_0}\meanwidth{\tilde{L}} &\leq \tfrac{1}{t_0}\meanwidth{(L^+ - \grtrmu) \intersec t_0\ball[2][n]} \\
		&\leq \meanwidth{\tfrac{1}{t_0}(\ssetmu - \grtrmu) \intersec \ball[2][n]} \stackrel{\eqref{eq:efflessconic}}{\leq} \sqrt{\effdim[\conic]{\sset - \grtr}}, \label{eq:proofs:generalset:global:coniceffdim}
	\end{align}
	the assumption \eqref{eq:results:generalset:global:meas} implies that $\excessloss(\x) > 0$ for all $\x \in \tilde{\sset}$.
	In conclusion, Fact~\ref{fact} states that every minimizer $\solu$ of \eqref{eq:estimatortuned} that belongs to $L^+$ must be also contained in $\cyl(\grtr,\scalfac)$.
	Lemma~\ref{lem:proofs:generalset:cylinder}\ref{lem:proofs:generalset:cylinder:errorbound} finally yields
	\begin{equation}
		\lnorm[auto]{\grtr - \frac{\solu}{\lnorm{\solu}}} \lesssim \frac{1}{\scalfac} \ .
	\end{equation}
	
	\textbf{Part 2:}
	Analogously to Part~1, we now invoke Lemma~\ref{lem:proofs:generalset:excessloss} with
	\begin{equation}
		L^- \coloneqq L \coloneqq \Big\{ \x = \grtrmu + \h \in \ssetmu \suchthat \sp{\grtr}{\tfrac{\h}{\lnorm{\h}}} \leq 0 \Big\}.
	\end{equation}
	For $\h \in (L^- \intersec \cyl(\grtr, \scalfac)) - \grtrmu$, the definition of $\cyl(\grtr,\scalfac)$ implies that $-\tfrac{\scalfac}{2} \leq \sp{\h}{\grtr}$.
	Hence, 
	\begin{equation}
		\lnorm{\h}^2 \stackrel{\eqref{eq:proofs:generalset:global:cylboundh}}{\leq} \abs{\sp{\h}{\grtr}}^2 + 1 \leq \tfrac{\scalfac^2}{4}+1 \lesssim \scalfac^2,
	\end{equation}
	and therefore $t_0 = \max\{ 1 , \rad((L^- \intersec \cyl(\grtr, \scalfac)) - \grtrmu)\} \lesssim \scalfac$.
	By \eqref{eq:proofs:generalset:excesslossbound} and \eqref{eq:proofs:generalset:global:coniceffdim} again, we obtain
	\begin{align}
		\excessloss(\x) &\gtrsim t_0 \cdot \Big( \frac{1}{t_0 \scalfac} - C_2 \cdot \frac{\tfrac{1}{t_0}\meanwidth{\tilde{L}} + \sqrt{\log(\eta^{-1})}}{\sqrt{m}} \Big) \\
		&\gtrsim t_0 \cdot \Big( \frac{1}{\scalfac^2} - C_2 \cdot \frac{\sqrt{\effdim[\conic]{\sset - \grtr}} + \sqrt{\log(\eta^{-1})}}{\sqrt{m}} \Big) \stackrel{\eqref{eq:results:generalset:global:meas}}{>} 0 \label{eq:proofs:generalset:global:excesslosslowerneg}
	\end{align}
	for all $\x \in \tilde{\sset} = L^- \intersec \boundary \cyl(\grtr, \scalfac)$ with probability at least $1 - \tfrac{\eta}{2}$.
	As in the first part, we can conclude that every minimizer $\solu$ of \eqref{eq:estimatortuned} that belongs to $L^-$ satisfies
	\begin{equation}
		\lnorm[auto]{\grtr - \frac{\solu}{\lnorm{\solu}}} \lesssim \frac{1}{\scalfac} \ .
	\end{equation}
	
	\textbf{Part 3:} It remains to deal with those vectors $\x = \grtrmu + \h \in \ssetmu$ with $\sp{\grtr}{\tfrac{\h}{\lnorm{\h}}} > \sqrt{1 - \tfrac{1}{D^2 \scalfac^2}}$.
	In fact, such points satisfy the desired error bound \eqref{eq:results:generalset:global:bound}, no matter if they solve \eqref{eq:estimatortuned} or not.
	To verify this claim, observe that
	\begin{equation}
		\lnorm[auto]{\grtr - \frac{\x}{\lnorm{\x}}} = \left[2 \bigg(1 - \frac{\scalfac + \lnorm{\h} \sp{\grtr}{\vec{v}}}{\sqrt{\scalfac^2 + 2\scalfac\lnorm{\h}\sp{\grtr}{\vec{v}} + \lnorm{\h}^2}} \bigg)\right]^{1/2},
	\end{equation}
	where $\vec{v} \coloneqq \h / \lnorm{\h} \in \S^{n-1}$.
	It is not hard to see that the expression on the right-hand side is monotonically increasing in $\lnorm{\h}$.
	Thus, by taking the limit $\lnorm{\h} \to \infty$, we obtain
	\begin{align}
		\lnorm[auto]{\grtr - \frac{\x}{\lnorm{\x}}} &\leq \lim_{\lnorm{\h} \to \infty} \left[2 \bigg(1 - \frac{\scalfac + \lnorm{\h} \sp{\grtr}{\vec{v}}}{\sqrt{\scalfac^2 + 2\scalfac\lnorm{\h}\sp{\grtr}{\vec{v}} + \lnorm{\h}^2}} \bigg)\right]^{1/2} \\
		&= \sqrt{2 (1 - \sp{\grtr}{\vec{v}})} \leq \sqrt{2 (1 - \sp{\grtr}{\vec{v}}^2)} < \frac{\sqrt{2}}{D \scalfac} \lesssim \frac{1}{\scalfac} \ .
	\end{align}
	This proves the claim under the hypothesis of \eqref{eq:results:generalset:global:meas}.
	
	The same argument works out if \eqref{eq:results:generalset:global:meas:global} is satisfied instead.
	Indeed, the only estimate that needs to be changed in Part~1 is \eqref{eq:proofs:generalset:global:excesslosslower}:
	\begin{align}
		\excessloss(\x) &\gtrsim t_0 \cdot \Big( \frac{1}{t_0 \scalfac} - \frac{C_1}{\scalfac^2} - C_2 \cdot \frac{\tfrac{1}{t_0}\meanwidth{\tilde{L}} + \sqrt{\log(\eta^{-1})}}{\sqrt{m}} \Big) \\ 
		&= \frac{1}{\scalfac} - \frac{C_1 t_0}{\scalfac^2} - C_2 \cdot \frac{\meanwidth{\tilde{L}} + t_0 \sqrt{\log(\eta^{-1})}}{\sqrt{m}} \\
		&\geq \frac{1}{\scalfac} - \frac{C_1 D}{\scalfac} - C_2 \cdot \frac{\meanwidth{\scalfac(\sset - \grtr)} + D \scalfac \sqrt{\log(\eta^{-1})}}{\sqrt{m}} \\
		&\gtrsim \frac{1}{\scalfac} - C_2 \cdot \scalfac \cdot \frac{\meanwidth{\sset} + \sqrt{\log(\eta^{-1})}}{\sqrt{m}} \ .
	\end{align}
	Since the same lower bound can be achieved in Part~2, the positivity of $\excessloss(\cdot)$ is a consequence of \eqref{eq:results:generalset:global:meas:global} in both cases.
\end{proof}
\begin{remark}
	A key step of the above proof is to control the radii $t_0$ of the respective cylindrical intersections.
	While this strategy works out nicely in Part~1 and Part~2 with $t_0 \lesssim \scalfac$, the remaining vectors of Part~3 need to be treated very differently.
	In fact, the statistical argument of Lemma~\ref{lem:proofs:generalset:excessloss} would completely fail in this case.
	But fortunately, every vector $\x = \grtrmu + \h \in \ssetmu$ that satisfies $\sp{\grtr}{\tfrac{\h}{\lnorm{\h}}} > \sqrt{1 - \tfrac{1}{D^2 \scalfac^2}}$ already lies in a ``narrow'' conic segment around $\spann\{\grtr\}$, whose projection onto the sphere is sufficiently close to $\grtr$.
	But one should be aware of the fact that these $\x$ do not necessarily belong to $\cyl(\grtr,\scalfac)$.
\end{remark}

The proof of Theorem~\ref{thm:results:generalset:local} is slightly more involved.
A careful study of the proof of Theorem~\ref{thm:results:generalset:global} reveals that the set $(L^- \intersec \cyl(\grtr, \scalfac)) - \grtrmu$ considered in Part~2 is simply too large and its radius $t_0$ may scale in the order of $\scalfac$. In order to ensure positivity of the excess risk, we therefore have to accept a suboptimal factor of $\scalfac^4$ in \eqref{eq:results:generalset:global:meas} .
However, this drawback can be avoided by showing that any minimizer of \eqref{eq:estimatortuned} actually lies on the boundary of $\ssetmu$ with high probability.
For this purpose, we first verify that the empirical risk $\x \mapsto \lossemp(\x)$ does not vanish in a neighborhood of $\grtrmu$, i.e., $\grtrmu$ is not too far behind the classification margin (cf. Subsection~\ref{subsec:results:generalset:geometry}).
This is precisely what is claimed by the following lemma, which is based on a standard concentration argument:
\begin{lemma}\label{lem:proofs:generalset:lossempconcentration}
	Let $L \subset t \ball[2][n]$ be a subset and assume that $\scalfac \gtrsim \max\{1,t\}$. 
	For every $\eta \in \intvop{0}{\tfrac{1}{2}}$, the following holds true with probability at least $1 - \eta$:
	Supposed that 
	\begin{equation}\label{eq:proofs:generalset:lossempconcentration:meas}
		m \gtrsim \scalfac^{2} \cdot \max\{\meanwidth{L}^2, \log(\eta^{-1})\},
	\end{equation}
	the empirical risk is positive on $L + \grtrmu$, i.e.,
	\begin{equation}
		\lossemp(\grtrmu + \h) > 0 \quad \text{for all $\h \in L$.}
	\end{equation}
\end{lemma}
\begin{proof}
	The monotonicity of the hinge loss yields the following lower bound for the empirical risk function:
	\begin{align}
		\lossemp(\grtrmu + \h) &= \tfrac{1}{m} \sum_{i = 1}^m \pospart{1 - y_i\sp{\a_i}{\grtrmu + \h}} \geq \tfrac{1}{m} \sum_{i = 1}^m \underbrace{\pospart{1 - \abs{\sp{\a_i}{\grtrmu + \h}}}}_{\eqqcolon Z_i(\h)}
	\end{align}
	for all $\h \in \R^n$. Let us establish a concentration inequality for the empirical process on the right-hand side of this estimate.
	Since $Z_i(\cdot) \in \intvcl{0}{1}$, the bounded difference inequality \cite[Thm.~6.2]{boucheron2013concentration} implies that
	\begin{equation}
		\sup_{\h \in L} \tfrac{1}{m}\sum_{i = 1}^m \Big(\mean[Z_i(\h)] - Z_i(\h)\Big) \leq \underbrace{\mean{} \Big[ \sup_{\h \in L} \tfrac{1}{m}\sum_{i = 1}^m \Big(\mean[Z_i(\h)] - Z_i(\h)\Big) \Big]}_{\eqqcolon E} + \frac{\sqrt{2\log(\eta^{-1})}}{\sqrt{m}} \label{eq:proofs:generalset:lossempconcentration:bnddiff}
	\end{equation}
	with probability at least $1 - \eta$.
	
	In order to bound $E$ from above, we first apply a standard symmetrization argument (cf. \cite[Pf.~of~Lem.~6.3]{ledoux1991pbs}):
	\begin{align}
		E &= \mean{} \Big[ \sup_{\h \in L} \tfrac{1}{m}\sum_{i = 1}^m \Big((1 - Z_i(\h)) - \mean[1 - Z_i(\h)]\Big) \Big] \leq 2 \mean{} \Big[ \sup_{\h \in L} \tfrac{1}{m}\sum_{i = 1}^m \epsilon_i (1 - Z_i(\h)) \Big],
	\end{align}
	where $\epsilon_i$ are independent Rademacher variables. 
	Next, we define $\psi(v) \coloneqq 1 - \pospart{1 - \abs{v}}$ for $v \in \R$ and observe that $\psi(\sp{\a_i}{\grtrmu + \h}) = 1 - \pospart{1 - \abs{\sp{\a_i}{\grtrmu + \h}}} = 1 - Z_i(\h)$. Since $\psi$ is a contraction that fixes the origin, the contraction principle \cite[Eq.~(4.20)]{ledoux1991pbs} finally leads to
	\begin{align}
		E &\leq 2 \mean{} \Big[ \sup_{\h \in L} \tfrac{1}{m}\sum_{i = 1}^m \epsilon_i \sp{\a_i}{\grtrmu + \h} \Big] = 2 \mean{} [ \sup_{\h \in L} \tfrac{1}{\sqrt{m}}\sp{\a}{\grtrmu + \h} ] = \tfrac{2}{\sqrt{m}} \mean{} [ \sup_{\h \in L} \sp{\gaussian}{\h} ] = \tfrac{2\meanwidth{L}}{\sqrt{m}},
	\end{align}
	where $\gaussian \distributed \Normdistr{\vnull}{\I{n}}$.
	Therefore, under the hypothesis of \eqref{eq:proofs:generalset:lossempconcentration:bnddiff}, we have
	\begin{align}
		\lossemp(\grtrmu + \h) \geq \tfrac{1}{m}\sum_{i = 1}^m Z_i(\h) \gtrsim \mean[Z_1(\h)] - \frac{\meanwidth{L} + \sqrt{\log(\eta^{-1})}}{\sqrt{m}}
		\label{eq:proofs:generalset:lossempconcentration:lowerbnd}
	\end{align}
	for all $\h \in L$. To conclude the proof, it remains to estimate $\mean[Z_1(\h)]$ from below: 
	\begin{align}
		\mean[Z_1(\h)] = \mean{} [\pospart[]{ 1 - \abs{\underbrace{\sp{\a}{\grtrmu + \h}}_{\distributed \Normdistr{0}{\lnorm{\grtrmu + \h}^2}}} }] \gtrsim \frac{1}{\lnorm{\grtrmu + \h}} \geq \frac{1}{\scalfac + t} \gtrsim \frac{1}{\scalfac} \ ,
	\end{align}
	where we have also used that $\scalfac \gtrsim \max\{1, t\}$.
	The claim now follows from \eqref{eq:proofs:generalset:lossempconcentration:lowerbnd} and the assumption of \eqref{eq:proofs:generalset:lossempconcentration:meas}.
\end{proof}

We are now ready to prove Theorem~\ref{thm:results:generalset:local}.
\begin{proof}[Proof of Theorem~\ref{thm:results:generalset:local}]
	\textbf{Part 1:} For $\x \in \R^n$, we denote by $\x^{\uparrow}$ the ray that starts at $\x$ and proceeds in the direction of $\grtr$ (parallel to $\spann\{\grtr\}$), i.e.,
	\begin{equation}
		\x^{\uparrow} \coloneqq \{ \x + \tau\grtr \suchthat \tau \geq 0 \} \subset \R^n.
	\end{equation}
	If $\x \in \ssetmu$, then $\x^{\uparrow}$ intersects the boundary of $\ssetmu$ in a point $\boundary_0 \x \coloneqq \x + \tau_0 \grtr \in \x^\uparrow \intersec \boundary(\ssetmu)$ with\footnote{While the intersection of $\x^\uparrow$ and $\boundary(\ssetmu)$ is not necessarily a single point, the boundary point $\boundary_0 \x$ is uniquely defined due to the definition of $\tau_0$.}
	\begin{equation}
		\tau_0 = \sup \{ \tau \geq 0 \suchthat \x + \tau \grtr \in \ssetmu \}.
	\end{equation}
	
	Since we consider perfect $1$-bit observations, i.e., $y_i = \sign(\sp{\a_i}{\grtr})$, it is not hard to see that for every $\x \in \R^n$, there exists $\x^\natural \in \x^\uparrow \intersec \setcompl{(\ssetmu)}$ such that $\lossemp(\x^\natural) = 0$; note that the choice of $\x^\natural$ is highly non-unique and may strongly depend on the measurement vectors $\a_1, \dots, \a_m$.
	Hence, similarly to the assertion of Fact~\ref{fact}, the convexity of $\lossemp(\cdot)$ and $\ssetmu$ implies that every minimizer $\solu$ of \eqref{eq:estimatortuned} satisfies the following property:
	\begin{fact}\label{fact:proofs:generalset:local}
		If $\solu \neq \boundary_0\solu$, then we have $\lossemp(\solu^\uparrow) = \{0\}$. 
		Or equivalently, $\lossemp(\solu) > 0$ implies that $\solu = \boundary_0\solu$.
	\end{fact}
	A particular consequence of Fact~\ref{fact:proofs:generalset:local} is that every point in $\solu^\uparrow \intersec \ssetmu$ is also a minimizer of \eqref{eq:estimatortuned}.
	The proof strategy of the following two parts is visualized in Figure~\ref{fig:proofs:generalset:proof-generalset-local}.
	\begin{figure}
	\centering
	\begin{subfigure}[t]{0.4\textwidth}
		\centering
		\tikzstyle{blackdot}=[shape=circle,fill=black,minimum size=1mm,inner sep=0pt,outer sep=0pt]
		\begin{tikzpicture}[scale=2,extended line/.style={shorten >=-#1,shorten <=-#1},extended line/.default=1cm]]
			\coordinate (K1) at (0,0);
			\coordinate (K2) at (.85,-1.13333);
			\coordinate (K3) at (.5,-1.75);
			\coordinate (K4) at (-.5,-1.5);
			\coordinate (K5) at (-.5,-.75);
			\coordinate (P0) at (0,-1);
			\path[name path=K1--K2] (K1) -- (K2);
			\path[name path=P0--X0] (P0) -- ++(65:1);
			\path[name intersections={of=P0--X0 and K1--K2,by=X0}];
			\coordinate (X0muhalf) at ($(X0)!.5!(P0)$);
			\coordinate (SpanEnd) at ($(X0)!-1.5!(P0)$);
			
			\path (X0muhalf) -- ++(-25:.2) coordinate (CylR);
			\coordinate (CylL) at ($(X0muhalf)!-1!(CylR)$);
			\coordinate (CylLtop) at ($(CylL)!(SpanEnd)!90:(CylR)$);
			\coordinate (CylRtop) at ($(CylR)!(SpanEnd)!90:(CylL)$);
			
			\path[name path=SpanEndCompl] (CylLtop) -- ++(-25:1);

			\coordinate (XhatP2) at ($(P0) + (0,.2)$);
			\path[name path=XhatP2ray] (XhatP2) -- ++(65:2);
						
			\draw[fill=gray!20!white,thick] (K1) -- (K2) -- (K3) -- (K4) -- (K5) -- cycle;
			\draw[fill=gray!20!white] (CylLtop) -- (CylL) -- (CylR) -- (CylRtop);
			\begin{scope}
				\clip (K1) -- (K2) -- (K3) -- (K4) -- (K5) -- cycle;
				\draw[fill=gray!60!white] (CylLtop) -- (CylL) -- (CylR) -- (CylRtop);
			\end{scope}
			\draw[thick] (K1) -- (K2) -- (K3) -- (K4) -- (K5) -- cycle;

			\draw[thick,name intersections={of=SpanEndCompl and XhatP2ray}] (XhatP2) -- (intersection-1) coordinate [pos=.8] (XhatP2rayanchor);
			\path[name intersections={of=XhatP2ray and K1--K2,by=XhatP2b}];
			
			\node[left=.7 of XhatP2rayanchor] (XhatP2braylabel) {$\solu^\uparrow$};
			\path[<-,shorten <=3pt,>=stealth] (XhatP2rayanchor) edge (XhatP2braylabel) ;

			\node[blackdot,fill=red] at (XhatP2b) {};
			\node[left=1.2 of XhatP2b] (XhatP2blabel) {$\boundary_0 \solu$};
			\path[<-,shorten <=3pt,>=stealth] (XhatP2b) edge (XhatP2blabel) ;

			\draw[dashed] (SpanEnd) -- ($(X0)!2.75!(P0)$);
			\node at (barycentric cs:K1=1,K2=1,K3=8,K4=1,K5=1) {$\scalfac\sset$};
			\node[blackdot, label=0:$\vnull$] at (P0) {};
			\node[blackdot, label=right:$\grtrmu$] at (X0) {};
			\node[label={[label distance=-3pt]left:$\spann\{\grtr\}$}] at ($(X0)!2.35!(P0)$) {};
			\node[below right=.2 and .0 of CylRtop] {$\cyl(\grtr,\scalfac)$};
			\node[blackdot, label=left:$\solu$] at (XhatP2) {};
		\end{tikzpicture}
		\caption{Situation of Part~2}
		\label{fig:proofs:generalset:proof-generalset-local:part2}
	\end{subfigure}%
	\qquad\qquad
	\begin{subfigure}[t]{0.4\textwidth}
		\centering
		\tikzstyle{blackdot}=[shape=circle,fill=black,minimum size=1mm,inner sep=0pt,outer sep=0pt]
		\begin{tikzpicture}[scale=2,extended line/.style={shorten >=-#1,shorten <=-#1},extended line/.default=1cm]]
			\coordinate (K1) at (0,0);
			\coordinate (K2) at (.85,-1.13333);
			\coordinate (K3) at (.5,-1.75);
			\coordinate (K4) at (-.5,-1.5);
			\coordinate (K5) at (-.5,-.75);
			\coordinate (P0) at (0,-1);
			\path[name path=K1--K2] (K1) -- (K2);
			\path[name path=P0--X0] (P0) -- ++(65:1);
			\path[name intersections={of=P0--X0 and K1--K2,by=X0}];
			\coordinate (X0muhalf) at ($(X0)!.5!(P0)$);
			\coordinate (SpanEnd) at ($(X0)!-1.5!(P0)$);
			
			\path (X0muhalf) -- ++(-25:.2) coordinate (CylR);
			\coordinate (CylL) at ($(X0muhalf)!-1!(CylR)$);
			\coordinate (CylLtop) at ($(CylL)!(SpanEnd)!90:(CylR)$);
			\coordinate (CylRtop) at ($(CylR)!(SpanEnd)!90:(CylL)$);
			
			\path[name path=SpanEndCompl] (CylLtop) -- ++(-25:1);

			\coordinate (XhatP3) at ($(P0) + (0.5,-.15)$);
			\path[name path=XhatP3ray] (XhatP3) -- ++(65:2);
						
			\draw[fill=gray!20!white,thick] (K1) -- (K2) -- (K3) -- (K4) -- (K5) -- cycle;
			\draw[fill=gray!20!white] (CylLtop) -- (CylL) -- (CylR) -- (CylRtop);
			\begin{scope}
				\clip (K1) -- (K2) -- (K3) -- (K4) -- (K5) -- cycle;
				\draw[fill=gray!60!white] (CylLtop) -- (CylL) -- (CylR) -- (CylRtop);
			\end{scope}
			\draw[thick] (K1) -- (K2) -- (K3) -- (K4) -- (K5) -- cycle;

			\draw[thick, name intersections={of=SpanEndCompl and XhatP3ray}] (XhatP3) -- (intersection-1) node [pos=.7,right] {$\solu^\uparrow$};
			\path[name intersections={of=XhatP3ray and K1--K2,by=XhatP3b}];
			
			\node[blackdot,fill=red] at (XhatP3b) {};
			\node[right=.7 of XhatP3b] (XhatP3blabel) {$\boundary_0 \solu$};
			\path[<-,shorten <=3pt,>=stealth] (XhatP3b) edge (XhatP3blabel) ;

			\draw[dashed] (SpanEnd) -- ($(X0)!2.75!(P0)$);
			\node at (barycentric cs:K1=1,K2=1,K3=8,K4=1,K5=1) {$\scalfac\sset$};
			\node[blackdot, label=0:$\vnull$] at (P0) {};
			\node[blackdot, label=right:$\grtrmu$] at (X0) {};
			\node[label={[label distance=-3pt]left:$\spann\{\grtr\}$}] at ($(X0)!2.35!(P0)$) {};
			\node[below left=0 and .2 of CylLtop] {$\cyl(\grtr,\scalfac)$};
			\node[blackdot, label=right:$\solu$] at (XhatP3) {};
		\end{tikzpicture}
		\caption{Situation of Part~3}
		\label{fig:proofs:generalset:proof-generalset-local:part3}
	\end{subfigure}%
	\caption{The geometric situations of Part~2 and Part~3 in the proof of Theorem~\protect{\ref{thm:results:generalset:local}}. In Part~2, we show that $\solu \neq \boundary_0\solu$ leads to a contradiction (see \subref{fig:proofs:generalset:proof-generalset-local:part2}), whereas the event of Part~3 cannot occur because $\solu^\uparrow$ always intersects a region on which the excess risk is positive.}
	\label{fig:proofs:generalset:proof-generalset-local}
	\end{figure}
	
	\textbf{Part 2:} Let us assume that $\solu$ is a minimizer of \eqref{eq:estimatortuned} such that $\boundary_0 \solu \in \cyl(\grtr,\scalfac)$.
	If we could show that $\solu = \boundary_0\solu$, the desired error bound would follow again from Lemma~\ref{lem:proofs:generalset:cylinder}\ref{lem:proofs:generalset:cylinder:errorbound}.
	Towards a contraction, let us therefore assume that $\solu \neq \boundary_0\solu$. Then Fact~\ref{fact:proofs:generalset:local} yields $\lossemp(\boundary_0\solu) = 0$.
	In order to show that this event does not occur with high probability, we apply Lemma~\ref{lem:proofs:generalset:lossempconcentration} with
	\begin{equation}
		L = (\boundary(\ssetmu) \intersec \cyl(\grtr,\scalfac)) - \grtrmu \subset t \ball[2][n],
	\end{equation}
	where $t = t_0 = \max\{1 , \rad(L)\}$; see also \eqref{eq:results:generalset:t0}.
	Note that the condition \eqref{eq:proofs:generalset:lossempconcentration:meas} is fulfilled by \eqref{eq:results:generalset:local:meas} and
	\begin{align}
		\max\{\meanwidth{L}^2, \log(\eta^{-1})\} &\stackrel{t_0 \gtrsim 1}{\lesssim} t_0^2 \cdot \max\{\tfrac{1}{t_0^2}\meanwidth{L}^2, \log(\eta^{-1})\} \\
		&\stackrel{\eqref{eq:efflessconic}}{\leq} t_0^2 \cdot \max\{\effdim[\conic]{\sset - \grtr}, \log(\eta^{-1})\}.
	\end{align}
	Lemma~\ref{lem:proofs:generalset:lossempconcentration} now states that, with probability at least $1- \tfrac{\eta}{2}$, we have $\lossemp(\grtrmu + \h) > 0$ for all $\h \in L$, or equivalently, $\lossemp(\x) > 0$ for all $\x \in \boundary(\ssetmu) \intersec \cyl(\grtr,\scalfac)$.
	In particular, it holds that $\lossemp(\boundary_0\solu) > 0$, which is a contradiction.
		
	\textbf{Part 3:} It remains to verify that the event $\boundary_0 \solu \not\in \cyl(\grtr,\scalfac)$ cannot occur with high probability.
	Towards a contradiction, let us assume that there exists a minimizer $\solu$ of \eqref{eq:estimatortuned} such that $\boundary_0 \solu \not\in \cyl(\grtr,\scalfac)$.
	
	Next, we apply Lemma~\ref{lem:proofs:generalset:excessloss} with $L \coloneqq \boundary(\ssetmu)$ and proceed analogously to the proof of Theorem~\ref{thm:results:generalset:global}:
	Using \eqref{eq:results:generalset:local:meas} and the assumption $\scalfac \gtrsim t_0$, we can conclude from \eqref{eq:proofs:generalset:excesslossbound} that with probability at least $1- \tfrac{\eta}{2}$, it holds that $\excessloss(\x) > 0$ for all $\x \in \boundary(\ssetmu) \intersec \boundary \cyl(\grtr, \scalfac)$. 
	For the remainder of the proof, we assume that this event has indeed occurred.
	
	Our hypothesis $\boundary_0 \solu \not\in \cyl(\grtr,\scalfac)$ implies that there exists a directional vector $\h' \in ( \boundary(\ssetmu) \intersec \boundary \cyl(\grtr, \scalfac)) - \grtrmu$ such that the ray $\{ \grtrmu + \tau \h' \suchthat \tau \geq 1 \}$ intersects $\solu^\uparrow$, say in $\x' = \grtrmu + \tau' \h'$. Note that $\h'$ lies in the plane spanned by $\grtr$ and $\solu$; see also Figure~\ref{fig:proofs:generalset:intersect-part3} for an illustration of this planar-geometric argument.
	According to the above event, we know that $\excessloss(\grtrmu + \h') > 0$, and by the convexity of $\excessloss(\cdot)$, also that $\excessloss(\x') = \excessloss(\grtrmu + \tau'\h') > 0$. 
	Moreover, there exists a point $\solu^\natural \in \solu^\uparrow$ such that $\lossemp(\solu^\natural) = 0$ and $\x' \in \convhull\{\solu, \solu^\natural\}$.
	This eventually contradicts the convexity of the excess risk, since $\excessloss(\solu) \leq 0$ and $\excessloss(\solu^\natural) = \lossemp(\solu^\natural) - \lossemp(\grtrmu) = 0 - \lossemp(\grtrmu) \leq 0$.
	\begin{figure}
	\centering
	\tikzstyle{blackdot}=[shape=circle,fill=black,minimum size=1mm,inner sep=0pt,outer sep=0pt]
		\begin{tikzpicture}[scale=2,extended line/.style={shorten >=-#1,shorten <=-#1},extended line/.default=1cm]]
			\coordinate (K1) at (0,.1);
			\coordinate (K2) at (.6,-.2);
			\coordinate (K3) at (.85,-1.13333);
			\coordinate (K4) at (.5,-1.75);
			\coordinate (K5) at (-.5,-1.5);
			\coordinate (K6) at (-.5,-.75);
			\coordinate (P0) at (0,-1);
			\path[name path=K1--K2] (K1) -- (K2);
			\path[name path=K2--K3] (K2) -- (K3);
			\path[name path=P0--X0] (P0) -- ++(70:1);
			\path[name intersections={of=P0--X0 and K1--K2,by=X0}];
			\coordinate (X0muhalf) at ($(X0)!.5!(P0)$);
			\coordinate (SpanEnd) at ($(X0)!-.9!(P0)$);
			
			\path (X0muhalf) -- ++(-20:.22) coordinate (CylR);
			\coordinate (CylL) at ($(X0muhalf)!-1!(CylR)$);
			\coordinate (CylLtop) at ($(CylL)!(SpanEnd)!90:(CylR)$);
			\coordinate (CylRtop) at ($(CylR)!(SpanEnd)!90:(CylL)$);
			
			\path[name path=CylRbarrel] (CylR) -- (CylRtop);
			\path[name path=SpanEndCompl] (CylLtop) -- ++(-20:1);

			\coordinate (XhatP3) at ($(P0) + (0.55,-.15)$);
			\path[name path=XhatP3ray] (XhatP3) -- ++(70:2.5);
						
			\draw[fill=gray!20!white,thick] (K1) -- (K2) -- (K3) -- (K4) -- (K5) -- (K6) -- cycle;
			\draw[fill=gray!20!white] (CylLtop) -- (CylL) -- (CylR) -- (CylRtop);
			\begin{scope}
				\clip (K1) -- (K2) -- (K3) -- (K4) -- (K5) -- (K6) -- cycle;
				\draw[fill=gray!60!white] (CylLtop) -- (CylL) -- (CylR) -- (CylRtop);
			\end{scope}
			\draw[thick] (K1) -- (K2) -- (K3) -- (K4) -- (K5) -- (K6) -- cycle;

			\draw[thick, name intersections={of=SpanEndCompl and XhatP3ray}] (XhatP3) -- (intersection-1) node [pos=.9,right] {$\solu^\uparrow$};
			\path[name intersections={of=XhatP3ray and K2--K3,by=XhatP3b}];
			
			\node[blackdot] at (XhatP3b) {};
			\node[right=.7 of XhatP3b,yshift=-10pt] (XhatP3blabel) {$\boundary_0 \solu$};
			\path[<-,shorten <=3pt,>=stealth] (XhatP3b) edge (XhatP3blabel) ;
			
			\draw[dashed,name path=Hray] (X0) -- ($(X0)!4!(K2)$);
			
			\draw[red,thick,->,>=stealth,name intersections={of=Hray and CylRbarrel}] (X0) -- (intersection-1) node [pos=.7,above] {{\smaller$\h'$}};
			
			\path[blackdot,name intersections={of=Hray and XhatP3ray, by=I0}];

			\draw[dashed] (SpanEnd) -- ($(X0)!2.1!(P0)$);
			\node at (barycentric cs:K1=1,K2=1,K3=1,K4=8,K5=3,K6=1) {$\scalfac\sset$};
			\node[blackdot, label=0:$\vnull$] at (P0) {};
			\node[blackdot, label={[yshift=-2pt]left:$\grtrmu$}] at (X0) {};
			\node[label={[label distance=-3pt]left:$\spann\{\grtr\}$}] at ($(X0)!1.85!(P0)$) {};
			\node[below left=0 and .2 of CylLtop] {$\cyl(\grtr,\scalfac)$};
			\node[blackdot, label=right:$\solu$] at (XhatP3) {};
			\node[blackdot,fill=red] at (I0) {};
			
			\node[above right=.1 and 1 of I0] (I0label) {$\excessloss(\grtrmu + \tau'\h') > 0$};
			\path[<-,shorten <=3pt,>=stealth] (I0) edge (I0label) ;
		\end{tikzpicture}
	\caption{Illustration of the argument in Part~3 of the proof of Theorem~\protect{\ref{thm:results:generalset:local}}. The ray $\solu^\uparrow$ always intersects with another ray that starts at $\grtrmu$ and passes through a point $\x = \grtrmu + \h' \in \boundary(\ssetmu) \intersec \boundary \cyl(\grtr, \scalfac)$. The excess risk is positive at the intersection point $\x' = \grtrmu + \tau'\h'$, which leads to a contradiction. The figure shows the projection of all objects onto the plane spanned by $\grtr$ and $\solu$.}
	\label{fig:proofs:generalset:intersect-part3}
	\end{figure}
\end{proof}

\begin{remark}\label{rmk:proofs:generalset:complexity}
	Carefully reviewing the estimates on the Gaussian width in the above proofs indicates that there is certain room for improvements, e.g., see the rough bound in \eqref{eq:proofs:generalset:global:coniceffdim}. In principle, we could replace the conic effective dimension in \eqref{eq:results:generalset:global:meas} and \eqref{eq:results:generalset:local:meas} by a non-standard version of the Gaussian width that is based on cylindrical instead of spherical localization.
	While this might lead to better (maybe even optimal) sampling rates, it is by far not clear how to compute these quantities for our examples of interest in Subsection~\ref{subsec:appl:signalsets}.
	Therefore, we have decided to state our main results by means of the convenient and well-known notion of conic effective dimension.
\end{remark}

\subsection{Proof of Proposition~\ref{prop:asympexprisk}}
\label{subsec:proofs:asympexprisk}

\begin{proof}[Proof of Proposition~\ref{prop:asympexprisk}]
	Let $\x^\ast$ be an expected risk minimizer on $\scalfac\sset$. Since $\grtrmu\in \scalfac\sset$, we have that
	\begin{equation}\label{eq:upperbound_exprisk}
		\lossexp(\x^\ast)\leq \lossexp(\grtrmu)=R(\scalfac),
	\end{equation}
	where $R(\scalfac) \coloneqq \mean[\losshng(\scalfac \cdot \abs{\gaussianuniv})]$ and $\gaussianuniv \coloneqq \sp{\a}{\grtr} \distributed \Normdistr{0}{1}$; note that this is actually the same function $R$ as in \eqref{eq:L} with $\fobs = \sign$. By elementary calculations, it is not hard to see that the mapping $\scalfac \mapsto R(\scalfac)$ is monotonically decreasing and satisfies
	\begin{equation}
		R(\scalfac) \leq  \sqrt{\frac{2}{\pi}} \cdot \frac{1}{\mu} \quad \text{for all $\scalfac > 0$.}
	\end{equation}
	Next, using the orthogonal decomposition $\x^\ast=\sp{\x^\ast}{\grtr}\grtr + \proj_{\orthcompl{\grtr}}(\x^\ast)$, we observe that
	\begin{align}
		\lossexp(\x^\ast)=\mean[\losshng(y \sp{\a}{\x^\ast})] &= \mean[\losshng(y \sp{\x^\ast}{\grtr}\sp{\a}{\grtr}+ y\sp{\a}{\proj_{\orthcompl{\grtr}}(\x^\ast)})]\\
		&=\mean{} [\losshng( \underbrace{\sp{\x^\ast}{\grtr}}_{\eqqcolon \alpha} \cdot \underbrace{\abs{\sp{\a}{\grtr}}}_{= \abs{\gaussianuniv}}+ \underbrace{\lnorm{\proj_{\orthcompl{\grtr}}(\x^\ast)}}_{\eqqcolon \beta} \cdot \underbrace{y\sp{\a}{\tfrac{\proj_{\orthcompl{\grtr}}(\x^\ast)}{\lnorm{\proj_{\orthcompl{\grtr}}(\x^\ast)}}}}_{\eqqcolon \orthcompl{\gaussianuniv}})] \\
		&= \mean{} [\losshng( \alpha \cdot \abs{\gaussianuniv}+ \beta \cdot \orthcompl{\gaussianuniv})],
	\end{align}
	where $\gaussianuniv, \orthcompl{\gaussianuniv} \distributed \Normdistr{0}{1}$ are independent. 
	Since $\losshng(\cdot)$ is non-negative and $\sign(\orthcompl{\gaussianuniv})$ is a symmetric random variable that is independent of $\abs{\orthcompl{\gaussianuniv}}$, it holds that
	\begin{equation}
		\mean{} [\losshng( \alpha \cdot \abs{\gaussianuniv}+ \beta \cdot \orthcompl{\gaussianuniv})]\geq
		\frac{1}{2} \cdot \mean[\losshng( \alpha \cdot \abs{\gaussianuniv} - \beta \cdot \abs{\orthcompl{\gaussianuniv}})],
	\end{equation}
	and altogether
	\begin{equation}\label{eq:proofs:asympexprisk:key_inequality}
		\mean[\losshng( \alpha \cdot \abs{\gaussianuniv} - \beta \cdot \abs{\orthcompl{\gaussianuniv}})]\leq 2 \cdot R(\scalfac) \lesssim \frac{1}{\mu} \ .
	\end{equation}
	Let us now assume that $\alpha < 1$. Then, the left-hand side of \eqref{eq:proofs:asympexprisk:key_inequality} is bounded from below by an absolute constant:
	\begin{align}
	\mean[\losshng( \alpha \cdot \abs{\gaussianuniv} - \beta \cdot \abs{\orthcompl{\gaussianuniv}})] \geq \mean[\losshng(  \abs{\gaussianuniv} - \beta \cdot \abs{\orthcompl{\gaussianuniv}})] \geq \mean[\losshng( \abs{\gaussianuniv}) ].
	\end{align}
	But this contradicts the assumption of $\scalfac \gtrsim 1$. Indeed, if $\scalfac \geq C$ for a sufficiently large numerical constant $C > 0$, the upper bound in \eqref{eq:proofs:asympexprisk:key_inequality} becomes smaller than $\mean[\losshng( \abs{\gaussianuniv}) ]$.
	Hence, we can assume that $\alpha \geq 1$ in the following. Our next step is to show that
	\begin{equation}
		\frac{\beta}{\alpha}\lesssim \mean[\losshng( \alpha \cdot \abs{\gaussianuniv} - \beta \cdot \abs{\orthcompl{\gaussianuniv}})].
	\end{equation}
	This lower bound follows from the independence of $\gaussianuniv$ and $\orthcompl{\gaussianuniv}$:
	\begin{align}
	\mean[\losshng( \alpha \cdot \abs{\gaussianuniv} - \beta \cdot \abs{\orthcompl{\gaussianuniv}})] &\asymp \int_0^\infty e^{-y^2/2} \int_0^\infty \pospart{1- \alpha x + \beta y} e^{-x^2/2} dx\, dy\\
	&=\frac{1}{\alpha} \int_0^\infty e^{-y^2/2}  \int_0^\infty \pospart{1- s + \beta y} e^{-s^2/2\alpha^2} ds\, dy\\
	&=\frac{1}{\alpha} \int_0^\infty e^{-y^2/2}\int_0^{1+\beta y} (1- s + \beta y) e^{-s^2/2\alpha^2} ds\, dy\\
	&\geq \frac{1}{\alpha} \int_0^\infty e^{-y^2/2}\int_0^{1} (1- s + \beta y) e^{-s^2/2\alpha^2} ds\, dy\\
	&\geq \frac{1}{\alpha} \int_0^\infty e^{-y^2/2}\int_0^{1} \beta y e^{-s^2/2} ds\, dy\asymp \frac{\beta}{\alpha} \ ,
	\end{align}
	where we have substituted $x=\frac{s}{\alpha}$ and used that $\alpha\geq 1$. 
	Combining this inequality with \eqref{eq:proofs:asympexprisk:key_inequality} it follows that
	\begin{equation}
		\frac{\lnorm{\proj_{\orthcompl{\grtr}}(\x^\ast)}}{\sp{\x^\ast}{\grtr}}=\frac{\beta}{\alpha}\lesssim \frac{1}{\mu} \ ,
	\end{equation}
	and therefore
	\begin{equation}
		\lnorm[auto]{\grtr - \frac{\x^\ast}{\lnorm{\x^\ast}}}\leq \lnorm[auto]{\grtr - \frac{\x^\ast}{\sp{\x^\ast}{\grtr}}} + \lnorm[auto]{\frac{\x^\ast}{\sp{\x^\ast}{\grtr}} - \frac{\x^\ast}{\lnorm{\x^\ast}}}\leq 2 \cdot \frac{\lnorm{\proj_{\orthcompl{\grtr}}(\x^\ast)}}{\sp{\x^\ast}{\grtr}}\lesssim \frac{1}{\scalfac} \ .
	\end{equation}
\end{proof}

\subsection{Proof of Proposition~\ref{prop:bitflipaddnoisecond}}
\label{subsec:proofs:noisemodel}

\begin{proof}[Proof of Proposition \ref{prop:bitflipaddnoisecond}]
Let us first consider the additive Gaussian noise model introduced in Example~\ref{ex:appl:noisepatterns}\ref{ex:addGauss}, i.e., $f(v) = f_{\sigma}(v)= \sign(v+ \tau)$ with $\tau\distributed \Normdistr{0}{\sigma^2}$.
Using the symmetry of $\tau$ and the independence from $g = \sp{\a}{\grtr} \distributed \Normdistr{0}{1}$, we observe that
\begin{align}
\prob[\sign(\gaussianuniv + \tau)=\sign(\gaussianuniv)\suchthat \abs{\gaussianuniv}]&=
\prob[\gaussianuniv\geq 0, \tau\geq -\gaussianuniv\suchthat \abs{\gaussianuniv}]+
\prob[\gaussianuniv\leq 0, \tau\leq -\gaussianuniv\suchthat \abs{\gaussianuniv}]\\
&=\prob[\gaussianuniv\geq 0, -\tau\geq -\gaussianuniv\suchthat \abs{\gaussianuniv}]+
\prob[\gaussianuniv\leq 0, -\tau\leq -\gaussianuniv\suchthat \abs{\gaussianuniv}]\\
&=\prob[\gaussianuniv\geq 0, \tau\leq \gaussianuniv\suchthat \abs{\gaussianuniv}]+
\prob[\gaussianuniv\leq 0, \tau\geq \gaussianuniv\suchthat \abs{\gaussianuniv}]\\
&\geq\prob[\gaussianuniv\geq 0, \tau\leq -\gaussianuniv\suchthat \abs{\gaussianuniv}]+
\prob[\gaussianuniv\leq 0, \tau\geq -\gaussianuniv\suchthat \abs{\gaussianuniv}]\\
&=\prob[\sign(\gaussianuniv + \tau)\neq\sign(\gaussianuniv)\suchthat \abs{\gaussianuniv}]\quad \text{(a.s.)}.
\end{align}
Consequently,
\begin{align}
0&\leq \prob[\sign(\gaussianuniv + \tau)=\sign(\gaussianuniv)\suchthat \abs{\gaussianuniv}]-
\prob[\sign(\gaussianuniv + \tau)\neq\sign(\gaussianuniv)\suchthat \abs{\gaussianuniv}]\\
&=\mean[\sign(\gaussianuniv + \tau)\sign(\gaussianuniv)\suchthat \abs{\gaussianuniv}]\\
&=\mean[f_{\sigma}(\gaussianuniv)\sign(\gaussianuniv)\suchthat \abs{\gaussianuniv}]\quad \text{(a.s.)},
\end{align}
which implies that $f_{\sigma}$ indeed satisfies \ref{enum:cond_c2} in Assumption~\ref{model:correlation}.

Next, we calculate the correlation parameter $\lambda_{f_{\sigma}}$.
By conditioning on $\gaussianuniv$ and using the symmetry of $\tau$ again, we observe that
\begin{align}
\mean_{\tau}[\indprob{\sign(\gaussianuniv)=\sign(\gaussianuniv+\tau)}] = \mean_{\tau}[\indprob{\tau\leq \abs{\gaussianuniv}}],
\end{align}
which leads to
\begin{align}
\mean_{\tau}[\indprob{\sign(\gaussianuniv)=\sign(\gaussianuniv+\tau)}
-\indprob{\sign(\gaussianuniv)\neq\sign(\gaussianuniv+\tau)}]&=\mean_{\tau}[\indprob{\tau\leq \abs{\gaussianuniv}}]-
\mean_{\tau}[\indprob{\tau\geq \abs{\gaussianuniv}}]\\
&=\mean_{\tau}[\indprob{\tau\leq \abs{\gaussianuniv}}]-
\mean_{\tau}[\indprob{\tau\leq -\abs{\gaussianuniv}}]\\
&=\mean_{\tau}[\indprob{\abs{\tau}\leq \abs{\gaussianuniv}}]\label{eq:condg}.
\end{align}
In this way, we can compute the correlation parameter:
\begin{align}
\lambda_{f_{\sigma}}=\mean[\sign(\gaussianuniv + \tau)\gaussianuniv]
	&=\mean[\abs{\gaussianuniv} \cdot (\indprob{\sign(\gaussianuniv)=\sign(\gaussianuniv+\tau)}
	-\indprob{\sign(\gaussianuniv)\neq\sign(\gaussianuniv+\tau)})]\\
	&=\mean_{\gaussianuniv}[\abs{\gaussianuniv}\cdot\mean_{\tau}[\indprob{\sign(\gaussianuniv)=\sign(\gaussianuniv+\tau)}
	-\indprob{\sign(\gaussianuniv)\neq\sign(\gaussianuniv+\tau)}]]\\
	&\stackrel{{\eqref{eq:condg}}}{=}\mean_{\gaussianuniv}[\abs{\gaussianuniv} \cdot \mean_{\tau}[\indprob{\abs{\tau}\leq \abs{\gaussianuniv}}]] \\
	&=\tfrac{2}{\sqrt{2\pi}} \integ[0][\infty]{x \cdot \Big(\tfrac{2}{\sqrt{2\pi \sigma^2}}\integ[0][x]{e^{-y^2/(2\sigma^2)}}{dy} \Big) \cdot e^{-x^2/2}}{dx} \\
	&=\tfrac{2}{\sqrt{2\pi}} \integ[0][\infty]{x \cdot \erf(\tfrac{x}{\sqrt{2}\sigma}) \cdot e^{-x^2/2}}{dx},
\end{align}
where $\erf(\cdot)$ denotes the error function. From the asymptotic equivalence
\begin{equation}
	\erf(\tfrac{x}{\sqrt{2}\sigma}) \asymp \min\{ \tfrac{x}{\sigma} , 1 \}, \quad x \geq 0,
\end{equation}
we can conclude that
\begin{align}
	\lambda_{f_{\sigma}} \asymp \integ[0][\infty]{x \cdot \min\{ \tfrac{x}{\sigma} , 1 \} \cdot e^{-x^2/2}}{dx}.
\end{align}
This immediately implies the desired upper bound:
\begin{align}
	\lambda_{f_{\sigma}} \lesssim \min\Big\{ \tfrac{1}{\sigma} \cdot \integ[0][\infty]{x^2 e^{-x^2/2}}{dx}, \integ[0][\infty]{x e^{-x^2/2}}{dx} \Big\} \lesssim \min\{ \tfrac{1}{\sigma}, 1\} \asymp \tfrac{1}{1 + \sigma}.
\end{align}
In order to see that this bound is tight, we finally make a case distinction in $\sigma$:
\begin{alignat}{2}
	\sigma \geq 1 &\quad \implies \quad \min\{ \tfrac{x}{\sigma} , 1 \} \geq \tfrac{1}{\sigma} \cdot \min\{ x , 1 \} &&\quad \implies \quad \lambda_{f_{\sigma}} \gtrsim \tfrac{1}{\sigma}, \\
	\sigma < 1 &\quad \implies \quad \min\{ \tfrac{x}{\sigma} , 1 \} \geq \min\{ x , 1 \} &&\quad \implies \quad \lambda_{f_{\sigma}} \gtrsim 1,
\end{alignat}
so that $\lambda_{f_{\sigma}} \gtrsim \min\{ \tfrac{1}{\sigma}, 1\} \asymp \tfrac{1}{1 + \sigma}$ for every $\sigma > 0$.
In particular, the correlation condition of \ref{enum:cond_c1} is fulfilled.

Let us now consider the bit flip model from Example~\ref{ex:appl:noisepatterns}\ref{ex:bitflip}, i.e., $f(v) = f_p(v)= \eps \cdot\sign(v)$ where $\eps\in \{-1,+1\}$ is a Bernoulli random variable with $\prob[\eps = 1] = p>\tfrac{1}{2}$. 
The condition \ref{enum:cond_c2} directly follows from the independence of $\eps$ and $\gaussianuniv$:
\begin{align}
\mean[f(\gaussianuniv)\sign(\gaussianuniv)\suchthat \abs{\gaussianuniv}]=\mean[\eps\suchthat \abs{\gaussianuniv}]=\mean[\eps]
=p-(1-p)=2p-1>0\quad  \text{(a.s.)}.
\end{align}
Moreover, we observe that
\begin{align}
\lambda_{f_p}=\mean[f_p(\gaussianuniv)\gaussianuniv]=\mean[\eps \cdot \abs{\gaussianuniv}]=\mean[\eps] \cdot \mean[\abs{\gaussianuniv}]=(2p-1)\sqrt{\tfrac{2}{\pi}}>0,
\end{align}
which shows that the condition \ref{enum:cond_c1} is satisfied as well.
\end{proof}



\section*{Acknowledgements}
The authors would like to thank Sjoerd Dirksen for initiating this project and for many fruitful discussions.
M.G. is supported by the European Commission Project DEDALE (contract no. 665044) within the H2020 Framework Program. A.S acknowledges funding by the Deutsche Forschungsgemeinschaft (DFG) within the priority program SPP 1798 Compressed Sensing in Information Processing through the project Quantized Compressive Spectrum Sensing (QuaCoSS). Finally, the authors would like to thank the anonymous referees
for their useful comments and suggestions which have helped to improve the original manuscript.

\renewcommand*{\bibfont}{\small}
\begin{refcontext}[sorting=nyt]
	\printbibliography
\end{refcontext}
\addcontentsline{toc}{section}{\refname}


\end{document}